\documentclass[10pt]{amsart}
\usepackage{amssymb,amsmath,amsfonts,amsthm,graphics,mathrsfs,amscd,hyperref}
\usepackage[hmargin=1in,vmargin=1in]{geometry}
\usepackage[all,cmtip]{xy}

\theoremstyle{plain}

\newtheorem{definition}[equation]{Definition}

\newtheorem{lemma}[equation]{Lemma}

\newtheorem{theorem}[equation]{Theorem}

\newtheorem{example}{Example}

\newtheorem{remark}[equation]{Remark}

\numberwithin{equation}{subsection}

\title{Obstructed and unobstructed Poisson deformations}
\author{Chunghoon Kim}

\email{ckim042@gmail.com}            

\begin{document}

\maketitle
\begin{abstract}
In this paper, we study obstructed and unobstructed (holomorphic) Poisson deformations with classical examples in deformation theory. 
\end{abstract}

\tableofcontents

\section{Introduction}

In this paper, we study obstructed and unobstructed (holomorphic) Poisson deformations with classical examples in deformation theory. A holomorphic Poisson manifold $X$ is a complex manifold such that its structure sheaf is a sheaf of Poisson algebras.\footnote{For general information on Poisson geometry, we refer to \cite{Lau13} .}  A holomorphic Poisson structure is encoded in a holomorphic section (a holomorphic bivector field) $\Lambda_0\in H^0(X,\wedge^2 \Theta_X)$ with $[\Lambda_0,\Lambda_0]=0$, where $\Theta_X$ is the sheaf of germs of holomorphic vector fields on $X$, and the bracket $[-,-]$ is the Schouten bracket on $X$. In the sequel a holomorphic Poisson manifold will be denoted by $(X,\Lambda_0)$. In \cite{Kim15}, we studied deformations of holomorphic Poisson structures on the basis of Kodaira-Spencer's deformation theory of complex structures. We defined a concept of a family of compact holomorphic Poisson manifolds, called a Poisson analytic family, which is based on a complex analytic family in the sense of Kodaira-Spencer's deformations theory (\cite{Kod05}), and proved theorem of existence and completeness for holomorphic Poisson structures as analogues of theorem of existence and completeness for deformations of complex structures (\cite{Kod05}). Throughout this paper, we will call deformations of complex structures `complex deformations', and deformations of holomorphic Poisson structures `Poisson deformations' for short. In this paper, we study Poisson deformations with classical examples in deformation theory.  We focus on obstructedness or unobstructedness of compact holomorphic Poisson manifolds in Poisson deformations. In particular, we provide examples which are
\begin{enumerate}
\item unobstructed in complex deformations, and unobstructed in Poisson deformations.
\item unobstructed in complex deformations, and obstructed in Poisson deformations.
\item obstructed in complex deformations, and unobstructed in Poisson deformations.
\item obstructed in complex deformations, and obstructed in Poisson deformations.
\end{enumerate}

In section \ref{section2}, we review deformations of holomorphic Poisson structures presented in \cite{Kim15}. We recall the definitions of a Poisson analytic family (see Definition \ref{o1}) and the associated Poisson Kodaira-Spencer map (see Definition \ref{o8}), and integrability condition (see Remark \ref{o9}). Given a compact holomorphic Poisson manifold $(X,\Lambda_0)$, infinitesimal (Poisson) deformations of $(X,\Lambda_0)$ are encoded in the first cohomology group $\mathbb{H}^1(X, \Theta_X^\bullet)$ of the complex of sheaves $\Theta_X^\bullet: \Theta_X \to \wedge^2 \Theta_X\to\wedge^3 \Theta_X \to \cdots $ induced by $[\Lambda_0,-]$. We say that $a\in \mathbb{H}^1(X, \Theta_X^\bullet)$ is called an obstructed element in Poisson deformations if there is no Poisson analytic family of deformations of $(X,\Lambda_0)$ such that $a$ is in the image of the Poisson Kodaira-Spencer map at the distinguished point (see Definition \ref{b11}). As $\theta\in H^1 (X, \Theta_X)$ is an obstructed element in complex deformations if $[\theta, \theta] \ne 0 $ in $H^2(X, \Theta_X)$, we similarly describe a condition when $a\in \mathbb{H}^1(X, \Theta_X^\bullet)$  is an obstructed element in Poisson deformations by using \v{C}ech reolution and Dolbeault resolution of $\Theta_X^\bullet$ (see Theorem \ref{o50} and Theorem \ref{o51}).

In section \ref{section3}, we study Poisson deformations of rational ruled surfaces $F_m=\mathbb{P}(\mathcal{O}_{\mathbb{P}_\mathbb{C}^1}(m)\oplus \mathcal{O}_{\mathbb{P}_\mathbb{C}^1}), m\geq 0$. It is known that since $H^2(F_m,\Theta_{F_m})=0$, $F_m$ are unobstructed in complex deformations. We determine obstructedness or unobstructedness in Poisson deformations for any holomorphic Poisson structure on $F_m$. We show that $F_0,F_1, F_2, F_3$ are unobstructed in Poisson deformations for any holomorphic Poisson structure. For $m\geq 4$, $(F_m, \Lambda_0)$ have both obstructed and unobstructed Poisson deformations depending on Poisson structure $\Lambda_0$. For unobstructed Poisson deformations, we will explicitly give examples of Poisson analytic families of deformations of $(F_m,\Lambda_0)$ such that the associated Poisson Kodaira-Spencer map is an isomorphism at the distinguished point.

In section \ref{section4}, we study Poisson deformations of (primary) Hopf surfaces. We determine obstructedness or unobstructedness of Poisson deformations except for two classes of Poisson Hopf surfaces (see Table $\ref{h51}$ and Remark $\ref{h95}$). By extending the method in \cite{Weh81} in the context of holomorphic Poisson deformations, in the case of unobstructed Poisson Hopf surfaces, we prove the unobstructedness by explicitly constructing Poisson analytic families such that the associated Poisson Kodaira-Spencer map is an isomorphism at the distinguished point.

In section \ref{section5}, we study Poisson deformations of $X=C_1\times C_2$ where $C_1$ and $C_2$ are nonsingular projective curves with genera $g(C_1)=g_1$ and $g(C_2)=g_2$ respectively. Since $X$ has only trivial Poisson structure for $g_1\geq 2$ or $g_2\geq 2$, we only consider $g_1\leq 1$ and $g_2\leq 1$. In this case, we show that $(X=C_1\times C_2,\Lambda_0)$ are unobstructed in Poisson deformations except for $(E\times \mathbb{P}_\mathbb{C}^1,\Lambda_0=0)$ where $E$ is an elliptic curve.

In section \ref{section6}, we study Poisson deformations of $T\times \mathbb{P}_\mathbb{C}^1$ where $T$ is a complex torus with dimension $2$. It is known that $X=T\times \mathbb{P}_\mathbb{C}^1$ where $T$ is a complex torus with dimension $2$, is obstructed in complex deformations (see \cite{Kod58} p.436). We determine obstructedness or unobstructedness in Poisson deformations for any holomorphic Poisson structure on $X$.  In particular, we show that there exist holomorphic Poisson structures $\Lambda_0$ on $X$ such that $(X,\Lambda_0)$ are unobstructed in Poisson deformations. $T\times \mathbb{P}_\mathbb{C}^1$ provides examples which are
\begin{enumerate}
\item obstructed in complex deformations, and unobstructed in Poisson deformations.
\item obstructed in complex deformations, and obstructed in Poisson deformations.
\end{enumerate}

\section{Obstructions}\label{section2}

\subsection{Review of deformations of holomorphic Poisson structures}\

We recall the definitions of a Poisson analytic family of deformations of compact holomorphic Poisson manifolds and the associated Poisson Kodaira-Spencer map (for the detail, see \cite{Kim15}).

\begin{definition}$($compare \cite{Kod05} p.59$)$ \label{o1}
Suppose that given a domain $B\subset \mathbb{C}^m$, there is a set $\{(M_t,\Lambda_t)|t \in B\}$ of $n$-dimensional compact holomorphic Poisson manifolds $(M_t,\Lambda_t)$, depending on $t=(t_1,...,t_m)\in B$. We say that $\{(M_t,\Lambda_t)|t\in B\}$ is a family of compact holomorphic Poisson manifolds or a Poisson analytic family of compact holomorphic Poisson manifolds if there exists a holomorphic Poisson manifold $(\mathcal{M},\Lambda)$ and a holomorphic map $\omega:\mathcal{M}\to B$ satisfing the following properties
\begin{enumerate}
\item $\omega^{-1}(t)$ is a compact holomorphic Poisson submanifold of $(\mathcal{M},\Lambda)$ for each $t\in B$.
\item $(M_t,\Lambda_t)=\omega^{-1}(t)(M_t$ has the induced Poisson holomorphic structure $\Lambda_t$ from $\Lambda)$.
\item The rank of Jacobian of $\omega$ is equal to $m$ at every point of $\mathcal{M}$.
\end{enumerate}
We will denote a Poisson analytic family by $(\mathcal{M},\Lambda,B,\omega)$. We also call $(\mathcal{M},\Lambda,B,\omega)$ a Poisson analytic family of deformations of a compact holomorphic Poisson manifold $(M_{t_0},\Lambda_{t_0})$ for each fixed point $t_0\in B$.
\end{definition}

\begin{remark}\label{o2}
Given a Poisson analytic family $(\mathcal{M},\Lambda,B,\omega)$ as in Definition $\ref{o1}$,  we can choose a locally finite open covering $\mathcal{U}=\{\mathcal{U}_j\}$ of $\mathcal{M}$ such that $\mathcal{U}_j$ are coordinate polydisks with a system of local complex coordinates $\{z_1,...,z_j,...\}$, where a local coordinate function $z_j:p\to z_j(p)$ on $\mathcal{U}_j$ satisfies $z_j(p)=(z_j^1(p),...,z_j^n(p),t_1,...,t_m)$, and $t=(t_1,...,t_m)=\omega(p)$. Then for a fixed $t_0\in B$, $\{p\mapsto (z_j^1(p),...,z_j^n(p))| \mathcal{U}_j \cap M_{t_0}\ne \emptyset\}$ gives a system of local complex coordinates on $M_{t_0}$. In terms of these coordinates, $\omega$ is the projection given by $(z_j,t)=(z_j^1,...,z_j^n,t_1,...,t_m)\to (t_1,...,t_m)$. For $j,k$ with $\mathcal{U}_j\cap \mathcal{U}_k\ne \emptyset$, we denote the coordinate transformations from $z_k$ to $z_j$ by $f_{jk}:(z_k^1,...,z_k^n,t)\to (z_j^1,...,z_j^n,t)=f_{jk}(z_k^1,...,z_k^n,t)$$($for the detail, see \cite{Kod05} p.60$)$. 
 
On the other hand, since $(M,\Lambda) \hookrightarrow (\mathcal{M},\Lambda)$ is a holomorphic Poisson submanifold for each $t\in B$ and $\mathcal{M}=\bigcup_t M_t$, the holomorphic Poisson structure $\Lambda$ on $\mathcal{M}$ can be expressed in terms of local coordinates as $\Lambda=\sum_{\alpha,\beta=1}^n g_{\alpha \beta}^j(z_j^1,...,z_j^n,t)\frac{\partial}{\partial{z_j^{\alpha}}}\wedge \frac{\partial}{\partial{z_j^{\beta}}}$ on $\mathcal{U}_j$, where $g_{\alpha\beta}^j(z_j,t)=g_{\alpha\beta}^j(z_j^1,...,z_n^n,t)$ is holomorphic with respect to $(z_j,t)$ with $g_{\alpha\beta}^j(z_j,t)=-g_{\beta\alpha}^j(z_j,t)$. For a fixed $t^0$, the holomorphic Poisson structure $\Lambda_{t^0}$ on $M_{t_0}$ is given by $\sum_{\alpha,\beta=1}^n g_{\alpha \beta}^j(z_j^1,...,z_j^n,t^0)\frac{\partial}{\partial{z_j^{\alpha}}}\wedge \frac{\partial}{\partial{z_j^{\beta}}}$ on $\mathcal{U}_j\cap M_{t_0}$.
\end{remark}

\begin{remark}
Let $(\mathcal{M},\Lambda,B,\omega)$ be a Poisson analytic family. Let $\Delta$ be an open set of $B$. Then the restriction $(\mathcal{M}_{\Delta}=\omega^{-1}(\Delta),\Lambda|_{M_{\Delta}},\Delta,\omega|_{\mathcal{M}_{\Delta}})$ is also a Poisson analytic family. We will denote the family by $(\mathcal{M}_{\Delta},\Lambda_{\Delta},\Delta,\omega)$.
\end{remark}

Given a Poisson analytic family $(\mathcal{M},\Lambda, B, \omega)$ of deformations of $(M, \Lambda_0)=\omega^{-1}(0), 0\in B$, the infinitesimal deformations of $(M, \Lambda_0)$ in the Poisson analytic family are encoded in $\mathbb{H}^1(M, \Theta_M^\bullet)$, which is the first cohomology group of the following complex of sheaves
\begin{align}\label{o5}
\Theta_{M}^\bullet : \Theta_{M}\xrightarrow{ [\Lambda_0,-]} \wedge^2 \Theta_{M}\xrightarrow{[\Lambda_0,-]}\cdots \xrightarrow{[\Lambda_0,-]} \wedge^n \Theta_{M}\to 0
\end{align}
where $\Theta_M$ is the sheaf of germs of holomorphic vector fields on $M$. We will denote the $i$-th cohomology group of the complex of sheave $(\ref{o5})$ by $\mathbb{H}^i(M, \Theta_{M}^\bullet)$. We can compute $\mathbb{H}^i(M,\Theta_M^\bullet)$ by using the double complex from \v{C}ech resolution or Dolbeault resolution of $\Theta_M^\bullet$ in the following: here $\mathcal{U}^0:=\mathcal{U}\cap M$ is an open covering of $M$, $\delta$ is the \v{C}ech map, and $A^{0,p}(M, \wedge^q \Theta_M)$ is the global section of sheaf of $C^\infty$ $p$-forms with coefficients in $\wedge^q \Theta_M$.
\begin{equation}\label{o6}
\begin{CD}
@A[\Lambda_0,-]AA\\
C^0(\mathcal{U}^0,\wedge^3 \Theta_M)@>-\delta>>\cdots\\
@A[\Lambda_0,-]AA @A[\Lambda_0,-]AA\\
C^0(\mathcal{U}^0,\wedge^2 \Theta_M)@>\delta>> C^1(\mathcal{U}^0,\wedge^2 \Theta_M)@>-\delta>>\cdots\\
@A[\Lambda_0,-]AA @A[\Lambda_0,-]AA @A[\Lambda_0,-]AA\\
C^0(\mathcal{U}^0,\Theta_M)@>-\delta>>C^1(\mathcal{U}^0,\Theta_M)@>\delta>>C^2(\mathcal{U}^0,\Theta_M)@>-\delta>>\cdots\\
\end{CD}
\end{equation}

\begin{equation}\label{o7}
\begin{CD}
@A[\Lambda_0,-]AA\\
A^{0,0}(M,\wedge^3 \Theta_M)@>\bar{\partial}>>\cdots\\
@A[\Lambda_0,-]AA @A[\Lambda_0,-]AA\\
A^{0,0}(M,\wedge^2 \Theta_M)@>\bar{\partial}>> A^{0,1}(M,\wedge^2 \Theta_M)@>\bar{\partial}>>\cdots\\
@A[\Lambda_0,-]AA @A[\Lambda_0,-]AA @A[\Lambda_0,-]AA\\
A^{0,0}(M,\Theta_M)@>\bar{\partial}>>A^{0,1}(M,\Theta_M)@>\bar{\partial}>>A^{0,2}(M, \Theta_M)@>\bar{\partial}>>\cdots\\
\end{CD}
\end{equation}

 Then we can define the Poisson Kodaira-Spencer map at $0\in B$ which encodes the information of infinitesimal Poisson deformations of $\omega^{-1}(0)=(M,\Lambda_0)$ in the Poisson analytic family $(X,\Lambda, B, \omega)$ (for the detail, see \cite{Kim15}).

\begin{definition}[Poisson Kodaira-Spencer map]\label{o8}
Let $(\mathcal{M},\Lambda,B,\omega)$ be a Poisson analytic family of deformations of $(M,\Lambda_0)=\omega^{-1}(0),0\in B$, where $B$ is a domain of $\mathbb{C}^m$. As in Remark $\ref{o2}$, let $\mathcal{U}=\{\mathcal{U}_j\}$ be an open covering of $\mathcal{M}$, and $(z_j,t)$ a local complex coordinate system on $\mathcal{U}_j$. The Poisson structure $\Lambda$ is expressed as $\sum_{\alpha,\beta=1}^n g_{\alpha \beta}^{j}(z_j,t)\frac{\partial}{\partial z_{j}^{\alpha}}\wedge \frac{\partial}{\partial z_{j}^{\beta}}$ on $\mathcal{U}_j$ where $g^{j}_{\alpha \beta}(z_j,t)$ is a holomorphic function with $g_{\alpha\beta}^j(z_j,t)=-g_{\beta\alpha}^j(z_j,t)$. For a tangent vector $\frac{\partial}{\partial t}=\sum_{\lambda=1}^{m} c_{\lambda}\frac{\partial}{\partial t_{\lambda}},c_{\lambda} \in \mathbb{C}$, of $B$, we put  
\begin{align*}
\frac{\partial (M_t, \Lambda_t)}{\partial t}|_{t=0}:=\left(  \{ \sum_{\alpha=1}^n \frac{\partial f_{jk}^\alpha(z_k,t)}{\partial t}|_{t=0}\frac{\partial}{\partial z_j^\alpha}     \} ,\{{\sum_{\alpha,\beta=1}^n\frac{\partial g_{\alpha \beta}^{j}(z_j,t)}{\partial t}|_{t=0} \frac{\partial}{\partial z_{j}^{\alpha}}\wedge \frac{\partial}{\partial z_{j}^{\beta}}}\}    \right) \in  C^1(\mathcal{U}^0, \Theta_{M}) \oplus C^0(\mathcal{U}^0, \wedge^2 \Theta_M)
\end{align*}
The Poisson Kodaira-Spencer map at $0\in B$ in the Poisson analytic family $(\mathcal{M},\Lambda, B, \omega)$ is defined to be a $\mathbb{C}$-linear map by using \v{C}ech resolution $(\ref{o6})$ of $\Theta_M^\bullet$.  
\begin{align*}
\varphi_0:T_0(B) &\to \mathbb{H}^1(M,\Theta_{M}^\bullet)\\
\frac{\partial}{\partial t} &\mapsto \frac{\partial (M_t,\Lambda_t)}{\partial t}|_{t=0}
\end{align*}
\end{definition}

\begin{remark}[see \cite{Kim15}]\label{o9}
Given a Poisson analytic family $(\mathcal{M},\Lambda , B, \omega)$ of deformations of $\omega^{-1}(0)=(M,\Lambda_0), 0\in B$, for a sufficiently small neighborhood $\Delta$ of $0\in B$, the restriction $(\mathcal{M}_\Delta, \Lambda_\Delta, \Delta, \omega)$ is represented by the convergent power series $\varphi(t)$ in $t\in \Delta$ with coefficients in $A^{0,1}(M,\Theta_M)$ and the convergent power series $\Lambda(t)$ in $t\in \Delta$ with coefficients in $A^{0,0}(M,\wedge^2 \Theta_M)$ such that $\varphi(0)=0, \Lambda(0)=0$, and 
\begin{align}\label{o4}
L (\varphi(t)+\Lambda(t))+\frac{1}{2}[\varphi(t)+\Lambda(t),\varphi(t)+\Lambda(t)]=0,\,\,\,\,\,\,\,\text{where}\,\,\,\,\,L=\bar{\partial}-+[\Lambda_0,-].
\end{align}
Conversely, let $(M,\Lambda_0)$ be a compact holomorphic Poisson manifold, and $\Delta$ be a neighborhood of $0\in \mathbb{C}^m$. Given a convergent power series $\varphi(t)$ in $t\in \Delta$ with coefficients in $A^{0,1}(M,\Theta_M)$ and a convergent power series $\Lambda(t)$ in $t\in \Delta$ with coefficients in $A^{0,0}(M,\wedge^2 \Theta_M)$ such that $\varphi(0)=0, \Lambda(0)=0$ and satisfy $(\ref{o4})$. $(\varphi(t),\Lambda(t))$ defines a Poisson analytic family $(\mathcal{M},\Lambda, \Delta, \omega)$ of deformations of $(M,\Lambda_0)=\omega^{-1}(0)$. In this case, the associated Poisson Kodaira-Spencer map at $t=0$ is described in the following way by using Daulbeault resolution $(\ref{o7})$ of $\Theta_M^\bullet$. 
\begin{align}
\varphi_0: T_0 \Delta &\to \mathbb{H}^1(M,\Theta_M^\bullet)\\
 \frac{\partial}{\partial t}&\mapsto \left(  \frac{\partial \varphi(t)}{\partial t}|_{t=0}, \frac{\partial \Lambda(t)}{\partial t}|_{t=0}   \right) \notag
\end{align}
\end{remark}

\subsection{Description of obstructed elements in \v{C}ech Resolution of $\Theta_M^\bullet$}\

Based on the previous section, now we consider our main topic in this paper: obstructedness or unobstructedness of compact holomorphic Poisson manifolds in Poisson deformations.

\begin{definition}
We say that a compact holomorphic Poisson manifold $(M,\Lambda_0)$ is unobstructed if there is a Poisson analytic family $(\mathcal{M},\Lambda, B, \omega)$ of deformations of $\omega^{-1}(0)=(M,\Lambda_0), 0\in B$ such that the Poisson Kodaira-Spencer map $\varphi_0: T_0 B \to \mathbb{H}^1(M,\Theta_M^\bullet)$ is an isomorphism at $0\in B$. Otherwise, we say that $(M,\Lambda_0)$ is obstructed in Poisson deformations.

\end{definition}

\begin{remark}
We note that by theorem of existence for deformations of holomorphic Poisson structures $($see \cite{Kim15}$)$, if a compact holomorphic Poisson manifold $(M,\Lambda_0)$ satisfies $\mathbb{H}^2(M,\Theta_M^\bullet)=0$, there is a Poisson analytic family $(\mathcal{M},\Lambda, B, \omega)$ such that $\omega^{-1}(0)=(M,\Lambda_0) ,0\in B$ and the Poisson Kodaira-Spence map is an isomorphism at $0\in B$ so that $(M,\Lambda_0)$ is unobstructed in Poisson deformaitons.
\end{remark}

\begin{example}\label{o19}
Let $(M,\Lambda_0)$ be any Poisson Del Pezzo surface. Then $\mathbb{H}^2(M,\Theta_M^\bullet)=0$ $($see \cite{Pin11}$)$ so that  $(M,\Lambda_0)$ is unobstructed in Poisson deformations.
\end{example}

\begin{example}
Let $(M,\Lambda_0)$ be any Poisson $K3$ surface. Since $H^1(M,\wedge^2 \Theta_X)=H^2(M,\Theta_M)=0$, we have $\mathbb{H}^2(M,\Theta_M^\bullet)=0$ so that $(M,\Lambda_0)$ is unobstructed in Poisson deformations.
\end{example}

\begin{example}[compare \cite{Kod05} p.230-232]\label{torus}
Let $(z^1,...,z^n)$ be a complex coordinate of $\mathbb{C}^n$.  Let $(M,\Lambda_0)$ be a Poisson complex torus of dimension $n$ defined by the period matrix of the form
\begin{equation*}
\left(\begin{matrix}
a_1^1 &\cdots & a_n^1\\
\cdots & \cdots &\cdots \\
a_n^1 &\cdots & a_n^n\\
1 &\cdots  &0\\
\cdots & \cdots &\cdots\\
0 &\cdots &1
\end{matrix}\right)
\end{equation*}
and Poisson structure $\Lambda_0=\sum_{1\leq \alpha<\beta \leq n} b_{\alpha\beta}\frac{\partial}{\partial z^\alpha}\wedge \frac{\partial}{\partial z^\beta},b_{\alpha\beta}\in \mathbb{C}$ on $M$. Let us consider the spectral sequence associated with the double complex $(\ref{o7})$. At $E_1$, we have
\begin{equation}
\begin{CD}
H^0(M,\wedge^3 \Theta_M) @. H^1(M,\wedge^3\Theta_M)@. H^2(M,\wedge^3 \Theta_M)\\
@A[\Lambda_0,-]AA @A[\Lambda_0,-]AA @A[\Lambda_0,-]AA\\
H^0(M, \wedge^2 \Theta_M) @. H^1(M, \wedge^2 \Theta_M) @. H^2(M,\wedge^2 \Theta_M)\\
@A[\Lambda_0,-]AA @A[\Lambda_0,-]AA @A[\Lambda_0,-]AA\\
H^0(M,\Theta_M) @.H^1(M,\Theta_M) @. H^2(M,\Theta_M)
\end{CD}
\end{equation}
Since $[\Lambda_0,-]$ is the zero map, we have
\begin{align*}
\mathbb{H}^1(M,\Theta_M^\bullet)\cong H^0(M,\wedge^2 \Theta_M)\oplus H^1(M,\Theta_M)
\end{align*}
so that $\dim_\mathbb{C} \mathbb{H}^1(M,\Theta_M^\bullet)=n^2+\binom{n}{2}$. We show that $(M,\Lambda_0)$ is unobstructed by constructing a complete Poisson analytic family of deformations of $(M, \Lambda_0)$ such that the Poisson Kodaira-Spencer map is an isomorphism.

Let $(M,\Lambda)$ be a Poisson complex torus of dimension $n$ with the period matrix of the from
\begin{equation*}
\left(\begin{matrix}
t_1^1 &\cdots & t_n^1\\
\cdots & \cdots &\cdots \\
t_n^1 &\cdots & t_n^n\\
1 &\cdots  &0\\
\cdots & \cdots &\cdots\\
0 &\cdots &1
\end{matrix}\right)
\end{equation*}
with Poisson structure $\Lambda=\sum_{1\leq \alpha < \beta <n} v_{\alpha\beta} \frac{\partial}{\partial z^\alpha}\wedge \frac{\partial}{\partial z^\beta}$.

Put $t=(t_\alpha^\beta)_{\alpha,\beta=1,...,n}$, $v=(v_{\alpha\beta})_{1\leq \alpha<\beta \leq n}$, and write $(M_t,\Lambda_v)$ for $(M,\Lambda)$. Let
\begin{align*}
\omega_j(t)=(\omega_j^1(t),\cdots, \omega_j^n(t))
\end{align*}
be the $k$-th row of the period matrix. Namely
\begin{equation}
\omega_j^\beta(t)=\begin{cases}
t_j^\beta, & j=1,...,n\\
\delta_{j-n}^\beta & j=n+1,...,2n
\end{cases}
\end{equation}

Then $M_t=\mathbb{C}^n/G_t$ where $G_t=\{\sum_{j=1}^{2n} m_j \omega_j(t)|m_j\in \mathbb{Z}\}$. Let $B=\{(t,v)\in \mathbb{C}^{n^2}\times \mathbb{C}^{\binom{n}{2}}|\det\, \textnormal{Im} \, t>0\}$ where $\textnormal{Im}\, t=(\textnormal{Im}\, t_\alpha^\beta)_{\alpha,\beta=1,...,n}$. Then $\{M_t|t\in B\}$ forms a complex analytic family $(\mathcal{M}, B ,\pi)$, where $\mathcal{M}=\mathbb{C}^n\times B/\mathcal{G}$ is the quotient space by $\mathcal{G}$ and the projection $\pi:(z,t,v)\mapsto (t,v)$. Here $\mathcal{G}$ is the group of automorphism of $\mathbb{C}^n\times B$ consisting of all automorphisms defined by
\begin{align*}
(z,t)\to (z+\sum_{j=1}^{2n} m_j\omega_j(t),t,v),\,\,\,\,\,m_j\in \mathbb{Z},\,\,\,\,\,j=1,...,2n.
\end{align*}
Moreover $\Lambda= \sum_{1\leq \alpha<\beta\leq n}^n v_{\alpha\beta}\frac{\partial}{\partial z^\alpha}\wedge \frac{\partial}{\partial z^\beta}$ is invariant under the action $\mathcal{G}$ so that $(\mathcal{M},\Lambda,B, \pi)$ is a Poisson analytic family and $\pi^{-1}(t,v)=(\mathbb{C}^n\times t/G_t, \Lambda_{v})=(M_t,\Lambda_v)$. Then the Poisson Kodaira-Spencer map at $(t,v)$ is an isomorphism
\begin{align*}
\varphi_{(t,v)}:T_{(t,v)} B&\to \mathbb{H}^1(M_t, \Theta_{M_t}^\bullet)\cong H^1(M_t, \Theta_{M_t})\oplus H^0(M_t, \wedge^2 \Theta_{M_t})\\
        \left(\frac{\partial}{\partial t_{\alpha}^\beta}, \frac{\partial}{\partial v_{\alpha\beta} }\right) &\mapsto  \left( \frac{\partial M_t}{\partial t_{\alpha}^\beta},\frac{\partial \Lambda_v}{\partial v_{\alpha\beta} } \right)=\left( \sum_{\gamma=1}^n u_\gamma^\alpha d\bar{z}^\gamma\frac{\partial}{\partial z^\beta},  \frac{\partial}{\partial z^\alpha}\wedge \frac{\partial}{\partial z^\beta}  \right)
\end{align*}
$($for the definition of $u_\gamma^\alpha$, and the description of $\frac{\partial M_t}{\partial t_\alpha^\beta}$, see \cite{Kod05} $p.231-233$.$)$
\end{example}

As we find an obstructed element $\theta\in H^1(M, \Theta_M)$ (see \cite{Kod05} p.209-214) in complex deformations in order to show that a given compact complex manifold $M$ is obstructed, analogously an one way to determine obstructedeness of a compact holomorphic Poisson manifold in Poisson deformations is to find an obstructed element in Poisson deformations as we define in the following way.

\begin{definition}\label{b11}
Let $(M,\Lambda_0)$ be a compact holomorphic Poisson manifold. We say that $a\in \mathbb{H}^1(M,\Theta_M^\bullet)$ is an obstructed element in Poisson deformations of $(M,\Lambda_0)$ if there is no Poisson analytic family $(\mathcal{M},\Lambda,B, \omega)$ of deformations of $(M,\Lambda_0)=\omega^{-1}(0),0\in B$ such that $a\in \varphi_0(T_0(M))$ where $\varphi_0:T_0 M\to \mathbb{H}^1(M,\Theta_M^\bullet)$ is the Poisson Kodaira-Spencer map at $0\in B$ of $(\mathcal{M},\Lambda, B, \pi)$.
\end{definition}

Hence if there is an obstructed element $a\in \mathbb{H}^1(M,\Theta_M^\bullet)$ in Poisson deformations, $(M,\Lambda_0)$ is obstructed. As $\theta\in H^1(M,\Theta_M)$ is an obstructed element in complex deformation if $[\theta,\theta]\ne 0 \in H^2(X,\Theta_X)$, we now describe a condition when $a\in \mathbb{H}^1(M, \Theta_M)$ is an obstructed element in Poisson deformations. We extend the arguments in \cite{Kod05} p.210-214 in the context of a Poisson analytic family. 

Let $(\mathcal{M},\Lambda, B, \omega)$ with $0\in B\subset \mathbb{C}$ be a Poisson analytic family such that $\omega^{-1}(0)=(M,\Lambda_0)$. Take a small disk $\Delta$ with centre $0$ such that $0\in \Delta \subset B$, and represent $(\mathcal{M}_\Delta,\Lambda_\Delta)=\omega^{-1}(\Delta)$ in the form
\begin{align*}
(\mathcal{M}_\Delta, \Lambda_\Delta)=\bigcup_{j=1}^l (U_j\times \Delta, \Lambda_j(t))
\end{align*}
where each $U_j$ is a Poisson polydisk equipped with the Poisson structure $\Lambda_j(t)=\sum_{\alpha,\beta=1}^n g_{\alpha\beta}^j (z_j,t)\frac{\partial}{\partial z_j^\alpha}\wedge \frac{\partial}{\partial z_j^\beta}$ with $g_{\beta\alpha}^j(z_j,t)=-g_{\alpha\beta}^j(z_j,t)$ and $[\Lambda_j(t), \Lambda_j(t)]=0$, and $(z_j,t)\in U_j\times \Delta$ and $(z_k,t)\in U_k\times \Delta$ are the same point on $\mathcal{M}_\Delta$ if $z_j^\alpha=f_{jk}^\alpha(z_k,t)$ for $\alpha=1,...,n$. Here each $f_{jk}^\alpha(z_k,t)=f_{jk}^\alpha(z_k^1,...,z_k^n,t)$ is a holomorphic Poisson map of $z_k^1,...,z_k^n,t$ defined on $(U_k\times \Delta,\Lambda_k(t)) \cap  ( U_j\times \Delta ,\Lambda_j (t) ) \ne \emptyset$.

As in Definition \ref{o8}, the infinitesimal deformation $(\theta(t),\lambda(t))\in \frac{\partial (M_t,\Lambda_t)}{\partial t}\in \mathbb{H}^1(M_t, \Theta_t^\bullet)$ is the cohomology class of the $1$-cocycle $(\{\lambda_j(t)\},\{\theta_{jk}(t)\} )\in C^0(\mathcal{U}_t, \wedge^2 \Theta_t)\oplus C^1(\mathcal{U}_t,\Theta_t)$ where $\mathcal{U}_t=\{U_j\times t\}$, the vector field 
\begin{align*}
\theta_{jk}(t)=\sum_{\alpha=1}^n \theta_{jk}^\alpha(z_k,t)\frac{\partial}{\partial z_j^\alpha},\,\,\,\,\,\,\,\,\,\text{where}\,\,\,\,\,\theta_{jk}^\alpha(z_k,t)=\frac{\partial f_{jk}^\alpha(z_k,t)}{\partial t},\,\,\,\,\,z_k=f_{kj}(z_j,t),
\end{align*}
and the bivector field
\begin{align*}
\lambda_j(t)=\sum_{\alpha,\beta=1}^n \lambda_{\alpha\beta}^j(z_j,t)\frac{\partial}{\partial z_j^\alpha}\wedge \frac{\partial}{\partial z_j^\beta}, \,\,\,\,\,\,\,\,\,\text{where}\,\,\,\,\,\lambda_{\alpha\beta}^j(z_j,t)=\frac{\partial g_{\alpha\beta}^j(z_j,t)}{\partial t}
\end{align*}

On $(U_k\times \Delta) \cap (U_j\times \Delta)\ne \emptyset$, we have the equalities
\begin{align} 
f_{ik}^\alpha(z_k,t)&=f_{ij}^\alpha(f_{jk}(z_k,t),t),\,\,\,\,\,\alpha=1,...,n. \label{o15}\\
g_{\alpha\beta}^j(f_{jk}(z_k,t),t)&=\sum_{p,q=1}^n g_{pq}^k(z_k,t)\frac{\partial f_{jk}^\alpha(z_k,t)}{\partial z_k^p}\frac{f_{jk}^\beta(z_k,t)}{\partial z_k^q},\,\,\,\,\,\,\alpha,\beta=1,...,n \label{o16}
\end{align}
and on $U_j\times \Delta$, we have the equalities
\begin{align}
[\sum_{\alpha,\beta=1}^n g_{\alpha\beta}^j(z_j,t)\frac{\partial}{\partial z_j^\alpha}&\wedge \frac{\partial}{\partial z_j^\beta},  \sum_{\alpha,\beta=1}^n g_{\alpha\beta}^j(z_j,t)\frac{\partial}{\partial z_j^\alpha}\wedge \frac{\partial}{\partial z_j^\beta}]=0, \,\,\,\,\,\alpha,\beta=1,...,n. \label{o17}
\end{align}

By differentiating both sides of $(\ref{o15}),(\ref{o16})$ and $(\ref{o17})$  in $t$, we have
\begin{align}
&\theta_{ik}^\alpha(z_i,t)=\theta_{ij}^\alpha(z_i,t)+\sum_{\beta=1}^n \frac{\partial z_i^\alpha}{\partial z_j^\beta}\theta_{jk}^\beta(z_j,t),\,\,\,\,\alpha=1,...,n. \label{b1}\\
&\lambda_{\alpha\beta}^j(z_j,t)+\sum_{r=1}^n \frac{\partial g_{\alpha\beta}^j}{\partial z_j^r}\theta_{jk}^r(z_j,t),\,\,\,\,\,\alpha,\beta=1,...,n. \label{b2}\\
&=\sum_{p,q=1}^n\lambda_{pq}^k(z_k,t)\frac{\partial f_{jk}^\alpha}{\partial z_k^p}\frac{\partial f_{jk}^\beta}{\partial z_k^q}+\sum_{p,q=1}^n g_{pq}^k(z_k,t)\frac{\partial \theta_{jk}^\alpha(z_j,t)}{\partial z_k^p}\frac{\partial f_{jk}^\beta}{\partial z_k^q}+\sum_{p,q=1}^n g_{pq}^k(z_k,t)\frac{\partial f_{jk}^\alpha}{\partial z_k^p} \frac{\partial \theta_{jk}^\beta(z_j,t)}{\partial z_k^q} \notag \\
&=\sum_{p,q=1}^n\lambda_{pq}^k(z_k,t)\frac{\partial f_{jk}^\alpha}{\partial z_k^p}\frac{\partial f_{jk}^\beta}{\partial z_k^q}+\sum_{p,q=1}^n g_{pq}^j(z_j,t)\frac{\partial \theta_{jk}^\alpha(z_j,t)}{\partial z_j^p}\frac{\partial f_{jk}^\beta}{\partial z_j^q}+\sum_{p,q=1}^n g_{pq}^j(z_j,t)\frac{\partial f_{jk}^\alpha}{\partial z_j^p} \frac{\partial \theta_{jk}^\beta(z_j,t)}{\partial z_j^q} \notag \\
& [\sum_{\alpha,\beta=1}^n g_{\alpha\beta}^j(z_j,t)\frac{\partial}{\partial z_j^\alpha}\wedge \frac{\partial}{\partial z_j^\beta},\sum_{\alpha,\beta=1}^n \lambda_{\alpha\beta}^j(z_j,t)\frac{\partial}{\partial z_j^\alpha}\wedge \frac{\partial}{\partial z_j^\beta}]=0. \label{b3}
\end{align}

Next we will take the derivative of $(\ref{b1}),(\ref{b2}), (\ref{b3})$ in $t$ as functions of $z_j^1,...,z_j^n,t$ again. As in \cite{Kod05} p.211 we sometimes write $\left(\frac{\partial}{\partial t}\right)_j$ instead of $\left(\frac{\partial}{\partial t}\right)$ in order to make explicit that $\frac{\partial}{\partial t}$ denotes the differentiation of a function of $z_j^1,...,z_j^n,t$ with respect to $t$. On $(U_j\times \Delta ) \cap (  U_k\times \Delta ) \ne  \emptyset$, we have the following equalities: $\frac{\partial}{\partial z_k^\alpha}=\sum_{\beta=1}^n \frac{\partial z_j^\beta}{\partial z_k^\alpha}\frac{\partial}{\partial z_j^\beta}$, where $\frac{\partial z_j^\beta}{\partial z_k^\alpha}=\frac{\partial f_{jk}^\beta(z_k,t)}{\partial z_k^\alpha}$, and
\begin{align}
\left(\frac{\partial}{\partial t}\right)_k&=\sum_{s=1}^n \frac{\partial f_{jk}^s(z_k,t)}{\partial t}\frac{\partial}{\partial z_j^s}+\left( \frac{\partial}{\partial t} \right)_j=\sum_{s=1}^n \theta_{jk}^s(z_j,t)\frac{\partial}{\partial z_j^s} +\left( \frac{\partial}{\partial t} \right)_j \notag \\
&\iff \left(\frac{\partial}{\partial t}\right)_j=-\sum_{s=1}^n \theta_{jk}^s(z_j,t)\frac{\partial}{\partial z_j^s}+\left(\frac{\partial}{\partial t}\right)_k  \label{o20}
\end{align}

By taking the derivative $(\ref{b1})$ with respect to $t$, and putting
\begin{align*}
\dot{\theta}_{ij}(t)=\sum_{\alpha=1}^n \dot{\theta}_{ij}^\alpha(z_i,t) \frac{\partial}{\partial z_i^\alpha},\,\,\,\,\,\text{where}\,\,\,\,\,\dot{\theta}_{ij}^\alpha(z_i,t)=\frac{\partial \theta_{ij}^\alpha(z_i,t)}{\partial t}
\end{align*}
 we obtain (for the detail, see \cite{Kod05} p.211-213)
 \begin{align}\label{o35}
 \dot{\theta}_{ij}(t)-\dot{\theta}_{ik}(t)+\dot{\theta}_{jk}(t)=[\theta_{ij}(t),\theta_{jk}(t)]
 \end{align}
 
 By taking the derivative $(\ref{b2})$ with respect to $t$, and putting
 \begin{align*}
 \dot{\lambda}_j(t)=\sum_{\alpha,\beta=1}^n \dot{\lambda}_{\alpha\beta}^j(z_j,t)\frac{\partial}{\partial z_j^\alpha}\wedge \frac{\partial}{\partial z_j^\beta},\,\,\,\,\,\,\text{where}\,\,\,\,\, \dot{\lambda}_{\alpha\beta}^j(z_j,t)=\frac{\partial {\lambda}_{\alpha\beta}^j (z_j,t)  }{\partial t} ,
 \end{align*}
we obtain
\begin{align*}
&[\sum_{\alpha,\beta=1}^n g_{\alpha\beta}^j(z_j,t)\frac{\partial}{\partial z_j^\alpha}\wedge \frac{\partial}{\partial z_j^\beta},\sum_{\alpha,\beta=1}^n \dot{\lambda}_{\alpha\beta}^j(z_j,t)\frac{\partial}{\partial z_j^\alpha}\wedge \frac{\partial}{\partial z_j^\beta}]\\
&+ [\sum_{\alpha,\beta=1}^n \lambda_{\alpha\beta}^j(z_j,t)\frac{\partial}{\partial z_j^\alpha}\wedge \frac{\partial}{\partial z_j^\beta},\sum_{\alpha,\beta=1}^n \lambda_{\alpha\beta}^j(z_j,t)\frac{\partial}{\partial z_j^\alpha}\wedge \frac{\partial}{\partial z_j^\beta}]=0
\end{align*}
which is equivalent to
\begin{align}\label{o36}
[\Lambda(t),\dot{\lambda}_j(t)]=-[\lambda_j(t),\lambda_j(t)]
\end{align}

Lastly we take the derivative of $(\ref{b2})$ with respect to $t$ (i.e.$\left(\frac{\partial}{\partial t}\right)_j$). Let us consider the left hand side of $(\ref{b2})$.
\begin{align}\label{o21}
\left(\frac{\partial}{\partial t}\right)_j \left( \lambda_{\alpha\beta}^j(z_j,t)     +\sum_{r=1}^n \frac{\partial g_{\alpha\beta}^j}{\partial z_j^r}\theta_{jk}^r(z_j,t)   \right)=\dot{\lambda}_{\alpha\beta}^j(z_j,t)+\sum_{r=1}^n \frac{\partial \lambda_{\alpha\beta}^j}{\partial z_j^r}\theta_{jk}^r(z_j,t)+\sum_{r=1}^n\frac{\partial g_{\alpha\beta}^j}{\partial z_j^r}\dot{\theta}_{jk}^r(z_j,t)
\end{align} 

Let us consider the right hand side of $(\ref{b2})$. We take the derivative of each term with respect to $t$ in the following. We take the derivative of the first term of the right hand side of $(\ref{b2})$ with respect to $t$. From $(\ref{o20})$, we have

\begin{align}\label{o22}
&\left(\frac{\partial}{\partial t}\right)_j \sum_{p,q=1}^n \lambda_{pq}^k(z_k,t)\frac{\partial f_{jk}^\alpha}{\partial z_k^p}\frac{\partial f_{jk}^\beta}{\partial z_k^q} \\
&=-\sum_{p,q=1}^n\sum_{s=1}^n \theta_{jk}^s(z_j,t)\frac{\partial }{\partial z_j^s}\left(\lambda_{pq}^k(z_k,t)\frac{\partial f_{jk}^\alpha}{\partial z_k^p}\frac{\partial f_{jk}^\beta}{\partial z_k^q}\right) \notag\\
&+\sum_{p,q=1}^n\dot{ \lambda}_{pq}^k(z_k,t)\frac{\partial f_{jk}^\alpha}{\partial z_k^p}\frac{\partial f_{jk}^\beta}{\partial z_k^q}
+\sum_{p,q=1}^n\lambda_{pq}^k(z_k,t)\frac{\partial \theta_{jk}^\alpha}{\partial z_k^p}\frac{\partial f_{jk}^\beta}{\partial z_k^q} 
+\sum_{p,q=1}^n\lambda_{pq}^k(z_k,t)\frac{\partial f_{jk}^\alpha}{\partial z_k^p}\frac{\partial \theta_{jk}^\beta}{\partial z_k^q} \notag
\end{align}

We take the derivative of the second term of the right hand side of $(\ref{b2})$ with respect to $t$.

\begin{align}
\left(\frac{\partial}{\partial t}\right)_j \sum_{p,q=1}^n g_{pq}^j(z_j,t)\frac{\partial \theta_{jk}^\alpha(z_j,t)}{\partial z_j^p}\frac{\partial f_{jk}^\beta}{\partial z_j^q}=\left(\frac{\partial}{\partial t}\right)_j \sum_{p=1}^n g_{p\beta}^j(z_j,t)\frac{\partial \theta_{jk}^\alpha(z_j,t)}{\partial z_j^p}=\sum_{p=1}^n \lambda_{p\beta}^j(z_j,t)\frac{\partial \theta_{jk}^\alpha}{\partial z_j^p}+\sum_{p=1}^n g_{p\beta}^j\frac{\partial \dot{\theta}_{jk}^\beta}{\partial z_j^p} \label{o23}
\end{align}

We take the derivative of the third term of the right hand side of $(\ref{b2})$ with respect to $t$.

\begin{align}\label{o24}
&\left(\frac{\partial}{\partial t}\right)_j \sum_{p,q=1}^n g_{pq}^j(z_j,t)\frac{\partial f_{jk}^\alpha}{\partial z_j^p}\frac{\partial \theta_{jk}^\beta(z_j,t)}{\partial z_j^q}=\left(\frac{\partial}{\partial t}\right)_j \sum_{q=1}^n g_{\alpha q}^j(z_j,t)\frac{\partial \theta_{jk}^\beta}{\partial z_j^q}= \sum_{q=1}^n \lambda_{\alpha q}^j(z_j,t)\frac{\partial \theta_{jk}^\beta  }{\partial z_j^q}+\sum_{q=1}^n g_{\alpha q}^j\frac{\partial \dot{\theta}_{jk}^\beta}{\partial z_j^q}
\end{align}

Then from $(\ref{b2}),(\ref{o21}),(\ref{o22}),(\ref{o23})$ and $(\ref{o24})$, we have
\begin{align}\label{o251}
&\sum_{\alpha,\beta=1}^n \dot{\lambda}^j_{\alpha\beta}\frac{\partial}{\partial z_j^\alpha}\wedge \frac{\partial}{\partial z_j^\beta}+\sum_{\alpha,\beta,r =1}^n  \frac{\partial \lambda_{\alpha\beta}^j}{\partial z_j^r}\theta_{jk}^r \frac{\partial}{\partial z_j^\alpha}\wedge \frac{\partial}{\partial z_j^\beta}+\sum_{\alpha,\beta, r=1}^n\frac{\partial g_{\alpha\beta}^j}{\partial z_j^r}\dot{\theta}_{jk}^r\frac{\partial}{\partial z_j^\alpha}\wedge \frac{\partial}{\partial z_j^\beta}\\
=&-\sum_{\alpha,\beta, p,q,s=1}^n\frac{\partial}{\partial z_j^s}\left(  \lambda_{pq}^k(z_k,t)\frac{\partial f_{jk}^\alpha}{\partial z_k^p}\frac{\partial f_{jk}^\beta}{\partial z_k^q}   \right)\theta_{jk}^s\frac{\partial }{\partial z_j^\alpha}\wedge \frac{\partial}{\partial z_j^\beta}+\sum_{p,q=1}^n \dot{\lambda}_{pq}^k \frac{\partial}{\partial z_k^p}\wedge \frac{\partial}{\partial z_k^q}
+2\sum_{p,q=1}^n \lambda_{pq}^k(z_k,t)\frac{\partial \theta_{jk}^\alpha}{\partial z_k^p}\frac{\partial f_{jk}^\beta}{\partial z_k^q}\frac{\partial}{\partial z_j^\alpha}\wedge \frac{\partial}{\partial z_j^\beta} \notag \\
&+2\sum_{\alpha,\beta,p=1}^n \lambda_{p\beta}^j\frac{\partial \theta_{jk}^\alpha}{\partial z_j^p}\frac{\partial}{\partial z_j^\alpha}\wedge \frac{\partial}{\partial z_j^\beta}+2\sum_{\alpha,\beta,p=1}^n g_{p\beta}^j\frac{\partial \dot{\theta}_{jk}^\alpha}{\partial z_j^p} \frac{\partial}{\partial z_j^\alpha}\wedge \frac{\partial}{\partial z_j^\beta}\notag
\end{align}
\begin{lemma}
$(\ref{o25})$ is equivalent to
\begin{align}
\dot{\lambda}_j(t)-\dot{\lambda}_k(t)= [\Lambda(t), \dot{\theta}_{jk}(t)] +[\lambda_j(t),\theta_{jk}(t)]+[\lambda_k(t),\theta_{jk}(t)] \label{o261}\\
\iff \dot{\lambda}_k(t)-\dot{\lambda}_j(t)+[\Lambda(t),\dot{\theta}_{jk}(t)]=-[\lambda_j(t)+\lambda_k(t),\theta_{jk}(t)] \label{o31}
\end{align}
\end{lemma}
\begin{proof}
We compute each term of $(\ref{o261})$.
Let us compute
\begin{align} \label{o27}
\dot{\lambda}_j(t)-\dot{\lambda}_k(t)=\sum_{\alpha,\beta=1}^n \dot{\lambda}_{\alpha\beta}^j\frac{\partial}{\partial z_j^\alpha}\wedge \frac{\partial}{\partial z_j^\beta}-\sum_{p,q=1}^n\dot{\lambda}_{pq}^k\frac{\partial}{\partial z_k^p}\wedge \frac{\partial}{\partial z_k^q}
\end{align}

Let us compute
\begin{align} \label{o28}
[\Lambda(t),\dot{\theta}_{jk}(t)]&=[\sum_{\alpha,\beta=1}^n g_{\alpha\beta}^j(z_j,t)\frac{\partial}{\partial z_j^\alpha}\wedge \frac{\partial}{\partial z_j^\beta},\sum_{r=1}^n \dot{\theta}_{jk}^r(z_j,t)\frac{\partial}{\partial z_j}]\\
&=2\sum_{\alpha,\beta,p=1}^ng_{p\beta}^j\frac{\partial \dot{\theta}_{jk}^\alpha}{\partial z_j^r}\frac{\partial}{\partial z_j^\alpha}\wedge \frac{\partial}{\partial z_j^\beta}-\sum_{\alpha,\beta,r=1}^n\dot{\theta}_{jk}^r\frac{\partial g_{\alpha\beta}^j}{\partial z_j^r}\frac{\partial}{\partial z_j^\alpha}\wedge \frac{\partial}{\partial z_j^\beta} \notag
\end{align}

Let us compute
\begin{align} \label{o29}
[\lambda_j(t),\theta_{jk}(t)]&=[\sum_{\alpha,\beta=1}^n \lambda_{\alpha\beta}^j(z_j,t)\frac{\partial}{\partial z_j^\alpha}\wedge \frac{\partial}{\partial z_j^\beta},\sum_{r=1}^n \theta_{jk}^r(z_j,t)\frac{\partial}{\partial z_j^r}]\\
&=2\sum_{\alpha,\beta,p=1}^n\lambda_{p\beta}^j\frac{\partial \theta_{jk}^\alpha}{\partial z_j^r}\frac{\partial}{\partial z_j^\alpha}\wedge \frac{\partial}{\partial z_j^\beta}-\sum_{\alpha,\beta,r=1}^n\theta_{jk}^r\frac{\partial \lambda_{\alpha\beta}^j}{\partial z_j^r}\frac{\partial}{\partial z_j^\alpha}\wedge \frac{\partial}{\partial z_j^\beta}\notag\end{align}

Let us compute
\begin{align} \label{o30}
[\lambda_k(t),\theta_{jk}(t)]&=[\sum_{p,q=1}^n\lambda_{pq}^k(z_k,t)\frac{\partial}{\partial z_k^p}\wedge \frac{\partial}{\partial z_k^q} , \sum_{r=1}^n \theta_{jk}^r(z_j,t)\frac{\partial}{\partial z_j^r}]=[\sum_{\alpha,\beta,p,q=1}^n\lambda_{pq}^k(z_k,t)\frac{\partial f_{jk}^\alpha}{\partial z_k^p}\frac{\partial f_{jk}^\beta}{\partial z_k^q}\frac{\partial}{\partial z_j^\alpha}\wedge \frac{\partial}{\partial z_j^\beta} , \sum_{r=1}^n \theta_{jk}^r(z_j,t)\frac{\partial}{\partial z_j^r}]  \\
&=2\sum_{\alpha,\beta,p,q=1}^n \lambda_{pq}^k(z_k,t)\frac{\partial \theta_{jk}^\alpha}{\partial z_k^p}\frac{\partial f_{jk}^\beta}{\partial z_k^q} \frac{\partial}{\partial z_j^\alpha}\wedge \frac{\partial}{\partial z_j^\beta}-\sum_{\alpha,\beta,p,q,r =1}^n \theta_{jk}^r\frac{\partial}{\partial z_j^r}\left( \lambda_{pq}^k(z_k,t)\frac{\partial f_{jk}^\alpha}{\partial z_k^p}\frac{\partial f_{jk}^\beta}{\partial z_k^q} \right)\frac{\partial}{\partial z_j^\alpha}\wedge \frac{\partial}{\partial z_j^\beta} \notag
\end{align}

By considering $(\ref{o251}),(\ref{o27}),(\ref{o28}),(\ref{o29}),(\ref{o30})$, we get $(\ref{o261})$.

\end{proof}

From $(\ref{o35}),(\ref{o36})$ and $(\ref{o31})$,  and substituting $t=0$, we obtain
 \begin{align}\label{o37}
 \dot{\theta}_{ij}(0)-\dot{\theta}_{ik}(0)+\dot{\theta}_{jk}(0)&=[\theta_{ij}(0),\theta_{jk}(0)]   \notag  \\
\dot{\lambda}_k(0)-\dot{\lambda}_j(0)+[\Lambda_0,\dot{\theta}_{jk}(0)]&=-[\lambda_j(0)+\lambda_k(0),\theta_{jk}(0)]  \\
[\Lambda_0,\dot{\lambda}_j(0)]&=-[\lambda_j(0),\lambda_j(0)] \notag  
 \end{align}

Since $(M,\Lambda_0)=(M_0,\Lambda_0)$ by assumption, identifying $U_j$ with $U_j\times 0$, we may consider $\mathcal{U}=\{U_j\}$ as a finite open covering of $M$. For a given $(\lambda,\theta)\in \mathbb{H}^1(M,\Theta_M^\bullet)$, if $(\theta,\lambda)=\left( \frac{\partial (M_t,\Lambda_t)}{\partial t}\right)_{t=0}$, where $(\lambda, \theta)$ is the cohomology class of the $1$-cocycle $(\{\lambda_j(0)\},\{\theta_{jk}(0)\})\in C^0(\mathcal{U}, \wedge^2 \Theta_M)\oplus C^1(\mathcal{U},\Theta_M)$ as above, then $(\ref{o37})$ imposes a certain restriction on such $(\lambda,\theta)$. 

\begin{lemma}
For any $1$-cocycle $(\{\lambda_j\},\{\theta_{jk}\})\in C^0(\mathcal{U},\wedge^2 \Theta_M)\oplus C^1(\mathcal{U},\Theta_M)$ of $\Theta_M^\bullet$ so that $[\Lambda_0, \lambda_j]=0,\lambda_k-\lambda_j+[\Lambda_0, \theta_{jk}]=0 $ and $\theta_{jk}-\theta_{ik}+\theta_{ij}=0$,
\begin{align*}
(\{\gamma_j:=-[\lambda_j,\lambda_j ] \},\{\eta_{jk}:=-[\lambda_j+\lambda_k,\theta_{jk}]\}, \{\xi_{ijk}:=[\theta_{ij},\theta_{jk}]\})\in C^0(\mathcal{U},\wedge^3 \Theta_M)\oplus C^1(\mathcal{U},\wedge^2 \Theta_M)\oplus C^2(\mathcal{U},\Theta_M)
\end{align*}
defines a $2$-cocycle in \v{C}ech resolution $(\ref{o6})$ of $\Theta_M^\bullet$.
\end{lemma}
\begin{proof}
First we note that on $U_i\cap U_j\cap U_k\ne \emptyset$, $\xi_{ikj}=-\xi_{ijk}=\xi_{jik}$ and on $U_h\cap U_j\cap U_j\cap U_k\ne \emptyset$, $\xi_{ijk}-\xi_{hjk}+\xi_{hik}-\xi_{hij}=0$ (for the detail, see \cite{Kod05} p.213). Next we note that $[\Lambda_0, \{\gamma_j\}]=0$, $\eta_{kj}=-[\lambda_k+\lambda_j,\theta_{kj}]=[\lambda_j+\lambda_k,\theta_{jk}]=-\eta_{jk}$, and
\begin{align*}
-\delta(\{\gamma_{j}\})+[\Lambda_0,\{\eta_{jk}\}]=&-\delta(\{-[\lambda_j,\lambda_j]\})+[\Lambda_0, \{-[\lambda_j+\lambda_k,\theta_{jk}]\}]\\
&=\{ [\lambda_k,\lambda_k]-[\lambda_j,\lambda_j]-[\Lambda_0,[\lambda_j,\theta_{jk}]]-[\Lambda_0,[\lambda_k,\theta_{jk}]] \}\\
&=\{[\lambda_k,\lambda_k]-[\lambda_j,\lambda_j]-(-1)^9[\lambda_j, [\Lambda_0,\theta_{jk}]]-(-1)^9[\lambda_k,[\Lambda_0,\theta_{jk}]]\}\\
&=\{[\lambda_k,\lambda_k]-[\lambda_j,\lambda_j]+[\lambda_j,\lambda_j-\lambda_k]+[\lambda_k,\lambda_j-\lambda_k]\}=0
\end{align*}
Lastly we note that
\begin{align*}
-\delta(\{\eta_{jk}\})+[\Lambda_0,\{\xi_{ijk}\}]=&-\delta(\{-[\lambda_j+\lambda_k,\theta_{jk}]\})+[\Lambda_0,\{[\theta_{ij},\theta_{jk}]\} ]\\
&=\{[\lambda_j+\lambda_k,\theta_{jk}]-[\lambda_i+\lambda_k,\theta_{ik}]+[\lambda_{i}+\lambda_j,\theta_{ij}]+[[\Lambda_0,\theta_{ij}],\theta_{jk}]+[\theta_{ij},[\Lambda_0,\theta_{jk}]]\}\\
&=\{[\lambda_j+\lambda_k,\theta_{jk}]-[\lambda_i+\lambda_k,\theta_{ik}]+[\lambda_{i}+\lambda_j,\theta_{ij}]+[\lambda_i-\lambda_j,\theta_{jk}]+[\theta_{ij},\lambda_j-\lambda_k]\}\\
&=\{[\lambda_k,\theta_{jk}]-[\lambda_i,\theta_{ik}]-[\lambda_k,\theta_{ik}]+[\lambda_i,\theta_{ij}]+[\lambda_i,\theta_{jk}]+[\lambda_k,\theta_{ij}]\}=0
\end{align*}
\end{proof}

Then from $(\ref{o37})$ which is equivalent to
\begin{align*}
[\Lambda_0,\{\dot{\lambda}_j(0)\}]=\{-[\lambda_j(0),\lambda_j(0)] \},\,\,\,\delta(\dot{\lambda}_j(0))+[\Lambda_0,\{\dot{\theta}_{jk}\}]=\{-[\lambda_j(0)+\lambda_k(0),\theta_{jk}(0)]\},\,\,\, \delta(\{\dot{\theta}_{jk}\})=\{[\theta_{ij}(0),\theta_{jk}(0)]\},
\end{align*}
 we obtain the following theorem.
\begin{theorem}\label{o50}
Suppose given a compact holomorphic Poisson manifold $(M,\Lambda_0)$, and $(\lambda,\theta)=(\{\lambda_i\},\{\theta_{jk}\})\in \mathbb{H}^1(M,\Theta_M^\bullet)$. In order that there may exist a Poisson analytic family $(\mathcal{M},\Lambda, B,\omega)$ such that $\omega^{-1}(0)=(M,\Lambda_0)$, and that $\left( \frac{\partial (M_t,\Lambda_t)}{\partial t}\right)_{t=0}=(\theta, \lambda)$, it is necessary that 
\begin{align*}
(\{-[\lambda_j,\lambda_j]\},\{-[\lambda_j+\lambda_k,\theta_{jk}]\}, \{[\theta_{ij},\theta_{jk}]\})\in C^0(\mathcal{U},\wedge^3 \Theta_M)\oplus C^1(\mathcal{U},\wedge^2 \Theta_M)\oplus C^2(\mathcal{U},\Theta_M)
\end{align*}
 defines the $0$ class in $\mathbb{H}^2(M,\Theta_M^\bullet)$.
\end{theorem}

Given $(\lambda,\theta)=(\{\lambda_i\},\{\theta_{jk}\})\in \mathbb{H}^1(M,\Theta_M^\bullet)$, if the $2$-cocycle $(\{-[\lambda_j,\lambda_j]\},\{-[\lambda_j+\lambda_k,\theta_{jk}]\},\{[\theta_{ij},\theta_{jk}]\})$ is not $0$ in $\mathbb{H}^2(M,\Theta_M^\bullet)$, there is no Poisson deformation $(M_t,\Lambda_t)$ with $(M_0,\Lambda_0)=(M,\Lambda_0)$ and $\left(\frac{\partial (M_t,\Lambda_t)}{\partial t} \right)=(\theta,\lambda)$.

\begin{example}
While $(\mathbb{P}_\mathbb{C}^2,\Lambda_0=0)$ is unobstructed in Poisson deformations as in Example $\ref{o19}$, $(\mathbb{P}_\mathbb{C}^3,\Lambda_0=0)$ is obstructed in Poisson deformations. We note that $\mathbb{H}^1(\mathbb{P}_\mathbb{C}^3,\Theta_{\mathbb{P}_\mathbb{C}^3}^\bullet)=H^0(\mathbb{P}_\mathbb{C}^3,\wedge^2 \Theta_{\mathbb{P}_\mathbb{C}^3})$. $(\mathbb{P}_\mathbb{C}^3,\Lambda_0=0)$ is obstructed in Poisson deformations since there is an element $\Pi\in H^0(\mathbb{P}_\mathbb{C}^3,\wedge^2 \Theta_{\mathbb{P}_\mathbb{C}^3})$ such that $[\Pi,\Pi]\ne 0$.

\end{example}

\subsection{Description of obstructed elements in Dolbeault resolution of $\Theta_M^\bullet$}\

We have simpler description of obstructed elements in Poisson deformations when we use Dolbeault resolution $(\ref{o7})$ of $\Theta_M^\bullet$. Let  $(M,\Lambda_0)$ be a compact holomorphic Poison manifold, and $\Delta\subset \mathbb{C}$ be a neighborhood of $0$. As in Remark \ref{o9}, let $(\Lambda(t),\varphi(t))$ be convergent power series in $t\in \Delta$ with coefficients in $A^{0,0}(M, \wedge^2 \Theta_M)\oplus A^{0,1}(M, \Theta_M)$ such that $\varphi(0),\Lambda(0)=0$ and 
\begin{align}\label{o40}
L \alpha(t)+\frac{1}{2}[\alpha(t),\alpha(t)]=0,\,\,\,\,\,\,\,\,\,\,L=\bar{\partial}-+[\Lambda_0,-]
\end{align}
where $\alpha(t):=\varphi(t)+\Lambda(t)$ so that $\alpha(t)$ defines a Poisson analytic family $(\mathcal{M},\Lambda, \Delta, \omega)$ of deformations of $\omega^{-1}(0)=(M, \Lambda_0)$. We will denote $\frac{\partial^n \alpha(t)}{\partial t^n}$ by $\alpha^{(n)}(t)$. Then by taking the derivative of $(\ref{o40})$ with respect to $t$, we get
\begin{align}\label{o41}
L\alpha'(t)=-[\alpha(t),\alpha'(t)]
\end{align}
By taking the derivative of $(\ref{o41})$ with respect to $t$, and setting $t=0$, we get
\begin{align*}
L\alpha''(0)=-[\alpha(0),\alpha''(0)]-[\alpha'(0), \alpha'(0)]=-[\alpha'(0),\alpha'(0)]
\end{align*}

Hence
\begin{align*}
[\alpha'(0),\alpha'(0)]=[\Lambda(0),\Lambda(0)]+2[\Lambda(0),\varphi(0)]+[\varphi(0),\varphi(0)]\in A^{0,0}(M,\wedge^2 \Theta_M)\oplus A^{0,1}(M,\wedge^2 \Theta_M)\oplus A^{0,2}(M,\Theta_M)
\end{align*}
defines $0$ class in $\mathbb{H}^2(M,\Theta_M^\bullet)$. We can check that if $(\lambda, \theta)\in A^{0,0}(M, \wedge^2 \Theta_M)\oplus A^{0,1}(M, \Theta_M)$ defines an element in $\mathbb{H}^1(M, \Theta_M^\bullet)$ so that $[\Lambda_0, \lambda]=0, \bar{\partial} \lambda+[\Lambda_0, \theta]=0$ and $\bar{\partial}\theta=0$, then $[\theta+\lambda, \theta+\lambda]$ defines an element in $\mathbb{H}^2(M, \Theta_M^\bullet)$. Then we obtain the following theorem.

\begin{theorem}\label{o51}
Suppose given a compact holomorphic Poisson manifold $(M, \Lambda_0)$ and $(\lambda, \theta)\in A^{0,0}(M, \wedge^2 \Theta_M)\oplus A^{0,1}(M, \Theta_M)$ defines an element in $\mathbb{H}^1(M, \Theta_M^\bullet)$. In order that there exist a Poisson analytic family $(\mathcal{M},\Lambda, B, \omega)$ such that $\omega^{-1}(0)=(M, \Lambda_0)$, and $\left( \frac{\partial (M_t, \Lambda_t)}{\partial t}   \right)_{t=0}=(\theta,\lambda)$, it is necessary that 
\begin{align*}
[\theta+\lambda, \theta+\lambda]=[\lambda,\lambda]+2[\lambda, \theta]+[\theta,\theta]\in A^{0,0}(M, \wedge^2 \Theta_M)\oplus A^{0,1}(M, \wedge^2 \Theta_M)\oplus A^{0,2}(M, \Theta_M)
\end{align*}
define the $0$ class in $\mathbb{H}^2(M, \Theta_M^\bullet)$.
\end{theorem}

\section{Rational ruled surfaces}\label{section3}

In this section, we study Poisson deformations of rational ruled surfaces $F_m=\mathbb{P}(\mathcal{O}_{\mathbb{P}^1_\mathbb{C}}(m)\oplus \mathcal{O}_{\mathbb{P}^1_\mathbb{C}}),m\geq 0$. It is known that since $H^2(F_m,\Theta_{F_m})=0$, $F_m$ are unobstructed in complex deformations. We determine obstructedness or unobstructedness in Poisson deformations for any holomorphic Poisson structure on $F_m$. We show that $F_0,F_1, F_2, F_3$ are unobstructed in Poisson deformations for any holomorphic Poisson structure. For $m\geq 4$, $(F_m, \Lambda_0)$ have both obstructed and unobstructed Poisson deformations depending on Poisson structure $\Lambda_0$. For unobstructed Poisson deformations, we will explicitly give examples of Poisson analytic families of deformations of $(F_m,\Lambda_0)$ such that the associated Poisson Kodaira-Spencer map is an isomorphism at the distinguished point.

\begin{remark}\label{r1}
We recall that $F_m$ is constructed in the following way: take two copies of $U_i\times \mathbb{P}_\mathbb{C}^1$ where $U_i=\mathbb{C},i=1,2$, and write the coordinates as $(z,\xi)\in U_1\times \mathbb{P}_\mathbb{C}^1$ and $(z',\xi')\in U_2\times \mathbb{P}_\mathbb{C}^1$, where $z,z'$ are the coordinates of $\mathbb{C}$, respectively and $\xi,\xi'$ are the inhomogenous coordinates of $\mathbb{P}_\mathbb{C}^1$, respectively. Then $F_m$ is constructed by patching $U_i\times \mathbb{P}_\mathbb{C}^1, i=1,2$ by the relation $z'=\frac{1}{z}$ and $\xi'=z^m \xi$.
\end{remark}

\subsection{Cohomology groups $H^i(F_m, \wedge^j \Theta_{F_m})$}\

We explicitly describe cohomology groups $H^i(F_m, \wedge^j \Theta_{F_m}), i=0,1,2,j=1,2$. In the sequel we keep the notations in Remark \ref{r1}.

\subsubsection{Descriptions of $H^0(F_m, \Theta_{F_m})$}\ 

$H^0(F_m, \Theta_{F_m})$ are described on $U_1\times \mathbb{P}_\mathbb{C}^1$ (see \cite{Kod05} p.73-75) in the following way.
\begin{align}\label{r12}
g(z)\frac{\partial}{\partial z}+(a(z)+b(z)\xi+c(z)\xi^2)\frac{\partial}{\partial \xi}
\end{align}
\begin{enumerate}
\item In the case of $m=0$, we have $g(z)=g_0+g_1z+g_2z^2$, $a(z)=a$, $b(z)=d$, and $c(z)=c_0$ so that $H^0(F_0, \Theta_{F_0})\cong \mathbb{C}^6$.
\item In the case of $m\geq 1$, we have $g(z)=g_0+g_1z + g_2 z^2$, $a(z)=0$, $b(z)=-mg_2z+d$, and $c(z)=c_0+c_1z+\cdots +c_m z^m$ so that $H^0(F_m, \Theta_{F_m})\cong \mathbb{C}^{m+5}$.
\end{enumerate}

\subsubsection{Descriptions of $H^1(F_m, \Theta_{F_m})$}\

Let $\mathcal{U}=\{U_1\times \mathbb{P}_\mathbb{C}^1, U_2\times \mathbb{P}_\mathbb{C}^1\}$ be the open covering of $F_m$ as in Remark \ref{r1}. We can compute $H^i(F_m, \Theta_{F_m})$ by \v{C}ech resolution of $\Theta_{F_m}$ by using the open covering $\mathcal{U}$ of $F_m$. Then we have (see \cite{Kod05} p.312)
\begin{align*}
\dim_\mathbb{C} H^1(F_m, \Theta_{F_m})=
\begin{cases}
0,  & m=0,1\\
m-1, & m\geq 2.
\end{cases}
\end{align*}
and for $m\geq 2$, the $1$-cocycles
\begin{align}\label{r10}
\frac{1}{z^k}\frac{\partial}{\partial \xi}=z'^k\frac{\partial}{\partial \xi} \in C^1(\mathcal{U},\Theta_{F_m}),\,\,\,\,\,k=1,...,m-1
\end{align}
forms a  basis of $H^1(F_m,\wedge^2 \Theta_{F_m})$.

\subsubsection{Holomorphic Poisson structures $H^0(F_m, \wedge^2 \Theta_{F_m})$ on ${F}_m, m\geq 0$}\

We describe holomorphic Poisson structures on rational ruled surfaces $F_m$ explicitly. We note that $\frac{\partial}{\partial z'}=-z^2\frac{\partial}{\partial z}+mz\xi\frac{\partial}{\partial \xi}$ and $\frac{\partial}{\partial \xi'}=z^{-m}\frac{\partial}{\partial \xi}$ so that $\frac{\partial}{\partial z'}\wedge \frac{\partial}{\partial \xi'}=-z^{-m+2}\frac{\partial}{\partial z}\wedge \frac{\partial}{\partial \xi}$. We also note that  a holomorphic bivector field on $U_1\times \mathbb{P}_\mathbb{C}^1$ is of the form 
\begin{align}\label{r2}
(d(z)+e(z)\xi+f(z)\xi^2)\frac{\partial}{\partial z}\wedge \frac{\partial}{\partial \xi}, 
\end{align}
and a holomorphic bivector field on $U_2\times \mathbb{P}_\mathbb{C}^1$ is of the form 
\begin{align}\label{r3}
(p(z')+q(z')\xi'+r(z')\xi'^2)\frac{\partial}{\partial z'}\wedge \frac{\partial}{\partial \xi'}, 
\end{align}
where $d(z),e(z),f(z)$ are entire functions of $z$ and  $p(z'),q(z'),r(z')$ are entire functions of $z'$. For a holomorphic bivector field on $F_m$ which has the form $(\ref{r2})$ on $U_1\times \mathbb{P}_\mathbb{C}^1$ and, the form $(\ref{r3})$ on $U_2\times \mathbb{P}_\mathbb{C}^1$, we must have $d(z)+e(z)\xi+f(z)\xi^2=-(p(\frac{1}{z})+q(\frac{1}{z})z^m \xi+r(\frac{1}{z})z^{2m}\xi^2)z^{-m+2}=-p(\frac{1}{z})z^{-m+2}-q(\frac{1}{z})z^{2}\xi-r(\frac{1}{z})z^{m+2}\xi^2$
so that
\begin{align*}
d(z)=-p\left(\frac{1}{z}\right)z^{-m+2},\,\,\,e(z)=-q\left(\frac{1}{z}\right)z^{2},\,\,\, f(z)=-r\left(\frac{1}{z}\right)z^{m+2}
\end{align*}
\begin{enumerate}
\item In the case of $m=0$, we have $d(z)=d_0+d_1z+d_2z^2 ,e(z)=e_0+e_1z+e_2z^2 ,f(z)=f_0+f_1z+f_2z^2$ so that $H^0(F_0,\wedge^2 \Theta_{F_0})\cong \mathbb{C}^9$. \label{nn1}
\item In the case of $m=1$, we have $d(z)=d_0+d_1z, e(z)=e_0+e_1z+e_2z^2, f(z)=f_0+f_1z+f_2z^2+f_3z^3$ so that $H^0(F_1,\wedge^2 \Theta_{F_1})\cong \mathbb{C}^9$.\label{nn2}
\item In the case of $m=2$, we have $d(z)=d_0,e(z)=e_0+e_1z+e_2z^2,f(z)=f_0+f_1z+f_2z^2+f_3z^3+f_4z^4$ so that $H^0(F_2,\wedge^2 \Theta_{F_2})\cong \mathbb{C}^9$.\label{nn3}
\item In the case of $m\geq 3$, we have $d(z)=0,e(z)=e_0+e_1z+e_2z^2,f(z)=f_0+f_1z+\cdots f_{m+2}z^{m+2}$ so that $H^0(F_m,\wedge^2 \Theta_{F_m})\cong \mathbb{C}^{m+6}$. \label{nn4}
\end{enumerate}

\subsubsection{Descriptions of $H^1(F_m ,\wedge^2 \Theta_{F_m})$}\

Let $\mathcal{U}=\{U_1\times \mathbb{P}_\mathbb{C}^1, U_2\times \mathbb{P}_\mathbb{C}^1\}$ be the open covering of $F_m$ as in Remark \ref{r1}. We can compute $H^i(F_m, \wedge^2 \Theta_{F_m})$ by \v{C}ech resolution of $\Theta_{F_m}$ by using the open covering $\mathcal{U}$ of $F_m$. We represent a $1$-cocycle $\{\Lambda_{12},\Lambda_{21}\}\in C^1(\mathcal{U},\wedge^2\Theta_{F_m})$ by the holomorphic bivector fields with $\Lambda_{12}=-\Lambda_{21}$ on $(U_1\times \mathbb{P}_\mathbb{C}^1)\cap (U_2\times \mathbb{P}_\mathbb{C}^1)$. Then this $1$-cocycle belongs to $\delta C^0(\mathcal{U}, \wedge^2 \Theta_{F_m})$ if and only if  there exist  holomorphic bivector field $\Lambda_1$ and $\Lambda_2$ respectively on $U_1\times \mathbb{P}_\mathbb{C}^1$ and $U_2\times \mathbb{P}_\mathbb{C}^1$ such that
\begin{align*}
\Lambda_2-\Lambda_1=\Lambda_{12}
\end{align*}
We write $\Lambda_1$ in the form of $(\ref{r2})$ and $\Lambda_2$ in the form of $(\ref{r3})$. We write $\Lambda_{12}$ in terms of the coordinates $(z,\xi):$
\begin{align*}
\Lambda_{12}=(\alpha(z)+\beta(z)\xi+\gamma(z)\xi^2)\frac{\partial}{\partial z}\wedge \frac{\partial}{\partial \xi}
\end{align*}
where $\alpha(z),\beta(z), \gamma(z)$ are holomorphic functions of $z$ on $U_1\cap U_2=\mathbb{C}^*$ so that they are expanded into Laurent series in $z$. In terms of the coordinates $(z,\xi)$ with $z=\frac{1}{z'},\xi=z'^m \xi'$ in place of $(z',\xi')$, $\Lambda_2$ is written in the form
\begin{align*}
-\left(p\left(\frac{1}{z}\right) z^{-m+2}+q\left(\frac{1}{z}\right)z^2\xi  +r\left(\frac{1}{z}\right)z^{m+2}\xi^2 \right)\frac{\partial}{\partial z}\wedge \frac{\partial}{\partial \xi}
\end{align*}
Hence the equation is reduced to the following system of equations:
\begin{align*}
\begin{cases}
-p\left(\frac{1}{z}\right)z^{-m+2}-d(z)&=\alpha(z)\\
-q\left(\frac{1}{z}\right)z^2-e(z)&=\beta(z)\\
-r\left(\frac{1}{z}\right)z^{m+2}-f(z)&=\gamma(z)
\end{cases}
\end{align*}
In case of $m=0,1,2,3$, these equations always have a solution. For $m\geq 4$, let $\alpha(z)=\sum_{n=-\infty}^{\infty} c_nz^n$ be the Laurent expansion of $\alpha(z)$. Then the above equations have a solution if and only if $c_{-1}=c_{-2}=\cdots =c_{-(m-3)}=0$. Hence we obtain
\begin{align*}
\dim_\mathbb{C} H^1(F_m, \wedge^2 \Theta_{F_m})=
\begin{cases}
0,  & m=0,1,2,3\\
m-3, & m\geq 4.
\end{cases}
\end{align*}
We note that for $m\geq 4$, the $1$-cocycles
\begin{align}\label{r11}
\frac{1}{z^k}\frac{\partial}{\partial z}\wedge \frac{\partial}{\partial \xi}=z'^k\frac{\partial}{\partial z}\wedge \frac{\partial}{\partial \xi}\in C^1(\mathcal{U}, \wedge^2 \Theta_{F_m}),\,\,\,\,\,k=1,...,m-3
\end{align}
forms a  basis of $H^1(F_m,\wedge^2 \Theta_{F_m})$.

\subsubsection{Descriptions of $H^2(F_m, \Theta_{F_m})$ and $H^2(F_m, \wedge^2 \Theta_{F_m})$}\

We have $H^2(F_m, \Theta_{F_m})=0$ for any $m\geq 0$ (see \cite{Kod05} p.312), and similarly we can show that $H^2(F_m, \wedge^2 \Theta_{F_m})=0$ for any $m \geq 0$.

\begin{lemma}\label{r4}
Let $(F_m,\Lambda_0), m \geq 0$ be a Poisson rational ruled surface. Then 
\begin{align*}
\mathbb{H}^0(F_m, \Theta_{F_m}^\bullet)&\cong ker(H^0(F_m, \Theta_{F_m})\xrightarrow{[\Lambda_0,-]} H^0(F_m, \wedge^2 \Theta_{F_m}))\\
\mathbb{H}^1(F_m, \Theta_{F_m}^\bullet)&\cong coker(H^0(F_m, \Theta_{F_m})\xrightarrow{[\Lambda_0,-]} H^0(F_m, \wedge^2 \Theta_{F_m}))\oplus ker(H^1(F_m, \Theta_{F_m})\xrightarrow{[\Lambda_0,-]} H^1(F_m,\wedge^2 \Theta_{F_m}))\\
\mathbb{H}^2(F_m,\Theta_{F_m}^\bullet)&\cong coker(H^1(F_m,\Theta_{F_m})\xrightarrow{[\Lambda_0,-]}H^1(F_m,\wedge^2 \Theta_{ F_m}))
\end{align*}
Let $\mathcal{U}=(U_1 \times \mathbb{P}_\mathbb{C}^1,U_2 \times \mathbb{P}_\mathbb{C}^1)$ be the open covering of $F_m$ as in Remark $\ref{r1}$. Then $(F_m,\Lambda_0)$ is obstructed in Poisson deformations if for some $a , b$ where $a\in H^0(F_m,\wedge^2 \Theta_{F_m})$, and $b\in C^1(\mathcal{U},\Theta_{F_m})$ which defines an element in $ker(H^1(F_m, \Theta_{F_m})\xrightarrow{[\Lambda_0,-]} H^1(F_m,\wedge^2 \Theta_{F_m}))$, under the following map
\begin{align*}
[-,-]:H^0(F_m, \wedge^2 \Theta_{F_m})\times H^1(F_m, \Theta_{F_m})\to H^1(F_m, \wedge^2 \Theta_{F_m})
\end{align*}
$[a,b]\in H^1(F_m,\wedge^2 \Theta_{F_m})$ is not in the image of $ H^1(F_m, \Theta_{F_m})\xrightarrow{[\Lambda_0,-]} H^1(F_m,\wedge^2 \Theta_{F_m}) $.
\end{lemma}

\begin{proof}
We can compute $\mathbb{H}^i(F_m,\Theta_{F_m}^\bullet)$ by using the following \v{C}ech resolution of $\Theta_{F_m}^\bullet$. By considering the spectral sequence associated with the double complex $(\ref{o6})$ and $H^2(F_m, \Theta_{F_m})=0$, we get the first claim.

Let us prove the second claim. Assume that $(F_m,\Lambda_0)$ is unobstructed in Poisson deformations. Choose any element $a,$ and $b=\{b_{jk}\}$ where $a\in H^0(F_m,\wedge^2 \Theta_{F_m})$, and $b=\{b_{jk}\}\in C^1(\mathcal{U},\Theta_{F_m})$ which is in $ker(H^1(F_m, \Theta_{F_m})\xrightarrow{[\Lambda_0,-]} H^1(F_m,\wedge^2 \Theta_{F_m}))$. Then there exists $\{c_i\}\in C^0(\mathcal{U}, \wedge^2 \Theta_{F_m})$ such that $c_k-c_j+[\Lambda_0, b_{jk}]=0$. We note that $(\{c_i\},\{b_{jk}\})$ and $(\{c_i+a\},\{b_{jk}\})$ define elements in $\mathbb{H}^1(F_m,\Theta_{F_m}^\bullet)$. Since $(F_m,\Lambda_0)$ is unobstructed,
\begin{align*}
(0,\{-[c_j+c_k,b_{jk}]\},\{[b_{ij},b_{jk}]\}),\,\,\,\,\, (0,\{-[c_j+c_k+2a,b_{jk}]\},\{[b_{ij},b_{jk}]\})
\end{align*}
define the $0$ class in $\mathbb{H}^2(F_m,\wedge^2 \Theta_{F_m})$. Hence there exist $(\{d_i\},\{e_{jk}\})$ and $(\{d_i'\},\{e_{jk}'\})\in C^0(\mathcal{U},\wedge^2 \Theta_{F_m})\oplus C^1(\mathcal{U},\Theta_{F_m})$ such that $\delta(\{d_i\})+[\Lambda_0,\{e_{jk}\}]=\{-[c_j+c_k,b_{jk}]\}$ and $\delta(\{d_i'\})+[\Lambda_0,\{e_{jk}'\}]=\{-[c_j+c_k+2a,b_{jk}]\}$, and we have $\delta(\{e_{jk}\})=\{[b_{ij},b_{jk}]\}$ and $\delta(\{e_{jk}'\})=\{[b_{ij},b_{jk}]\}$ so that $\{e_{jk}-e_{jk}'\} \in C^1(\mathcal{U},\Theta_{F_m})$ defines an element in $H^1(F_m,\Theta_{F_m})$. Then we have
\begin{align*}
\delta(\{\frac{1}{2}(d_i-d_i')\})+[\Lambda_0,\{\frac{1}{2}(e_{jk}-e_{jk}')\}]=\{[a, b_{jk}]\}
\end{align*}
Hence $\{[a,b_{jk}]\}$ and $[\Lambda_0,\{\frac{1}{2}(e_{jk}-e_{jk}')\}]$ define the same class in $H^1(F_m,\wedge^2 \Theta_{F_m})$ so that $\{[a,b_{jk}]\} \in H^1(F_m,\wedge^2 \Theta_{F_m})$ is  in the image of $ H^1(F_m, \Theta_{F_m})\xrightarrow{[\Lambda_0,-]} H^1(F_m,\wedge^2 \Theta_{F_m}) $.

\end{proof}

\subsection{Obstructedness or unobstructedness of $F_m$ in Poisson deformations}\

In this subsection, we determine the obstructedness or unobstructedness of Poisson rational ruled surfaces $(F_m,\Lambda_0)$. In the case of unobstructed $(F_m,\Lambda_0)$ in Poisson deformations, we prove the unobstructednss by showing that $\mathbb{H}^2(F_m, \Theta_{F_m}^\bullet)=0$, and we show that if $\mathbb{H}^2(F_m, \Theta_{F_m}^\bullet)\ne 0$, $(F_m, \Lambda_0)$ is obstructed in Poisson deformations.

\subsubsection{Computation of $\mathbb{H}^2(F_m, \Theta_{F_m}^\bullet)$}\

We note that for a Poisson rational ruled surface $(F_m, \Lambda_0)$, we have $\mathbb{H}^2(F_m, \Theta_{F_m}^\bullet)\cong coker(H^1(F_m,\Theta_{F_m})\xrightarrow{[\Lambda_0,-]} H^1(F_m, \wedge^2 \Theta_{F_m}))$ by Lemma \ref{r4}. Hence if $H^1(F_m, \Theta_{F_m})\xrightarrow{[\Lambda_0,-]} H^1(F_m, \wedge^2 \Theta_{F_m})$ is surjective, we have $\mathbb{H}^2(F_m, \Theta_{F_m}^\bullet)=0$ so that $(F_m, \Lambda_0)$ is unobstructed in Poisson deformations. Now we identify the condition when $\mathbb{H}^2(F_m, \Theta_{F_m}^\bullet)=0$. 

First we note that $H^1(F_m, \wedge^2 \Theta_{F_m})=0$ for $m=0,1,2,3$ so that $\mathbb{H}^2(F_m, \Theta_{F_m}^\bullet)=0$ and so $(F_m, \Lambda_0)$ are unobstructed in Poisson deformations for any Poisson structure $\Lambda_0$ on $F_m$, where $m=0,1,2,3$.

Now assume that $m\geq 4$. Let us describe $H^1(F_m,\Theta_{F_m})\xrightarrow{[\Lambda_0,-]} H^1(F_m, \wedge^2 \Theta_{F_m})$ by using $(\ref{r2})$, $(\ref{r10})$ and $(\ref{r11})$  in general. Let $(F_m, \Lambda_0=(e(z)\xi+f(z)\xi^2)\frac{\partial}{\partial z}\wedge \frac{\partial}{\partial \xi})$ be a Poisson rational ruled surface. Then we have
\begin{align*}
&[\Lambda_0, \left(\frac{\alpha_1}{z}+\frac{\alpha_2}{z}+\cdots+\frac{\alpha_{m-1}}{z^{m-1}}\right)\frac{\partial}{\partial \xi}]=[(e(z)\xi+f(z)\xi^2)\frac{\partial}{\partial z}\wedge \frac{\partial}{\partial \xi},\left(\frac{\alpha_1}{z}+\frac{\alpha_2}{z}+\cdots+\frac{\alpha_{m-1}}{z^{m-1}}\right)\frac{\partial}{\partial \xi}]\\
&=[(e(z)\xi+f(z)\xi^2)\frac{\partial}{\partial z},\left(\frac{\alpha_1}{z}+\frac{\alpha_2}{z^2}+\cdots+\frac{\alpha_{m-1}}{z^{m-1}}\right)\frac{\partial}{\partial \xi}]\wedge \frac{\partial}{\partial \xi}\\
&=-\left(\frac{\alpha_1}{z}+\frac{\alpha_2}{z^2}+\cdots+\frac{\alpha_{m-1}}{z^{m-1}}\right)(e(z)+2f(z)\xi)\frac{\partial}{\partial z}\wedge \frac{\partial}{\partial \xi}\equiv -\left(\frac{\alpha_1}{z}+\frac{\alpha_2}{z^2}+\cdots+\frac{\alpha_{m-1}}{z^{m-1}}\right)e(z)\frac{\partial}{\partial z}\wedge \frac{\partial}{\partial \xi}\\
&=-\left(\frac{\alpha_1}{z}+\frac{\alpha_2}{z^2}+\cdots+\frac{\alpha_{m-1}}{z^{m-1}}\right)(e_0+e_1z+e_2z^2)\frac{\partial }{\partial z}\wedge \frac{\partial}{\partial \xi}\\
&\equiv -\left(\frac{\alpha_1e_0+\alpha_2e_1+\alpha_3e_2}{z}+\frac{\alpha_2e_0+\alpha_3e_1+\alpha_4e_2}{z^2}+\frac{\alpha_3e_0+\alpha_4e_1+\alpha_5e_2}{z^3}+\cdots+\frac{\alpha_{m-3}e_0+\alpha_{m-2}e_1+\alpha_{m-1}e_2}{z^{m-3} }      \right)\frac{\partial}{\partial z}\wedge \frac{\partial}{\partial \xi}
\end{align*}
Here $a \equiv b$ means that $a $ and $b$ represent the same cohomology class in $H^1(F_m, \wedge^2 \Theta_{F_m})$. 

We represent our computation in the following matrix form with respect to bases $(\ref{r10})$ and $(\ref{r11})$.
\begin{equation*}
\left(\begin{matrix}
e_0 &e_1 & e_2 \\
 & e_0 & e_1 & e_2\\
& & e_0 &e_1 & e_2 &\\
& & & &\cdots & \cdots &\cdots \\
&&&&&&& e_0 & e_1 &e_2\\
\end{matrix}\right)
\left(
\begin{matrix}
\alpha_1\\
\alpha_2\\
\alpha_3\\
\cdots\\
\alpha_{m-1}
\end{matrix}
\right)
=
-\left(
\begin{matrix}
\alpha_1e_0+\alpha_2e_1+\alpha_3e_2\\
\alpha_2e_0+\alpha_3 e_1+\alpha_4 e_2\\
\alpha_3e_0+\alpha_4 e_1+\alpha_5 e_2\\
\cdots\\
\alpha_{m-3} e_0+\alpha_{m-2}e_1+\alpha_{m-1}e_2
\end{matrix}
\right)
\end{equation*}
We note that if $(e_0,e_1,e_2)\ne 0$, $H^1(F_m, \Theta_{F_m})\xrightarrow{[\Lambda_0,-]}H^1(F_m,\wedge^2 \Theta_{F_m})$ is surjective so that $\mathbb{H}^2(F_m,  \Theta_{F_m}^\bullet)=0$. Hence $(F_m, \Lambda_0=(e(z)\xi+f(z)\xi^2)\frac{\partial}{\partial z}\wedge \frac{\partial}{\partial \xi})$ are unobstructed in Poisson deformations when $e(z)\ne 0$ for $m\geq 4$.

On the other hand, assume that $(e_0,e_1,e_2)=0$. Then $H^1(F_m, \Theta_{F_m})\xrightarrow{[\Lambda_0,-]} H^1(F_m, \wedge^2 \Theta_{F_m})$ is zero map so that $\mathbb{H}^2(F_m, \Theta_{F_m}^\bullet)\cong H^1 (F_m, \wedge^2 \Theta_{F_m})$. Hence $\dim_\mathbb{C} \mathbb{H}^2(F_m,\Theta_{F_m}^\bullet)=m-3$ for $m\geq 4$. In this case, we show that $(F_m, \Lambda_0=f(z)\xi^2\frac{\partial}{\partial z}\wedge \frac{\partial}{\partial \xi})$ in $(\ref{r2})$ are obstructed in Poisson deformations for $m\geq 4$. Indeed, choose $\xi \frac{\partial}{\partial z}\wedge \frac{\partial}{\partial \xi}\in H^0(F_m, \wedge^2 \Theta_{F_m})$, and choose $\frac{1}{z}\frac{\partial}{\partial \xi }\in H^1(F_m, \Theta_{F_m})=ker(H^1(F_m, \Theta_{F_m}) \xrightarrow{[\Lambda_0,-]} H^1(F_m, \wedge^2 \Theta_{F_m}))$. Then
\begin{align*}
[\xi\frac{\partial}{\partial z}\wedge \frac{\partial}{\partial \xi}, \frac{1}{z}\frac{\partial}{\partial \xi}]= -\frac{1}{z}\frac{\partial}{\partial z}\wedge \frac{\partial}{\partial \xi}\ne 0 \in H^1(F_m, \wedge^2 \Theta_{F_m}).
\end{align*}
which is not in the image of $H^1(F_m,\Theta_{F_m})\xrightarrow{[\Lambda_0,-]} H^1( F_m, \wedge^2 \Theta_{F_m})$ since $[\Lambda_0=f(z)\xi^2\frac{\partial}{\partial z}\wedge \frac{\partial}{\partial \xi},-]$ is the zero map. Hence by Lemma \ref{r4}, $(F_m, \Lambda_0=f(z)\xi^2\frac{\partial}{\partial z}\wedge \frac{\partial}{\partial \xi})$ are obstructed in Poisson deformations for $m\geq 4$.

We summarize our discussion in Table \ref{ruled}.

\begin{table}
\begin{center}
\begin{tabular}{| c | c | c | c | } \hline
rational ruled surface $F_m$ &  Poisson structure & $\dim_\mathbb{C} \mathbb{H}^2(F_m,\Theta_{F_m}^\bullet)$ & obstructedness in Poisson deformations  \\ \hline
$F_0$ & any Poisson structure& $0$ & unobstructed    \\ \hline
$F_1$ & any Poisson structure& $0$ & unobstructed    \\ \hline
$F_2$ &  any Poisson structure & $0$ & unobstructed       \\ \hline
$F_3$ & any Poisson structure & $0$ &  unobstructed     \\ \hline
$F_m, m\geq 4$ & $e(z)\ne 0$ in $(\ref{r2})$ & $0$ & unobstructed   \\ \hline
$F_m,m\geq 4$ & $e(z)=0$ in $(\ref{r2})$ & $m-3$ & obstructed  \\ \hline
\end{tabular}
\end{center}
\caption{obstructed and unobstructedness of Poisson rational ruled surfaces} \label{ruled}
\end{table}

\subsection{Descriptions of $\mathbb{H}^1(F_m, \Theta_{F_m}^\bullet)$ for unobstructed $(F_m, \Lambda_0)$ in Poisson deformations}\

In this subsection, we explicitly give examples of Poisson analytic families of deformations of $(F_m, \Lambda_0)$ such that the Poisson Kodaira-Spencer map is an isomorphism at the distinguished point. First we need to compute $\mathbb{H}^1(F_m, \Theta_{F_m}^\bullet)$. We note that by Lemma \ref{r4},
\begin{align*}
\mathbb{H}^1(F_m, \Theta_{F_m}^\bullet)&\cong coker(H^0(F_m, \Theta_{F_m})\xrightarrow{[\Lambda_0,-]} H^0(F_m, \wedge^2 \Theta_{F_m}))\oplus ker(H^1(F_m, \Theta_{F_m})\xrightarrow{[\Lambda_0,-]} H^1(F_m,\wedge^2 \Theta_{F_m}))
\end{align*}

We describe $H^0(F_m, \Theta_{F_m}  )  \xrightarrow{[\Lambda_0,-]} H^0(F_m, \wedge^2 \Theta_{F_m}), m\geq 1$ by using $(\ref{r12})$ and $(\ref{r2})$.

{\small{\begin{align*}
&[(d(z)+e(z)\xi+f(z)\xi^2)\frac{\partial}{\partial z}\wedge \frac{\partial}{\partial \xi},g(z)\frac{\partial}{\partial z}+(b(z)\xi+c(z)\xi^2)\frac{\partial}{\partial \xi}]
= (A(z)+B(z)\xi+C(z)\xi^2)\frac{\partial}{\partial z}\wedge \frac{\partial }{\partial \xi}\\
\end{align*}}}
where
\begin{align}\label{nn30}
A(z):&=d(z)g'(z)-g(z)d'(z)+d(z)b(z) \notag \\
B(z):&=e(z)g'(z)-g(z)e'(z)+2d(z)c(z)\\
C(z):&=f(z)g'(z)-g(z)f'(z)+c(z)e(z)-b(z)f(z) \notag
\end{align}

\begin{example}
Let us consider $(F_2,\Lambda_0=0)$. Since $\dim_\mathbb{C} \mathbb{H}^2(F_2, \Theta_{F_2}^\bullet)=0$, it is unobstructed in Poisson deformations. We will describe a Poisson analytic family of deformations of $(F_2,\Lambda_0=0)$ such that the Poisson Kodaira-Spencer map is an isomorphism at the distinguished point. We note that $\dim_\mathbb{C} \mathbb{H}^1(F_2, \Theta_{F_2}^\bullet)=\dim_\mathbb{C}H^0(F_2, \wedge^2 \Theta_{F_2})\oplus H^1(F_2, \Theta_{F_2})=9+1=10$, and from $(\ref{r10})$, $\frac{1}{z}\frac{\partial}{\partial \xi}=z\frac{\partial}{\partial \xi'}$ is a basis of $H^1(F_2, \Theta_{F_2})$. Let us consider a complex analytic family $(\mathcal{F},\mathbb{C}^{10}, \omega)$ of deformations of $F_2=\omega^{-1}(0)$ defined by patching $U_i\times \mathbb{P}_\mathbb{C}^1\times \mathbb{C}^{10}, i=1,2, ( t_1,...,t_{10})\in \mathbb{C}^{10}$
\begin{align*}
(z',\xi')=\left(\frac{1}{z},z^2\xi+t_1z \right)
\end{align*}
where $\omega:\mathcal{F}\to \mathbb{C}^{10}$ is the natural projection. Then the Kodaira-Spencer map is described by $\frac{\partial}{\partial t_1}\mapsto z\frac{\partial}{\partial \xi'}=\frac{1}{z}\frac{\partial}{\partial \xi}$, and $\frac{\partial}{\partial t_i}\mapsto 0$ otherwise. We will put a Poisson structure $\Lambda$ on $\mathcal{F}$ such that $(\mathcal{F},\Lambda,\mathbb{C}^{10},\omega)$ is the desired Poisson analytic family.

From $(\ref{nn3})$, let us consider a holomorphic bivector field on $U_1\times \mathbb{P}_\mathbb{C}^1\times \mathbb{C}^{10}$
\begin{align}\label{i1}
\Pi=(t_2+(t_3+t_4z+t_5z^2)\xi+(t_6+t_7z+t_8z^2+t_9z^3+t_{10}z^4)\xi^2)\frac{\partial}{\partial z}\wedge \frac{\partial}{\partial \xi}
\end{align}

Since $z'=\frac{1}{z}$ and $\xi=\frac{\xi'-t_1 z}{z^2}=z'^2\xi'-t_1 z'$, and $\frac{\partial}{\partial z}\wedge \frac{\partial}{\partial \xi}=-\frac{\partial}{\partial z'}\wedge \frac{\partial}{\partial \xi'}$, $\Pi$ is translated  on $U_2\times \mathbb{P}_\mathbb{C}^1\times \mathbb{C}^{10}$ into
\begin{align}\label{nn10}
\left(t_2+\left(t_3+\frac{t_4}{z'}+\frac{t_5}{z'^2}\right)(z'^2\xi'-t_1z')+\left(t_6+\frac{t_7}{z'}+\frac{t_8}{z'^2}+\frac{t_9}{z'^3}+\frac{t_{10}}{z'^4} \right)( z'^4\xi'^2-2t_1z'^3 \xi'+t_1^2z'^2)\right)\left(-\frac{\partial}{\partial z'}\wedge \frac{\partial}{\partial \xi'}\right)
\end{align}
which is not holomorphic. We will modify $(\ref{i1})$ to define a global bivector field on $\mathcal{F}$. Consider the rational part of $(\ref{nn10})$
\begin{align*}
\left(-\frac{t_1t_5}{z'}+\frac{t_1^2t_9}{z'}-2\frac{t_1t_{10}\xi'}{z'}+\frac{t_1^2t_{10}}{z'^2}\right)\left(-\frac{\partial}{\partial z'}\wedge \frac{\partial}{\partial \xi'}\right)
\end{align*}
Then $\Lambda:=F(z,\xi,t)\frac{\partial}{\partial z}\wedge \frac{\partial}{\partial \xi}$ defines the Poisson structure on $\mathcal{F}$, where $F(z,\xi,t)$ is defined by
\begin{align*}
t_2+(t_3+t_4z+t_5z^2)\xi+(t_6+t_7z+t_8z^2+t_9z^3+t_{10}z^4)\xi^2+t_1t_5z-t_1^2t_9z+2t_1t_{10}z(z^2\xi+t_1z)-t_1^2t_{10}z^2
\end{align*}
and $(\mathcal{F},\Lambda,\mathbb{C}^{10},\omega)$ is a Poisson analytic family of deformations of $(F_2,\Lambda_0=0)=\omega^{-1}(0)$ such that the Poisson Kodaria-Spencer map is an isomorphism at $t=0$.
\end{example}

\begin{example}
Let us consider $(F_3,\Lambda_0=0)$. Since $\dim_\mathbb{C} \mathbb{H}^2(F_3,\Theta_{F_3}^\bullet)=0$, it is unobstructed in Poisson deformations. We will explicitly construct a Poisson analytic family of deformations of $(F_3, \Lambda_0=0)$ such that the Poisson Kodaira-Spencer map is an isomorphism at the distinguished point. We note that $\dim_\mathbb{C} \mathbb{H}^1(F_3, \Theta_{F_3}^\bullet)=\dim_\mathbb{C}H^0(F_3, \wedge^2 \Theta_{F_3})\oplus H^1(F_3, \Theta_{F_3})=9+2=11$, and from $(\ref{r10})$, $\frac{1}{z}\frac{\partial}{\partial \xi}=z^2\frac{\partial}{\partial \xi'}, \frac{1}{z^2}\frac{\partial}{\partial \xi}=z\frac{\partial}{\partial \xi'}$ is a basis of $H^1(F_2, \Theta_{F_2})$.
Let us consider a complex analytic family $(\mathcal{F},\mathbb{C}^{11},\omega)$ of deformations of $F_3=\omega^{-1}(0)$ defined by patching $U_i\times \mathbb{P}_\mathbb{C}^1\times \mathbb{C}^{11}, i=1,2, (t_1,...,t_{11})\in \mathbb{C}^{11}$ by the relation 
\begin{align*}
(z',\xi')=\left(\frac{1}{z}, z^3\xi+t_1z+t_2z^2\right)
\end{align*}
where $\omega:\mathcal{F}\to \mathbb{C}^{11}$ is the natural projection. Then the Kodaira-Spencer map is described by $\frac{\partial}{\partial t_1}\mapsto z\frac{\partial}{\partial \xi'}=\frac{1}{z^2}\frac{\partial}{\partial \xi}$ and $\frac{\partial}{\partial t_2}\mapsto z^2\frac{\partial}{\partial \xi'}=\frac{1}{z}\frac{\partial}{\partial \xi}$.
We will put a Poisson structure $\Lambda$ on $\mathcal{F}$ such that $(\mathcal{F},\Lambda,\mathbb{C}^{11},\omega)$ is the desired Poisson analytic family. 

From $(\ref{nn4})$, let us consider a holomorphic bivector field on $U_1\times \mathbb{P}_\mathbb{C}^1\times \mathbb{C}^9$
\begin{align}\label{i2}
\Pi=((t_3+t_4z+t_5z^2)\xi+(t_6+t_7z+t_8z^2+t_9z^3+t_{10}z^4+t_{11}z^5)\xi^2)\frac{\partial }{\partial z}\wedge \frac{\partial}{\partial \xi}
\end{align}
Since $z'=\frac{1}{z}$ and $\xi=\frac{\xi'-t_1z-t_2z^2}{z^3}=\xi'z'^3-t_1z'^2-t_2z'$, $\frac{\partial}{\partial z}\wedge \frac{\partial}{\partial \xi}=-\frac{1}{z'}\frac{\partial}{\partial z'}\wedge \frac{\partial}{\partial \xi'}$, and
\begin{align*}
(z'^3\xi'-t_1z'^2-t_2z')^2=z'^6\xi'^2+t_1^2z'^4+t_2^2z'^2-2t_1z'^5\xi'-2t_2z'^4 \xi'+2t_1t_2z'^3,
\end{align*}
$\Pi$ is translated on $U_2\times \mathbb{P}_\mathbb{C}^1\times \mathbb{C}^{11}$ into
\begin{align}\label{nn20}
&\left(\left(t_3+\frac{t_4}{z'}+\frac{t_5}{z'^2}\right)(z'^2\xi'-t_1z'-t_2)\right)\left(-\frac{\partial}{\partial z'}\wedge \frac{\partial}{\partial \xi'}\right)
\\
&+\left(\left(t_6+\frac{t_7}{z'}+\frac{t_8}{z'^2}+\frac{t_9}{z'^3}+\frac{t_{10}}{z'^4}+\frac{t_{11}}{z'^5}\right)(z'^5\xi'^2+t_1^2z'^3+t_2^2z'-2t_1z'^4\xi'-2t_2z'^3 \xi'+2t_1t_2z'^2)\right)\left(-\frac{\partial}{\partial z'}\wedge \frac{\partial}{\partial \xi'}\right)\notag
\end{align}
which is not holomorphic. We will modify $(\ref{i2})$ to define a global bivector field on $\mathcal{F}$. Let us consider the rational part of $(\ref{nn20})$.
{\small{\begin{align*}
&\left(-\frac{t_2 t_4}{z'}-\frac{t_1t_5}{z'}-\frac{t_2t_5}{z'^2}\right)\left(-\frac{\partial}{\partial z'}\wedge \frac{\partial}{\partial \xi'}\right)
\\
&+\left(\frac{t_8t_2^2}{z'}+\frac{t_9t_2^2}{z'^2}+2\frac{t_9t_1t_2}{z'}+\frac{t_{10}t_1^2}{z'}+\frac{t_{10}t_2^2}{z'^3}-2\frac{t_{10}t_2\xi'}{z'}+2\frac{t_{10}t_1t_2}{z'^2}+\frac{t_{11}t_1^2}{z'^2}+\frac{t_{11}t_2^2}{z'^4}-2\frac{t_{11}t_1\xi'}{z'}-2\frac{t_{11}t_2\xi'}{z'^2}+2\frac{t_{11}t_1t_2}{z'^3}\right)\left(-\frac{\partial}{\partial z'}\wedge \frac{\partial}{\partial w'}\right)
\end{align*}}}
Then $\Lambda:=F(z,\xi,t)\frac{\partial }{\partial z}\wedge \frac{\partial}{\partial \xi}$ defines the Poisson structure on $\mathcal{F}$, where
\begin{align*}
&F(z,\xi,t)=(t_3+t_4z+t_5z^2)\xi+(t_6+t_7z+t_8z^2+t_9z^3+t_{10}z^4+t_{11}z^5)\xi^2\\
&-(-t_2t_4-t_1t_5+t_8t_2^2+2t_1t_2t_9+t_1^2t_{10})-(-t_2t_5+t_9t_2^2+2 t_1t_2t_{10}+t_1^2t_{11} )z -( t_{10}t_2^2+2t_1t_2t_{11})z^2+t_{11}t_2^2z^3\\
&+2(t_2t_{10}+t_1t_{11}+t_2t_{11}z)(z^2\xi+t_1z+t_2z^2)
\end{align*}
and $(\mathcal{F},\Lambda, \mathbb{C}^{10},\omega)$ is a Poisson analytic family of deformations of $(F_3, \Lambda_0=0)=\omega^{-1}(0)$ such that the Poisson Kodaira-Spencer map is an isomorphism at $t=0$.
\end{example}

\begin{example}
Let us consider $(F_4, \Lambda_0=(z\xi+z\xi^2)\frac{\partial}{\partial z}\wedge \frac{\partial}{\partial \xi})$. Since $e(z)=z\ne 0 $ in $(\ref{r2})$ so that $\dim_\mathbb{C} \mathbb{H}^2(F_4,\Theta_{F_4}^\bullet)=0$  by Table $\ref{ruled}$, it is unobstructed in Poisson deformations. We will explicitly  construct a Poisson analytic family of deformations of $(F_4,\Lambda_0)$ such that the Poisson Kodaira-Spencer map is an isomorphism at the distinguished point. We describe $\mathbb{H}^1(\mathbb{F}_4,\Theta_{F_4}^\bullet)$.  Let us find $ker(H^1(F_4,\Theta_{F_4} )\xrightarrow{[\Lambda_0,-]} H^1(F_4, \wedge^2 \Theta_{F_4}))$. From $(\ref{r10})$, since
\begin{align*}
[(z\xi+z\xi^2)\frac{\partial}{\partial z}\wedge \frac{\partial}{\partial \xi},\left(\frac{b_1}{z}+\frac{b_2}{z^2}+\frac{b_3}{z^3}\right)\frac{\partial}{\partial \xi}]\equiv -\left(\frac{b_1}{z}+\frac{b_2}{z^2}+\frac{b_3}{z^3}\right)(z+2z\xi)\frac{\partial}{\partial z}\wedge \frac{\partial}{\partial \xi}\equiv -\frac{ b_2}{z}\frac{\partial}{\partial z}\wedge \frac{\partial}{\partial \xi}
\end{align*}
so that $\frac{1}{z}\frac{\partial}{\partial \xi}= z^3\frac{\partial}{\partial \xi'}$ and $\frac{ 1}{z^3}\frac{\partial}{\partial \xi}=z\frac{\partial}{\partial \xi'}$ is a basis of $ker(H^1(F_4,\Theta_{F_4} ) \xrightarrow{[\Lambda_0,-]} H^1(F_4, \wedge^2 \Theta_{F_4}))$. Let us find $coker(H^0(F_4,\Theta_{F_4} ) \xrightarrow{[\Lambda_0,-]} H^0(F_4, \wedge^2 \Theta_{F_4}))$. Since $d(z)=0,e(z)=z,f(z)=z$ in $(\ref{nn30})$, and we have $(\ref{r12})$,
\begin{align*}
&(e(z)g'(z)-g(z)e'(z))\xi+(f(z)g'(z)-g(z)f'(z)+c(z)e(z)-b(z)f(z))\xi^2\\
&=(z(g_1+2g_2z)-(g_0+g_1z+g_2z^2))\xi+(z(g_1+2g_2z)-(g_0+g_1z+g_2z^2)+z(c_0+c_1z+\cdots+c_4z^4)-z(-4g_2z+d))\xi^2\\
&=(-g_0+g_2z^2)\xi+(-g_0+(c_0-d)z+(5g_2+c_1)z^2+c_2z^3+c_3z^4+c_4z^5)\xi^2
\end{align*}
so that $\xi\frac{\partial}{\partial z}\wedge \frac{\partial}{\partial \xi}, z\xi\frac{\partial}{\partial z}\wedge \frac{\partial}{\partial \xi}, z^6\xi^2\frac{\partial}{\partial z}\wedge \frac{\partial}{\partial \xi}$ is a basis of $coker(H^0(F_4,\Theta_{F_4} )\xrightarrow{[\Lambda_0,-]} H^2(F_4, \wedge^2 \Theta_{F_4}))$. Hence $\mathbb{H}^1(F_4,\Theta_{F_4}^\bullet)\cong coker(H^0(F_4,\Theta_{F_4} )\xrightarrow{[\Lambda_0,-]} H^2(F_4, \wedge^2 \Theta_{F_4}))\oplus ker(H^1(F_4,\Theta_{F_4} ) \xrightarrow{[\Lambda_0,-]} H^1(F_4, \wedge^2 \Theta_{F_4}))$ is generated by
\begin{align}\label{nn50}
\mathbb{H}^1(F_4,\Theta_{F_4}^\bullet)\cong \left\langle \xi\frac{\partial}{\partial z}\wedge \frac{\partial}{\partial \xi}, z\xi\frac{\partial}{\partial z}\wedge \frac{\partial}{\partial \xi}, z^6\xi^2\frac{\partial}{\partial z}\wedge \frac{\partial}{\partial \xi} \right\rangle \oplus \left \langle z\frac{\partial}{\partial \xi'},z^3\frac{\partial}{\partial \xi'}  \right \rangle
\end{align}
We will construct a Poisson analytic family such that the Poisson Kodaira-Spencer map is an isomorphism at the distinguished point by using  the basis $(\ref{nn50})$ which forms a linear part. From $(\ref{nn50})$, let us consider a complex analytic family defined by patching $U_i\times \mathbb{P}_\mathbb{C}^1\times \mathbb{C}^5,i=1,2,(t_1,...,t_5)\in \mathbb{C}^5$ by the relation
\begin{align}\label{nn53}
(z',\xi')=\left(\frac{1}{z}, z^4\xi+t_1z+t_2z^3    \right)
\end{align}
We note that since $\frac{\partial}{\partial z}\wedge \frac{\partial}{\partial w}=-\frac{1}{z'^2}\frac{\partial}{\partial z'}\wedge \frac{\partial}{\partial w'}$,
\begin{align*}
&[(z\xi+z\xi^2)\frac{\partial}{\partial z}\wedge \frac{\partial}{\partial \xi}, \frac{1}{z}\frac{\partial}{\partial \xi}]=  -\left(    1+2\xi  \right) \frac{\partial}{\partial z}\wedge \frac{\partial}{\partial \xi}\\
&[(z\xi+z\xi^2)\frac{\partial}{\partial z}\wedge \frac{\partial}{\partial \xi}, \frac{1}{z^3}\frac{\partial}{\partial \xi}]=-\frac{1}{z^2}\left(1+2\xi    \right) \frac{\partial}{\partial z}\wedge \frac{\partial}{\partial \xi}=(1+2z'^4\xi')\frac{\partial}{\partial z'}\wedge \frac{\partial}{\partial \xi'}
\end{align*}
Then $(((1+2\xi)\frac{\partial}{\partial z}\wedge \frac{\partial}{\partial \xi},0), \frac{1}{z}\frac{\partial}{\partial \xi})$ and $( (0,(1+2z'^4\xi'))\frac{\partial}{\partial z'}\wedge \frac{\partial}{\partial \xi'}), \frac{1}{z^3}\frac{\partial}{\partial \xi})\in C^0(\mathcal{U},\wedge^2 \Theta_{F_4})\oplus C^1(\mathcal{U}, \Theta_{F_4})$ defines element in $\mathbb{H}^1(F_4,\Theta_{F_4}^\bullet)$. Let us consider a holomorphic bivector field on $U_1\times \mathbb{P}_\mathbb{C}^1\times \mathbb{C}^5$
\begin{align}\label{nn52}
\Pi=(t_2+ (2t_2+t_3 +z +t_4z)\xi+( z+t_5z^6)\xi^2)\frac{\partial}{\partial z}\wedge \frac{\partial }{\partial \xi}
\end{align}
Since $\xi=z'^4\xi'-t_1z'^3-t_2z'$,  and $(z'^4\xi'-t_1z'^3-t_2z')^2=z'^8\xi'^2+t_1^2z'^6+t_2^2z'^2-2t_1z'^7\xi'-2t_2z'^5\xi'+2t_1t_2z'^4$, $\Pi$ is translated on $U_2\times \mathbb{P}_\mathbb{C}^1\times \mathbb{C}^5$ into
\begin{align}\label{nn51}
&\left(\frac{t_2}{z'^2}+\left(2t_2+t_3+\frac{1}{z'}+\frac{t_4}{z'}\right)(z'^2\xi'-t_1z'-\frac{t_2}{z'})\right)\left( - \frac{\partial}{\partial z'}\wedge \frac{\partial }{\partial \xi'}  \right) \\
&+\left(\frac{1}{z'}+\frac{t_5}{z'^6}\right)(z'^6\xi'^2+t_1^2z'^4+t_2^2-2t_1z'^5\xi'-2t_2z'^3\xi'+2t_1t_2z'^2)\left( - \frac{\partial}{\partial z'}\wedge \frac{\partial }{\partial \xi'}  \right)\notag
\end{align}
which is not holomorphic. We note that for $k\geq 2$, $\frac{1}{z'^k}\frac{\partial}{\partial z'}\wedge \frac{\partial}{\partial \xi'}$ is translated into $z^{k-2}\frac{\partial}{\partial z}\wedge \frac{\partial}{\partial \xi}$ which is holomorphic. So we consider the $\frac{1}{z'}$ part of $(\ref{nn51})$ which is given by
\begin{align*}
\left(-2\frac{t_2^2}{z'}-\frac{t_2t_3}{z'}\right)\left( - \frac{\partial}{\partial z'}\wedge \frac{\partial }{\partial \xi'}  \right)+\frac{t_2^2}{z'}\left( - \frac{\partial}{\partial z'}\wedge \frac{\partial }{\partial \xi'}  \right)
\end{align*}

We will modify $(\ref{nn53})$ and $(\ref{nn52})$ so that $\frac{1}{z'}$ part does not appear on $U_2\times \mathbb{P}_\mathbb{C}^1\times \mathbb{C}^5$. Let us consider a complex analytic family $(\mathcal{F},\mathbb{C}^5,\omega)$ defined by patching $U_i\times \mathbb{P}_\mathbb{C}^1\times \mathbb{C}^5, i=1,2$ by the relation
\begin{align*}
(z',\xi') =\left(\frac{1}{z}, z^4\xi+t_1z+t_2z^3-(t_2^2+t_2t_3)z^2   \right)
\end{align*}
where $\omega:\mathcal{F}\to \mathbb{C}^5$ is the natural projection, and consider a holomorphic bivector field on $U_1\times \mathbb{P}_\mathbb{C}^1\times \mathbb{C}^5$
\begin{align*}
\Pi'=(t_2+ (2t_2+t_3 +z +t_4z+t_3t_4+t_2t_4)\xi+( z+t_5z^6)\xi^2)\frac{\partial}{\partial z}\wedge \frac{\partial }{\partial \xi}
\end{align*}

Since $\xi=z'^4\xi'-t_1z'^3-t_2z'+(t_2^2+t_2t_3)z'^2$, and
\begin{align*}
&(z'^4\xi'-t_1z'^3-t_2z'+(t_2^2+t_2t_3)z'^2)^2\\
&=z'^8\xi'^2+t_1^2z'^6+t_2^2z'^2+(t_2^2+t_2t_3)^2z'^4-2t_1z'^7\xi'-2t_2z'^5\xi'+(t_2^2+t_2t_3)z'^6\xi'+2t_1t_2z'^4\\
&-2t_1(t_2^2+t_2t_3)z'^5-2t_2(t_2^2+t_2t_3)z'^3
\end{align*}
$\Pi'$ is translated on $U_2\times \mathbb{P}_\mathbb{C}^1\times \mathbb{C}^5$ into
{\small{\begin{align}\label{i3}
&\left(\frac{t_2}{z'^2}+\left(2t_2+t_3+\frac{1}{z'}+\frac{t_4}{z'}+t_3t_4 +t_2t_4\right)\left(z'^2\xi'-t_1z'-\frac{t_2}{z'}+t_2^2+t_2t_3\right)\right)\left( - \frac{\partial}{\partial z'}\wedge \frac{\partial }{\partial \xi'}  \right)\\
&+\left(\frac{1}{z'}+\frac{t_5}{z'^6}\right)(z'^6\xi'^2+t_1^2z'^4+t_2^2+(t_2^2+t_2t_3)^2z'^2-2t_1z'^5\xi'-2t_2z'^3\xi'+(t_2^2+t_2t_3)z'^4\xi'+2t_1t_2z'^2 \notag\\
&-2t_1(t_2^2+t_2t_3)z'^3-2t_2(t_2^2+t_2t_3)z'
)\left( - \frac{\partial}{\partial z'}\wedge \frac{\partial }{\partial \xi'}  \right) \notag
\end{align}}}
which is not holomorphic. Consider the rational part of $(\ref{i3})$
\begin{align*}
&\left(\frac{t_2}{z'^2}-2\frac{t_2^2}{z'}-\frac{t_2t_3}{z'}-\frac{t_2}{z'^2}+\frac{t_2^2}{z'} +\frac{t_2t_3}{z'}-\frac{t_2t_4}{z'^2}+\frac{t_2^2t_4}{z'}+\frac{t_2t_3t_4}{z'}-\frac{t_2t_3t_4}{z'} -\frac{t_2^2t_4}{z'}   \right)\left( - \frac{\partial}{\partial z'}\wedge \frac{\partial }{\partial \xi'}  \right)\\
&+(\frac{t_2^2}{z'}+\frac{t_1^2t_5}{z'^2}+\frac{t_2^2t_5}{z'^6}+\frac{t_5(t_2^2+t_2t_3)^2}{z'^4}-2\frac{t_1t_5\xi'}{z'}-2\frac{t_2t_5\xi'}{z'^3}+\frac{t_5(t_2^2+t_2t_3)\xi'}{z'^2} +2\frac{t_1t_2t_5}{z'^4}\\
&-2\frac{t_1t_5(t_2^2+t_2t_3)}{z'^3}-2\frac{t_2t_5(t_2^2+t_2t_3)}{z'^5})\left( - \frac{\partial}{\partial z'}\wedge \frac{\partial }{\partial \xi'}  \right)
\end{align*}
Then the Poisson structure on $U_1\times \mathbb{P}_\mathbb{C}^1\times \mathbb{C}^5$
\begin{align*}
\Lambda:=A(z,\xi,t)\frac{\partial}{\partial z}\wedge \frac{\partial}{\partial \xi}
\end{align*}
where
\begin{align*}
&A(z,\xi,t)\\
&=t_2+ (2t_2+t_3 +z +t_4z+t_3t_4+t_2t_4)\xi+( z+t_5z^6)\xi^2-[-t_2t_4+t_1^2t_5+t_2^2t_5z^4+t_5(t_2^2+t_2t_3)^2z^2\\
&-2t_1t_5(z^3\xi+t_1+t_2z^2-(t_2^2+t_2t_3)z)-(2t_2t_5z-t_5(t_2^2+t_2t_3))(z^4\xi+t_1z+t_2z^3-(t_2^2+t_2t_3)z^2)\\
&+2t_1t_2t_5z^2-2t_1t_5(t_2^2+t_2t_3)z-2t_2t_5(t_2^2+t_2t_3)z^3]
\end{align*}
define a global bivector field on $\mathcal{F}$ so that $(\mathcal{F},\Lambda, \mathbb{C}^5 ,\omega )$ is a Poisson analytic family of deformations of $(F_4,\Lambda_0)=\omega^{-1}(0)$ such that the Poisson Kodaira-Spencer map is an isomorphism at $t=0$.
\end{example}

\begin{example}
Let us consider $(F_5, \Lambda_0=z\xi\frac{\partial}{\partial z}\wedge \frac{\partial}{\partial \xi})$.  Since $e(z)=z\ne 0$ in $(\ref{r2})$ so that $\dim_\mathbb{C} \mathbb{H}^2(F_5,\Theta_{F_5}^\bullet)=0$ by Table $\ref{ruled}$, it is unobstructed in Poisson deformations. We find $\mathbb{H}^1(F_5, \Theta_{F_5}^\bullet)$. First we find $ker(H^1(F_5,\Theta_{F_5})\xrightarrow{[\Lambda_0,-]} H^1(F_5,\wedge^2 \Theta_{F_5}))$. We note that from $(\ref{r10})$,
\begin{align*}
[z\xi\frac{\partial}{\partial z}\wedge \frac{\partial}{\partial \xi}, \left(\frac{b_1}{z}+\frac{b_2}{z^2}+\frac{b_3}{z^3}+\frac{b_4}{z^4}\right)\frac{\partial}{\partial \xi}]&=-\left(b_1+\frac{b_2}{z}+\frac{b_3}{z^2}+\frac{b_4}{z^3}\right)\frac{\partial}{\partial z}\wedge \frac{\partial}{\partial \xi} \equiv -\left(\frac{b_2}{z}+\frac{b_3}{z^2}\right)\frac{\partial}{\partial z}\wedge \frac{\partial}{\partial \xi}
\end{align*}

Hence $ker(H^1(F_5,\Theta_{F_5})\xrightarrow{[\Lambda_0,-]} H^1(F_5,\wedge^2 \Theta_{F_5}))$ is generated by $\frac{1}{z}\frac{\partial}{\partial \xi}=z^4\frac{\partial}{\partial \xi ' },\frac{1}{z^4}\frac{\partial}{\partial \xi}=z\frac{\partial}{\partial \xi'}$.
Let us find $coker(H^0(F_5, \Theta_{F_5})\xrightarrow{[\Lambda_0,-]} H^0(F_5,\wedge^2 \Theta_{F_5}))$. Since $d(z)=0, e(z)=z,f(z)=0$ in $(\ref{nn30})$, and we have $(\ref{r12})$, 
\begin{align*}
(zg'(z)-g(z))\xi+zc(z)\xi^2&=(z(g_1+2g_2z)-g_0-g_1z-g_2z^2)\xi+z(c_0+c_1z+c_2z^2+c_3z^3+c_4z^4+c_5z^5)\xi^2\\
&=(-g_0+g_2z^2)\xi+(c_0z+c_1z^2+c_2z^3+c_3z^4+c_4z^5+c_5z^6)\xi^2
\end{align*}
so that $coker(H^0(F_5, \Theta_{F_5})\xrightarrow{[\Lambda_0,-]} H^0(F_5,\wedge^2 \Theta_{F_5}))$ is generated by $z\xi\frac{\partial}{\partial z}\wedge \frac{\partial}{\partial \xi}, \xi^2\frac{\partial}{\partial z}\wedge \frac{\partial}{\partial \xi}, z^7\xi^2\frac{\partial}{\partial z}\wedge \frac{\partial}{\partial \xi}$. Hence $\mathbb{H}^1(F_5, \Theta_{F_5}^\bullet)\cong coker(H^0(F_5, \Theta_{F_5}\xrightarrow{[\Lambda_0,-]} H^0(F_5,\wedge^2 \Theta_{F_5}))\oplus ker(H^1(F_5,\Theta_{F_5})\xrightarrow{[\Lambda_0,-]} H^1(F_5,\wedge^2 \Theta_{F_5}))  $ is generated by
\begin{align}\label{r53}
\mathbb{H}^1(F_5,\Theta_{F_5}^\bullet)\cong \left\langle z\xi\frac{\partial}{\partial z}\wedge \frac{\partial}{\partial \xi}, \xi^2\frac{\partial}{\partial z}\wedge \frac{\partial}{\partial \xi}, z^7\xi^2\frac{\partial}{\partial z}\wedge \frac{\partial}{\partial \xi}  \right\rangle \oplus \left \langle z^4\frac{\partial}{\partial \xi'}, z\frac{\partial}{\partial \xi'} \right \rangle
\end{align}
We will construct a Poisson analytic family  such that  Poisson Kodaira-Spencer map is an isomorphism at the distinguished point by using the basis $(\ref{r53})$ which forms a linear part. From $(\ref{r53})$, let us consider a complex analytic family defined by patching $U_i\times \mathbb{P}_\mathbb{C}^1\times \mathbb{C}^5,i=1,2, (t_1,t_2,t_3,t_4,t_5)\in \mathbb{C}^5$ by the relation
\begin{align}\label{r50}
(z',\xi')=\left(\frac{1}{z},z^5\xi+t_1z+t_2z^4\right)
\end{align}
We note that since $\frac{\partial}{\partial z}\wedge \frac{\partial}{\partial w}=-\frac{1}{z'^3}\frac{\partial}{\partial z'}\wedge \frac{\partial}{\partial w'}$,
\begin{align*}
[z\xi\frac{\partial}{\partial z}\wedge \frac{\partial}{\partial \xi},\frac{1}{z}\frac{\partial}{\partial \xi}]&=-\frac{\partial}{\partial z}\wedge \frac{\partial}{\partial \xi}\\
[z\xi\frac{\partial}{\partial z}\wedge \frac{\partial}{\partial \xi},\frac{1}{z^4}\frac{\partial}{\partial \xi}]&=-\frac{1}{z^3}\frac{\partial}{\partial z}\wedge \frac{\partial}{\partial \xi}=\frac{\partial}{\partial z'}\wedge \frac{\partial}{\partial \xi'},
\end{align*}
Then $((\frac{\partial}{\partial z}\wedge \frac{\partial}{\partial \xi},0), \frac{1}{z}\frac{\partial}{\partial \xi})$ and $((0,\frac{\partial}{\partial z'}\wedge \frac{\partial}{\partial \xi'}), \frac{1}{z^4}\frac{\partial}{\partial \xi})\in C^0(\mathcal{U}, \wedge^2 \Theta_{F_5})\oplus C^1(\mathcal{U}, \Theta_{F_5})$ define elements in $\mathbb{H}^1(F_5,\Theta_{F_5}^\bullet)$.

Let us consider a bivector field on $U_1\times \mathbb{P}_\mathbb{C}^1\times \mathbb{C}^5$
\begin{align}\label{r51}
\Pi=(t_2+(z+t_3z)\xi+(t_4+t_5z^7)\xi^2 ) \frac{\partial}{\partial z}\wedge \frac{\partial}{\partial \xi}
\end{align}

We note that $\xi=\frac{\xi'-t_1z-t_2z^4}{z^5}=z'^5\xi'-t_1z'^4-t_2z'$, and $(z'^5\xi'-t_1z'^4-t_2z')^2=z'^{10}\xi'^2+t_1^2z'^8+t_2^2z'^2-2t_1z'^9\xi'-2t_2z'^6\xi'+2t_1t_2z'^5$ so that $\Pi$ is translated on $U_2\times \mathbb{P}_\mathbb{C}^1\times \mathbb{C}^5$ into
\begin{align}
&\left(\frac{t_2}{z'^3}+\left(\frac{1}{z'}+\frac{t_3}{z'}\right)\left(z'^2\xi'-t_1 z' -\frac{t_2}{z'^2}\right)\right)\left(-\frac{\partial}{\partial z'}\wedge \frac{\partial}{\partial \xi'}\right) \label{r52}\\
&+\left(t_4+\frac{t_5}{z'^7}\right)(z'^{7}\xi'^2+t_1^2z'^5+\frac{t_2^2}{z'}-2t_1z'^6\xi'-2t_2z'^3\xi'+2t_1t_2z'^2)\left(-\frac{\partial}{\partial z'}\wedge \frac{\partial}{\partial \xi'}\right) \notag
\end{align}
which is not holomorphic. Our purpose is to modify $(\ref{r50})$ and $(\ref{r51})$ to delete rational parts to make a global holomorphic bivector field both on $U_i\times \mathbb{P}_\mathbb{C}^1\times \mathbb{C}^5,i=1,2$. We note that $\frac{1}{z'^k}\frac{\partial}{\partial z'}\wedge \frac{\partial}{\partial \xi'},k\geq 3$ on $U_2\times \mathbb{P}^1\times \mathbb{C}^2$ is translated into $-z^{k-3}\frac{\partial}{\partial z}\wedge \frac{\partial}{\partial \xi}$ which is holomorphic. So we consider $\frac{1}{z'}$ and $\frac{1}{z'^2}$ parts in $(\ref{r52})$ which is given by
\begin{align*}
\left(\frac{t_2^2t_4}{z'}+\frac{t_1^2t_5}{z'^2}-2\frac{t_1t_5\xi'}{z'}\right)\left(-\frac{\partial}{\partial z'}\wedge \frac{\partial}{\partial \xi'}\right)
\end{align*}
We note that $2\frac{t_1t_5\xi'}{z'}\frac{\partial}{\partial z'}\wedge \frac{\partial}{\partial \xi'}$ is translated into $-2t_1t_5(z^3\xi+\frac{t_1}{z}+t_2z^2)\frac{\partial}{\partial z}\wedge \frac{\partial}{\partial \xi}$ by $(\ref{r50})$ which has the rational part as $-2\frac{t_1^2t_5}{z}\frac{\partial}{\partial z}\wedge \frac{\partial}{\partial \xi}$ which is translated into $2\frac{t_1^2t_5}{z'^2}\frac{\partial}{\partial z'}\wedge \frac{\partial}{\partial \xi'} $.

Let us consider a complex analytic family $(\mathcal{F},\mathbb{C}^5, \omega)$ of deformations of $F_5=\omega^{-1}(0)$  defined by patching $U_i\times \mathbb{P}_\mathbb{C}^1\times \mathbb{C}^5,i=1,2$ by the relation
\begin{align*}
(z',\xi')=\left(\frac{1}{z},z^5\xi+t_1z+t_2z^4+t_2^2t_4z^2-t_1^2t_5z^3\right)
\end{align*}
where $\omega:\mathcal{F}\to \mathbb{C}^5$ is the natural projection, and consider
\begin{align*}
\Pi'=(t_2+(z+t_3z)\xi+(t_4+t_5z^7+t_3t_4+t_3t_5z^7)\xi^2 ) \frac{\partial}{\partial z}\wedge \frac{\partial}{\partial \xi}
\end{align*}

Since $\xi=\frac{\xi'-t_1z-t_2z^4-t_2^2t_4z^2+t_1^2t_5z^3 }{z^5}=z'^5\xi'-t_1z'^4-t_2z'-t_2^2t_4z'^3+t_1^2t_5z'^2$, and
\begin{align*}
&(z'^5\xi'-t_1z'^4-t_2z'-t_2^2t_4z'^3+t_1^2t_5z'^2)^2\\
&=z'^{10}\xi'^2+t_1^2z'^8+t_2^2z'^2-2t_1z'^9\xi'-2t_2z'^6\xi'+2t_1t_2z'^5+t_2^4t_4^2z'^6+t_1^4t_5^2z'^4-2t_2^2t_4z'^8\xi'\\
&+2t_1^2t_5z'^7\xi'+2t_1t_2^2t_4z'^7-2t_1^3t_5z'^6+2t_2^3t_4z'^4-2t_1^2t_2t_5z'^3-2t_1^2t_5t_2^2t_4z'^5
\end{align*}
$\Pi'$ is translated into $U_2\times \mathbb{P}_\mathbb{C}^1\times \mathbb{C}^5$ as
\begin{align}\label{r55}
&\left(\frac{t_2}{z'^3}+\left(\frac{1}{z'}+\frac{t_3}{z'}\right)\left(z'^2\xi'-t_1z'-\frac{t_2}{z'^2}-t_2^2t_4+\frac{t_1^2t_5}{z'}\right)\right) \left(-\frac{\partial}{\partial z'}\wedge \frac{\partial}{\partial \xi'}\right)\\
&+\left(t_4+\frac{t_5}{z'^7}+t_3t_4+\frac{t_3t_5}{z'^7}\right)(z'^{7}\xi'^2+t_1^2z'^5+\frac{t_2^2}{z'}-2t_1z'^6\xi'-2t_2z'^3\xi'+2t_1t_2z'^2+t_2^4t_4^2z'^3+t_1^4t_5^2z'-2t_2^2t_4z'^5\xi' \notag\\
&+2t_1^2t_5z'^4\xi'+2t_1t_2^2t_4z'^4-2t_1^3t_5z'^3+2t_2^3t_4z'-2t_1^2t_2t_5-2t_1^2t_5t_2^2t_4z'^2)\left(-\frac{\partial}{\partial z'}\wedge \frac{\partial}{\partial \xi'}\right)\notag
\end{align}
Then $\frac{1}{z'}$ and $\frac{1}{z'^2}$ parts are
{\small{\begin{align*}
&\left(-\frac{t_2^2t_4}{z'}+\frac{t_1^2t_5}{z'^2}-\frac{t_2^2t_3t_4}{z'}+\frac{t_1^2t_3t_5}{z'^2}\right)\left(-\frac{\partial}{\partial z'}\wedge \frac{\partial}{\partial \xi'}\right)+\left(\frac{t_4t_2^2}{z'}+\frac{t_1^2t_5}{z'^2}-2\frac{t_1t_5\xi'}{z'}+\frac{t_3t_4t_2^2}{z'}+\frac{t_3t_5t_1^2}{z'^2}-2\frac{t_1t_3t_5\xi'}{z'}\right)\left(-\frac{\partial}{\partial z'}\wedge \frac{\partial}{\partial \xi'}\right)\\
&=\left(2\frac{t_1^2t_5}{z'^2}-2\frac{t_1t_5\xi'}{z'}+2\frac{t_1^2t_3t_5}{z'^2}-2\frac{t_1t_3t_5\xi'}{z'} \right) \left(-\frac{\partial}{\partial z'}\wedge \frac{\partial}{\partial w'}\right)
\end{align*}}}
We note that $\left(2\frac{t_1^2 t_5}{z'^2}-2\frac{t_1t_5\xi'}{z'}\right)\left(-\frac{\partial}{\partial z'}\wedge \frac{\partial}{\partial \xi'}\right)$ is translated on $U_1\times \mathbb{P}_\mathbb{C}^1\times \mathbb{C}^5$ into
\begin{align*}
\left(2\frac{t_1^2t_5}{z}-2\frac{t_1t_5(z^5\xi+t_1z+t_2z^4+t_2^2t_4z^2-t_1^2t_5 z^3)}{z^2}\right)\frac{\partial}{\partial z}\wedge \frac{\partial}{\partial \xi}=-2t_1t_5(z^3\xi+t_2z^2+t_2^2t_4-t_1^2t_5z)\frac{\partial}{\partial z}\wedge \frac{\partial}{\partial \xi}
\end{align*}
and $\left(2\frac{t_1^2t_3t_5}{z'^2}-2\frac{t_1t_3t_5\xi'}{z'}\right)\left(-\frac{\partial}{\partial z'}\wedge \frac{\partial}{\partial \xi'}\right)$ is translated on $U_1\times \mathbb{P}_\mathbb{C}^1\times \mathbb{C}^5$ into
\begin{align*}
\left(2\frac{t_1^2t_3t_5}{z}-2\frac{t_1t_3t_5(z^5\xi+t_1z+t_2z^4+t_2^2t_4z^2-t_1^2t_5 z^3)}{z^2}\right)\frac{\partial}{\partial z}\wedge \frac{\partial}{\partial \xi}=-2t_1t_3t_5(z^3 \xi+t_2z^2+t_2^2t_4-t_1^2t_5z)\frac{\partial}{\partial z}\wedge \frac{\partial}{\partial \xi}
\end{align*}

Let us consider the other rational parts in $(\ref{r55})$,
\begin{align*}
&(\frac{t_2}{z'^3}-\frac{t_2}{z'^3}-\frac{t_2t_3}{z'^3}+\frac{(t_5+t_3t_5)t_2^2}{z'^8}-2\frac{(t_5+t_3t_5)t_2\xi'}{z'^4}+2\frac{(t_5+t_3t_5)t_1t_2}{z'^5}+\frac{(t_5+t_3t_5)t_2^4t_4^2}{z'^4}\\
&+\frac{(t_5+t_3t_5)t_1^4t_5^2}{z'^6}-2\frac{(t_5+t_3t_5)t_2^2t_4\xi'}{z'^2}
+2\frac{(t_5+t_3t_5)t_1^2t_5\xi'}{z'^3}+2\frac{(t_5+t_3t_5)t_1t_2^2t_4}{z'^3}\\&-2\frac{(t_5+t_3t_5)t_1^3t_5}{z'^4}
+2\frac{(t_5+t_3t_5)t_2^3t_4}{z'^6}-2\frac{(t_5+t_3t_5)t_1^2t_2t_5}{z'^7}-2\frac{(t_5+t_3t_5)t_1^2t_2^2t_4t_5}{z'^5})\left(-\frac{\partial}{\partial z'}\wedge \frac{\partial}{\partial \xi'}\right)
\end{align*}

Then the Poisson structure on $U_1\times \mathbb{P}_\mathbb{C}^1\times \mathbb{C}^5$ given by
\begin{align*}
\Lambda=A(z,\xi,t)\frac{\partial}{\partial z}\wedge \frac{\partial}{\partial \xi}
\end{align*}
where
\begin{align*}
&A(z,\xi,t)\\
&=t_2+(z+t_3z)\xi +(t_4+t_5z^7 +t_3t_4+t_5z^7)\xi^2+2t_1t_5(z^3\xi+t_2z^2+t_2^2t_4-t_1^2t_5z)+2t_1t_3t_5(z^3\xi+t_2z^2+t_2^2t_4-t_1^2t_5z)    \\
&+t_2t_3-(t_5+t_3t_5)t_2^2z^5+2(t_5+t_3t_5)t_2(z^5\xi+t_1z+t_2z^4+t_2^2t_4z^2-t_1^2t_5 z^3)z-2(t_5+t_3t_5)t_1t_2z^2\\
&-(t_5+t_3t_5)t_2^4t_4^2z-(t_5+t_3t_5)t_1^4t_5^2z^3+2(t_5+t_3t_5)t_2^2t_4(z^4\xi+t_1+t_2z^3+t_2^2t_4z-t_1^2t_5 z^2)\\
&-2(t_5+t_3t_5)t_1^2t_5(z^5\xi+t_1z+t_2z^4+t_2^2t_4z^2-t_1^2t_5 z^3)-2(t_5+t_3t_5)t_1t_2^2t_4\\
&-2(t_5+t_3t_5)t_1^3t_5z-2(t_5+t_3t_5)t_2^3t_4z^3+2(t_5+t_3t_5)t_1^2t_2t_5z^4+2(t_5+t_3t_5)t_1^2t_2^2t_4 t_5z^2
\end{align*}
defines a global bivector field on $\mathcal{F}$ and $(\mathcal{F},\Lambda, \mathbb{C}^5,\omega)$ is a Poisson analytic family of deformations of $(F_5,\Lambda_0)=\omega^{-1}(0)$ such that the Poisson Kodaira-Spencer map is an isomorphism at $t=0$.

\end{example}

\section{Hopf surfaces}\label{section4}

In this section, we study Poisson deformations of (primary) Hopf surfaces. We determine obstructedness or unobstructedness of Poisson Hopf surfaces except for two classes of Poisson Hopf surfaces (see Table $\ref{h51}$ and Remark $\ref{h95}$). For unobstructed Poisson Hopf surfaces, we show the unobstructedness  by explicitly constructing Poisson analytic families of deformations of Poisson Hopf surfaces such that the Poisson Kodaira-Spencer map is an isomorphism at the distinguished point (see Theorem $\ref{h53}$). We extend the method in \cite{Weh81} in the context of holomorphic Poisson deformations of Hopf surfaces.

\subsection{Preliminaries}\

In this subsection, we review properties of Hopf surfaces. For the detail, we refer to \cite{Weh81}.

\begin{definition}
A compact complex surface $X$ is called $($primary$)$ Hopf surfaces if the universal covering is biholomorphically equivalent to the domain $W:=\mathbb{C}^2-\{0,0\}$ and the fundamental group equals the infinite cyclic group $\mathbb{Z}$.
\end{definition}

\begin{remark}\

\begin{enumerate}
\item The group of covering transformations of a Hopf surface is generated by a contraction
\begin{align*}
f:\mathbb{C}^2\to \mathbb{C}^2,\,\,\,\,\,f(0)=0
\end{align*}
After a suitable choice of coordinates in $\mathbb{C}^2$, the contraction has the normal form
\begin{align*}
f(z,w)=(\alpha z+\lambda w^p, \delta w),
\end{align*}
where $p\in \mathbb{N}-\{0\}$ and $\alpha,\delta, \lambda\in \mathbb{C}$ are constants subject to the restrictions $0<|\alpha |\leq |\delta |<1$ and $(\alpha-\delta^p)\lambda=0$.
\item Every Hopf surface $X$ is homeomorphic to $S^1\times S^3$ so that  the first Chern class of any Hopf surface $X$ is $c_1=0$, and the second Chern class of any Hopf surface $X$ is $c_2=0$.
\item We have $h^{0,0}=h^{0,1}=h^{2,1}=h^{2,2}=1$ and $h^{p,q}=0$ for all other $(p,q)$, where $h^{p,q}=\dim_\mathbb{C} H^q(X, \wedge^p\Omega_X^1)$.
\item $\dim_\mathbb{C} H^0(X, \Theta_X)=\dim_\mathbb{C} H^1(X,\Theta_X)$ and $H^2(X,\Theta_X)=0$.
\end{enumerate}
\end{remark}

\begin{theorem}[see \cite{Weh81} p.23]\label{h85}
Table $\ref{h15}$ shows the classification of Hopf surfaces $X=W/\langle f \rangle$ in the complex analytic case. Hopf surfaces of a fixed type are classified by the different values, which the parameters of the contraction
\begin{align*}
f:\mathbb{C}^2 \to \mathbb{C}^2,\,\,\,\,\,f(0)=0
\end{align*}
can take. Table $\ref{h16}$ shows a basis of the vector space $H^0(X, \Theta_X)$ and the group $Aut(X)$ of holomorphic automorphisms for the Hopf surfaces $X=W/\langle f \rangle $ of each type respectively.
\end{theorem}

\begin{table}
\begin{center}
\begin{tabular}{| c | c | c | c |} \hline
type & $\dim H^0(X,\Theta_X)$ & $f(z,w)$ & parameters \\ \hline
$\textnormal{IV}$ & 4 & $(\alpha z,\alpha w)$ & $0<|\alpha|<1$ \\ \hline
$\textnormal{III}$ & 3 & $(\delta^p z,\delta w)$ & $p\in \mathbb{N}-\{0,1\},0<|\delta|<1$\\ \hline
$\textnormal{II}_a$ & 2 & $(\delta^pz+w^p,\delta w)$ & $p \in \mathbb{N}-\{0,1\},0<|\delta|<1$ \\ \hline
$\textnormal{II}_b$ & 2 & $(\alpha z +w, \alpha w)$ & $0<|\alpha|<1$ \\ \hline
$\textnormal{II}_c$ & 2 & $(\alpha z,\delta w)$ & $0<|\alpha|<|\delta| <1, \alpha\ne \delta^p$ for all $p\in \mathbb{N}$ \\ \hline
\end{tabular}
\caption{Classification of Hopf surfaces and infinitesimal automorphisms} \label{h15}
\end{center}
\end{table}

\begin{table}
\begin{center}
\begin{tabular}{| c | c | c | c |} \hline
type & basis of $ H^0(X,\Theta_X)$ & $Aut(X)$ \\ \hline
$\textnormal{IV}$ & $z\frac{\partial}{\partial z}, w\frac{\partial}{\partial z},z\frac{\partial}{\partial w},w\frac{\partial}{\partial w}$  & $GL(2,\mathbb{C})/\langle f \rangle$ \\ \hline
$\textnormal{III}$ & $ z\frac{\partial}{\partial z},w\frac{\partial}{\partial w},w^p \frac{\partial}{\partial z}$  & $\{(z,w)\mapsto (az+bw^p,dw):ad\ne0,b\in \mathbb{C}\}/\langle f\rangle  $ \\ \hline
$\textnormal{II}_a$ & $pz\frac{\partial}{\partial z} + w\frac{\partial}{\partial w}, w^p \frac{\partial}{\partial z}$ & $\{(z,w)\mapsto (a^p z+ b w^p,aw):a\in \mathbb{C}^*,b\in \mathbb{C}\}/\langle f\rangle$ \\ \hline
$\textnormal{II}_b$ & $z \frac{\partial}{\partial z} + w\frac{\partial}{\partial w}, w\frac{\partial}{\partial z}$ & $\{(z,w)\mapsto (az+bw,aw):a\in \mathbb{C}^*,b\in \mathbb{C}\}/\langle f\rangle $\\ \hline
$\textnormal{II}_c$ &  $z\frac{\partial}{\partial z},w\frac{\partial}{\partial w}$ & $\{(z,w)\mapsto (az,dw):ad\ne 0\}/\langle f \rangle$ \\ \hline
\end{tabular}
\caption{Group of automorphisms on Hopf surfaces} \label{h16}
\end{center}
\end{table}

\subsection{Holomorphic Poisson structures on Hopf surfaces and Poisson automorphisms}\

In this subsection, we describe holomorphic Poisson structures on Hopf surfaces of each type in Table \ref{h15}, and for given a Poisson Hopf surface $(X,\Lambda_0)$, we compute the infinitesimal automorphisms $\mathbb{H}^0(X,\Theta_X^\bullet)$, and describe group $Aut(X,\Lambda_0)$ of (holomorphic) Poisson automorphisms for any Poisson Hopf surface $(X,\Lambda_0)$.

\subsubsection{Hopf surfaces of type $\textnormal{IV}$}\

We describe holomorphic Poisson structures on a Hopf surface of type $\textnormal{IV}$ defined by $f:(z,w)\to (\alpha z,\alpha w)$ for $0<|\alpha|<1$. A holomorphic bivector field on $X$ is induced from a $\langle f\rangle$-invariant holomorphic bivector field on $W$, equivalently an invariant bivector field on $W$ under $f^n:(z,w)\mapsto (z',w')=(\alpha^n z,\alpha^n w)$ for all $n\in\mathbb{Z}$. Let $g(z,w)\frac{\partial}{\partial z}\wedge \frac{\partial}{\partial w}$ be an $\langle f\rangle$-invariant holomorphic bivector field on $W$. Since $g(z,w)\frac{\partial}{\partial z}\wedge \frac{\partial}{\partial w}=g(z',w')\frac{\partial}{\partial z'}\wedge \frac{\partial}{\partial w'}$, and $\frac{\partial}{\partial z}\wedge \frac{\partial}{\partial w}=\alpha^{2n}\frac{\partial }{\partial z'}\wedge \frac{\partial}{\partial w'}$, we get $g(z,w)=\frac{1}{\alpha^{2n}}g(\alpha^n z,\alpha^n w)$. By Hartog's theorem, $g(z,w)$ is extended to a holomorphic function on the whole $\mathbb{C}^2$. Let $g(z,w)=\sum_{h,k\geq 0}^\infty c_{hk}z^hw^k$ be the power series expansion of $g(z,w)$. Then since $0<|\alpha|<1$,  we have 
\begin{align*}
g(z,w)=\lim_{n\to \infty}\frac{1}{\alpha^{2n}} \sum_{h,k \geq 0}c_{hk}\alpha^{hn}\alpha^{kn}z^h w^k=c_{20}z^2+c_{11}zw+c_{02}w^2.
\end{align*}
 Hence $H^0(X,\wedge^2 \Theta_X)$ is generated by $\left(z^2\frac{\partial}{\partial z}\wedge \frac{\partial}{\partial w}, zw\frac{\partial}{\partial z}\wedge \frac{\partial}{\partial w},w^2\frac{\partial}{\partial z}\wedge \frac{\partial}{\partial w}\right)$ which give holomorphic Poisson structures on $X$.

Now let $\Lambda_0=(Az^2+Bzw+Cw^2)\frac{\partial}{\partial z}\wedge \frac{\partial}{\partial w}$ be a Poisson structure on a Hopf surface of type $\textnormal{IV}$ defined by $f:(z,w)\to (\alpha z,\alpha w)$ for some constants $(A,B,C)\in \mathbb{C}^3$. We find infinitesimal Poisson automorphism $\mathbb{H}^0(X, \Theta_X^\bullet)=ker(H^0(X, \Theta_X)\xrightarrow{[\Lambda_0,-]} H^0(X, \wedge^2 \Theta_X))$. We note a basis of $H^0(X,\Theta_X)$ in Table \ref{h16}. Then

{\small{\begin{align}\label{h78}
&[(Az^2+Bzw+Cw^2)\frac{\partial}{\partial z}\wedge \frac{\partial}{\partial w}, (dz+ew)\frac{\partial}{\partial z}+ (fz+gw)\frac{\partial}{\partial w} ]\\
&=\left((-Ad-Bf+gA)z^2+(-2Ae-2Cf)zw+(Cd-Be-Cg)w^2\right)\frac{\partial}{\partial z}\wedge \frac{\partial}{\partial w}=0 \notag\\
&\iff  -Ad-Bf+gA=0,\,\,\,\,\,-2Ae-2Cf=0,\,\,\,\,\,Cd-Be-Cg=0.\notag
\end{align}}}
Let us represent our computation in the matrix form.
\begin{equation}\label{h35}
M\cdot v:=\left(
\begin{matrix}
-A & 0 & -B & A\\
0 & -2A & -2C & 0\\
C & -B & 0 & -C
\end{matrix}
\right)
\left(
\begin{matrix}
d\\
e\\
f\\
g
\end{matrix}
\right)
=
\left(
\begin{matrix}
-Ad-Bf+gA\\
-2Ae-2Cf\\
Cd-Be-Cg
\end{matrix}
\right)
\end{equation}
Since the first column and the last column of $M$ are dependent, and $\det \left(
\begin{matrix}
0 & -B & A\\
 -2A & -2C & 0\\
 -B & 0 & -C
\end{matrix}
\right)=0$, we see that if $(A,B,C)=0$, then $\dim_\mathbb{C} \mathbb{H}^0(X,\Theta_X^\bullet)=4$, and if $(A,B,C)\ne 0$, then $\dim_\mathbb{C} \mathbb{H}^0(X,\Theta_X^\bullet)=2$. Let us find the basis of $\mathbb{H}^0(X,\Theta_X^\bullet)$. 
\begin{enumerate}
\item If $\Lambda_0=0$, then we have
\begin{align}\label{h22}
\mathbb{H}^0(X,\Theta_X^\bullet)=Span_\mathbb{C}\left\langle z\frac{\partial}{\partial z}, w\frac{\partial}{\partial z}, z\frac{\partial}{\partial w}, w\frac{\partial}{\partial w} \right\rangle
\end{align}

\item If $\Lambda_0\ne 0$ (i.e. $(A,B,C)\ne0$), then $\mathbb{H}^0(X,\Theta_X^\bullet)$ is generated by $\left(\begin{matrix}
d\\
e\\
f\\
g
\end{matrix}
\right)=\left(
\begin{matrix}
1\\
0\\
0\\
1
\end{matrix}
\right)$, and
$\left(
\begin{matrix}
B\\
C\\
-A\\
0
\end{matrix}
\right)
$. In other words,
\begin{align}\label{h23}
\mathbb{H}^0(X,\Theta_X^\bullet)\cong Span_\mathbb{C} \left\langle  z\frac{\partial}{\partial z}+w\frac{\partial}{\partial w}, (Bz+Cw)\frac{\partial}{\partial z}-Az\frac{\partial}{\partial w} \right\rangle 
\end{align}
\end{enumerate}

Lastly we discuss groups of Poisson automorphisms of Hopf surfaces of type $\textnormal{IV}$.

\begin{lemma}
Let $X$ be a Hopf surface of type $\textnormal{IV}$. The group of Possion atuomorphisms of $(X,\Lambda_0=(Az^2+Bzw+Cw^2)\frac{\partial}{\partial z}\wedge \frac{\partial}{\partial w} ) $ is as follows.
\begin{enumerate}
\item if $(A,B,C)=0$,
\begin{align}\label{h20}
Aut(X,\Lambda_0=0)=Aut(X)=GL(2,\mathbb{C})/\langle f\rangle
\end{align} 
\item if $(A,B,C)\ne 0$,
\begin{align}\label{h21}
Aut(X,\Lambda_0)=\{(z,w)\mapsto ((a+bB)z+bCw,-bAz+aw) :a(a+bB)+b^2AC\ne 0  \} /\langle f\rangle
\end{align}
\end{enumerate}
\end{lemma}

\begin{proof}
Since $Aut(X)$ consists of linear transformations by Table $\ref{h16}$, $Aut(X,\Lambda_0)\subset Aut(X)$ preserving $\Lambda_0$ is completely determined by infinitesimal Poisson automorphisms $(\ref{h22})$  and $(\ref{h23})$ so that we get $(\ref{h20})$ and $(\ref{h21})$. We can also directly compute $Aut(X,\Lambda_0)$ by finding the subgroup of $Aut(X)$ which preserves $\Lambda_0$.
\end{proof}

We summarize our discussion in this subsection in Table $\ref{h28}$ and Table $\ref{h29}$.

\subsubsection{Hopf surfaces of type $\textnormal{III}$}\

We describe holomorphic Poisson structures on a Hopf surface of type $\textnormal{III}$ defined by $f:(z,w)\to (\delta^p  z,\delta w)$ for $p\in \mathbb{N}-\{0,1\}$, and $0<|\delta |<1$.  A holomorphic bivector field on $X$ is induced from a $\langle f\rangle$-invariant holomorphic bivector field on $W$, equivalently an invariant bivector field on $W$ under $f^n:(z,w)\mapsto (\delta^{np}z,\delta^nw)$ for all $n\in \mathbb{Z}$. Let $g(z,w)\frac{\partial}{\partial z}\wedge \frac{\partial}{\partial w}$ be an $\langle f\rangle$-invariant holomorphic bivector field on $W$. Since $g(z,w)\frac{\partial}{\partial z}\wedge \frac{\partial}{\partial w}=g(z',w')\frac{\partial}{\partial z'}\wedge \frac{\partial}{\partial w'}$, and $\frac{\partial}{\partial z}\wedge \frac{\partial}{\partial w}=\delta^{n(p+1)}\frac{\partial }{\partial z'}\wedge \frac{\partial}{\partial w'}$, we get $g(z,w)=\frac{1}{\delta^{n(p+1)}}g(\delta^{np} z,\delta^n w)$. By Hartog's theorem, $g(z,w)$ is extended to a holomorphic function on the whole $\mathbb{C}^2$. Let $g(z,w)=\sum_{h,k\geq 0}^\infty c_{hk}z^hw^k$ be the power series expansion of $g(z,w)$. Then since $0<|\delta|<1$,  we have 
\begin{align*}
g(z,w)=\lim_{n\to \infty}\frac{1}{\delta^{n(p+1)}} \sum_{h,k\geq 0} c_{hk}\delta^{nph} \delta^{nk}z^h w^k=\lim_{n\to \infty}\frac{1}{\delta^{n(p+1)}} \sum_{h,k\geq 0} c_{hk}\delta^{n(ph+k)}z^h w^k=c_{11}zw+c_{0(p+1)}w^{p+1}.
\end{align*}
Hence $H^0(X,\wedge^2 \Theta_X)$ is generated by $\left(zw\frac{\partial}{\partial z}\wedge \frac{\partial}{\partial w}, w^{p+1}\frac{\partial}{\partial z}\wedge \frac{\partial}{\partial w}\right)$ which give holomorphic Poisson structures on $X$.

Now let $\Lambda_0=(Azw+Bw^{p+1})\frac{\partial}{\partial z}\wedge \frac{\partial}{\partial w}$ be a Poisson structure on a Hopf surface of type $\textnormal{III}$ defined by $f:(z,w)\to (\delta^p z,\delta w)$ for $p\in \mathbb{N}-\{0,1\}$, and $0<|\delta| <1$ for some constants $(A,B)\in \mathbb{C}^2$. We find infinitesimal Poisson automorphisms $\mathbb{H}^0(X, \Theta_X^\bullet)=ker(H^0(X,\Theta_X)\xrightarrow{[\Lambda_0,-]} H^0(X, \wedge^2 \Theta_X))$. We note a basis of $H^0(X,\Theta_X)$ in Table $\ref{h16}$. Then
{\small{\begin{align}\label{h62}
&[(Azw+ Bw^{p+1})\frac{\partial}{\partial z}\wedge \frac{\partial}{\partial w},(dz+ew^p)\frac{\partial}{\partial z}+fw\frac{\partial}{\partial w}]=(Bd-Ae-pB f)w^{p+1}\frac{\partial}{\partial z}\wedge \frac{\partial}{\partial w}=0 \notag\\
&\iff Bd-Ae-pBf=0.\notag
\end{align}}}
If $A=B=0$ so that $\Lambda_0=0$, then $\mathbb{H}^0(X,\Theta_X^\bullet)=3$. If $\Lambda_0\ne 0$, then $\mathbb{H}^0(X,\Theta_X^\bullet)=2$. Let us find a basis of $\mathbb{H}^0(X, \Theta_X^\bullet)=ker(H^0(X, \Theta_X)\xrightarrow{[\Lambda_0,-]} H^0(X, \wedge^2 \Theta_X))$.
\begin{enumerate}
\item If $\Lambda_0\ne 0$, then $\mathbb{H}^0(X, \Theta_X^\bullet)=H^0(X, \Theta_X)$ so that
\begin{align}
\mathbb{H}^0(X, \Theta_X^\bullet)\cong Span_\mathbb{C} \left\langle z\frac{\partial}{\partial z}, w^p\frac{\partial}{\partial z}, w\frac{\partial}{\partial w}  \right\rangle
\end{align}
\item If $\Lambda_0=B w^{p+1}\frac{\partial }{\partial z}\wedge \frac{\partial}{\partial w}, B\ne 0$, then
\begin{align}
\mathbb{H}^0(X, \Theta_X^\bullet)\cong Span_\mathbb{C} \left\langle pz\frac{\partial}{\partial z}+w\frac{\partial}{\partial w}, w^p\frac{\partial}{\partial z} \right\rangle
\end{align}
\item If $\Lambda_0=(Azw+Bw^{p+1})\frac{\partial}{\partial z}\wedge \frac{\partial}{\partial w}, A\ne 0$, then
\begin{align}
\mathbb{H}^0(X, \Theta_X^\bullet)\cong Span_\mathbb{C} \left\langle \left( z+\frac{B}{A} w^p \right)\frac{\partial}{\partial z}  , -\frac{pB}{A}w^p\frac{\partial}{\partial z}+w\frac{\partial}{\partial w} \right\rangle
\end{align}
\end{enumerate}

Lastly we discuss groups of Poisson auotmorphisms of Hopf surfaces of type $\textnormal{III}$.
\begin{lemma}
Let $X$ be a Hopf surface of type $\textnormal{III}$. The group of Poisson automorphisms of $(X,\Lambda_0=(Azw+Bw^{p+1})\frac{\partial}{\partial z}\wedge \frac{\partial}{\partial w} ) $ is as follows.
\begin{enumerate}
\item if $\Lambda_0=0$,
\begin{align}
Aut(X,\Lambda_0)=Aut(X)=\{(z,w)\mapsto (az+bw^p,dw):ad\ne 0, b\in \mathbb{C}\}/\langle f\rangle
\end{align}
\item if $\Lambda_0=Bw^{p+1}\frac{\partial}{\partial z}\wedge \frac{\partial}{\partial w},B\ne 0$,
\begin{align}\label{o25}
Aut(X,\Lambda_0)=\{(z,w) \mapsto (d^p z+b w^p,dw):d\ne 0, b\in \mathbb{C}\}/\langle f \rangle
\end{align}
\item if $\Lambda_0=(Azw+Bw^{p+1})\frac{\partial}{\partial z}\wedge \frac{\partial}{\partial w}, A\ne 0$,
\begin{align}\label{o26}
Aut(X,\Lambda_0)=\{(z,w)\mapsto \left( az+\frac{B}{A}(a-d^p)w^p  , dw  \right): ad\ne 0\}/\langle f \rangle
\end{align}
\end{enumerate}
\end{lemma} 

\begin{proof}
We note that by Table $\ref{h16}$, we have $Aut(X)=\{(z,w)\mapsto (z',w')=(az+bw^p, dw): ad\ne 0, b\in\mathbb{C}\} / \langle f \rangle$. Note that $Aut(X,\Lambda_0)\subset Aut(X)$ preserves $\Lambda_0$. Assume that $(z,w)\mapsto (z',w')=(az+bw^p,dw)$ with $ad\ne 0, b\in \mathbb{C}$ preserves $\Lambda_0$. Then
\begin{enumerate}
\item if $\Lambda_0=Bw^{p+1}\frac{\partial}{\partial z}\wedge \frac{\partial}{\partial w},B\ne 0$, since $\frac{\partial}{\partial z}\wedge \frac{\partial}{\partial w}=ad\frac{\partial}{\partial z'}\wedge \frac{\partial}{\partial w'}$, we have
\begin{align*}
Bw^{p+1}ad=B(dw)^{p+1}=Bd^{p+1}w^{p+1}\iff a=d^p
\end{align*}
so that we get $(\ref{o25})$.
\item if $\Lambda_0=(Azw+Bw^{p+1})\frac{\partial}{\partial z}\wedge \frac{\partial}{\partial w}, A\ne 0$, 
\begin{align*}
(Azw+Bw^{p+1})ad=A(az+bw^p)dw+B(dw)^{p+1} \iff b=\frac{B}{A}(a-d^p)
\end{align*}
so that we get $(\ref{o26})$.
\end{enumerate}
\end{proof}

We summarize our discussion in this section in Table $\ref{h28}$ and Table $\ref{h29}$.

\subsubsection{Hopf surfaces of type $\textnormal{II}_a$}\

We describe holomorphic Poisson structures on a Hopf surface of type $\textnormal{II}_a$ defined by $f:(z,w)\to (\delta^p z +w^p ,\delta w)$ for $p\in \mathbb{N}-\{0,1\}$ and $0<|\delta|<1$. A holomorphic bivector field on $X$ is induced from a $\langle f\rangle$-invariant holomorphic bivector field on $W$, equivalently an invariant bivector field on $W$ under $f^n:(z,w)\mapsto (z',w')=(\delta^{np}z+n\delta^{(n-1)p}w^p,\delta^n w)$ for all $n\in \mathbb{Z}$. Let $g(z,w)\frac{\partial}{\partial z}\wedge \frac{\partial}{\partial w}$ be an $\langle f\rangle$-invariant holomorphic bivector field on $W$. Since $g(z,w)\frac{\partial}{\partial z}\wedge \frac{\partial}{\partial w}=g(z',w')\frac{\partial}{\partial z'}\wedge \frac{\partial}{\partial w'}$, and $\frac{\partial}{\partial z}=\delta^{np} \frac{\partial}{\partial z'}$ and $\frac{\partial}{\partial w}=np\delta^{(n-1)p}w^{p-1}\frac{\partial}{\partial z'}+\delta^n \frac{\partial}{\partial w'}$ so that $\frac{\partial}{\partial z}\wedge \frac{\partial}{\partial w}=\delta^{n(p+1)}\frac{\partial }{\partial z'}\wedge \frac{\partial}{\partial w'}$, we get $g(z,w)=\frac{1}{\delta^{n(p+1)}}g(\delta^{(n-1)p}(\delta^p z+nw^p),\delta^n w)$. By Hartog's theorem, $g(z,w)$ is extended to a holomorphic function on the whole $\mathbb{C}^2$. Let $g(z,w)=\sum_{h,k\geq 0}^\infty c_{hk}z^hw^k$ be the power series expansion of $g(z,w)$. Then since $0<|\delta|<1$,  we have 
\begin{align*}
g(z,w)=\lim_{n\to \infty} \frac{1}{\delta^{n(p+1)}} \sum_{h,k\geq 0} c_{hk}\delta^{(n-1)ph}\delta^{nk} (\delta^pz+nw^p)^h w^k=c_{0(p+1)} w^{p+1}.
\end{align*}
 Hence $H^0(X,\wedge^2 \Theta_X)$ is generated by $\left(w^{p+1}\frac{\partial}{\partial z}\wedge \frac{\partial}{\partial w}\right)$ which give holomorphic Poisson structures on $X$.
 
 Now let $\Lambda_0=Aw^{p+1}\frac{\partial}{\partial z}\wedge \frac{\partial}{\partial w}$ be a Poisson structure on a Hopf surface of type $\textnormal{II}_a$ defined by $f:(z,w)\to (\delta^p z +w^p ,\delta w)$ for $p\in \mathbb{N}-\{0,1\}$ and $0<|\delta|<1$ for some constant $A\in \mathbb{C}$. We find infinitesimal Poisson auotmorphisms $\mathbb{H}^0(X, \Theta_X^\bullet)=ker(H^0(X, \Theta_X)\xrightarrow{[\Lambda_0, -]} H^0(X, \wedge^2 \Theta_X))$.  We note a basis of $H^0(X,\Theta_X)$ in Table $\ref{h16}$. Then
 \begin{align}\label{h63}
&[A w^{p+1}\frac{\partial}{\partial z}\wedge \frac{\partial}{\partial w},(cpz+dw^p)\frac{\partial}{\partial z}+cw\frac{\partial}{\partial w}]=0
\end{align}
Hence we have $\mathbb{H}^0(X,\Theta_X^\bullet)=H^0(X, \Theta_X)$ so that
\begin{align}
\mathbb{H}^0(X, \Theta_X^\bullet)\cong Span_\mathbb{C} \left\langle  pz\frac{\partial}{\partial z}+w\frac{\partial}{\partial w}, w^p\frac{\partial}{\partial z}   \right\rangle
\end{align}

Lastly we discuss groups of Poisson automorphisms of Hopf surfaces of type $\textnormal{II}_a$.
\begin{lemma}
Let $X$ be a Hopf surface of type $\textnormal{II}_a$. The group of Poisson automorphisms of $(X,\Lambda_0=Aw^{p+1}\frac{\partial}{\partial z}\wedge \frac{\partial}{\partial w})$ is
\begin{align*}
Aut(X,\Lambda_0)=Aut(X)=\{(z,w)\mapsto (a^pz+bw^p,aw):a\in \mathbb{C}^*,b\in \mathbb{C}\}/\langle f \rangle
\end{align*}
\end{lemma}

We summarize our discussion in this subsection in Table $\ref{h28}$ and Table $\ref{h29}$.

\subsubsection{Hopf surfaces of type $\textnormal{II}_b$}\

We describe holomorphic Poisson structures on a Hopf surface of type $\textnormal{IV}$ defined by $f:(z,w)\to (\alpha z+w,\alpha w)$ for $0<|\alpha|<1$. A holomorphic bivector field on $X$ is induced from a $\langle f\rangle$-invariant holomorphic bivector field on $W$, equivalently an invariant bivector field on $W$ under $f^n:(z,w)\mapsto (z',w')=(\alpha^n z+n\alpha^{n-1}w,\alpha^n w)$ for all $n\in \mathbb{Z}$. Let $g(z,w)\frac{\partial}{\partial z}\wedge \frac{\partial}{\partial w}$ be an $\langle f\rangle$-invariant holomorphic bivector field on $W$. Since $g(z,w)\frac{\partial}{\partial z}\wedge \frac{\partial}{\partial w}=g(z',w')\frac{\partial}{\partial z'}\wedge \frac{\partial}{\partial w'}$, and $\frac{\partial}{\partial z}\wedge \frac{\partial}{\partial w}=\alpha^{2n}\frac{\partial }{\partial z'}\wedge \frac{\partial}{\partial w'}$, we get $g(z,w)=\frac{1}{\alpha^{2n}}g(\alpha^n z+n\alpha^{n-1}w,\alpha^n w)$. By Hartog's theorem, $g(z,w)$ is extended to a holomorphic function on the whole $\mathbb{C}^2$. Let $g(z,w)=\sum_{h,k\geq 0}^\infty c_{hk}z^hw^k$ be the power series expansion of $g(z,w)$. Then since $0<|\alpha|<1$,  we have 
\begin{align*}
g(z,w)=\lim_{n\to \infty}\frac{1}{\alpha^{2n}}\sum_{h,k\geq 0}c_{hk}(\alpha^n z+n\alpha^{n-1} w)^h\alpha^{kn} w^k=\lim_{n\to \infty}\frac{1}{\alpha^{2n}}\sum_{h,k\geq 0}c_{hk}\alpha^{(n-1)h}\alpha^{nk}(\alpha z+n w)^h  w^k=c_{02}w^2
\end{align*}
 Hence $H^0(X,\wedge^2 \Theta_X)$ is generated by $\left(w^2\frac{\partial}{\partial z}\wedge \frac{\partial}{\partial w}\right)$ which give holomorphic Poisson structures on $X$.
 
 Now let $\Lambda_0=Aw^2\frac{\partial}{\partial z}\wedge \frac{\partial}{\partial w}$ be a Poisson structure on a Hopf surface of type $\textnormal{II}_b$ defined by $f:(z,w)\to (\alpha z+w,\alpha w)$ for $0<|\alpha|<1$ for some constant $A\in \mathbb{C}$. We find infinitesimal Poisson auotmorphisms $\mathbb{H}^0(X, \Theta_X^\bullet)=ker(H^0(X, \Theta_X)\xrightarrow{[\Lambda_0, -]} H^0(X, \wedge^2 \Theta_X))$. We note a basis of $H^0(X, \Theta_X)$ in Table $\ref{h16}$. Then
\begin{align}\label{h64}
&[Aw^2\frac{\partial}{\partial z}\wedge \frac{\partial}{\partial w}, (bz+cw)\frac{\partial}{\partial z}+bw\frac{\partial}{\partial w}]=0
\end{align}
Hence we have $\mathbb{H}^0(X, \Theta_X^\bullet)=H^0(X, \Theta_X)$ so that
\begin{align*}
\mathbb{H}^0(X, \Theta_X^\bullet)= Span_\mathbb{C} \left\langle z\frac{\partial}{\partial z}+w\frac{\partial}{\partial w}, w\frac{\partial}{\partial z}   \right\rangle
\end{align*}

Lastly we discuss groups of Poisson automorphisms of Hopf surfaces of type $\textnormal{II}_b$.
\begin{lemma}
Let $X$ be a Hopf surface of type $\textnormal{II}_b$. The group of Poisson automorphisms of $(X,\Lambda_0=Aw^2\frac{\partial}{\partial z}\wedge \frac{\partial}{\partial w})$ is
\begin{align*}
Aut(X,\Lambda_0)=Aut(X)=\{(z,w)\mapsto (az+bw,aw):a\in \mathbb{C}^*,b\in \mathbb{C}\}/\langle f \rangle
\end{align*}
\end{lemma}

We summarize our discussion in this section in Table $\ref{h28}$ and Table $\ref{h29}$.

\subsubsection{Hopf surfaces of type $\textnormal{II}_c$}\

We describe holomorphic Poisson structures on a Hopf surface of type $\textnormal{IV}$ defined by $f:(z,w)\to (\alpha z,\delta w)$ for $0<|\alpha|<|\delta | <1$ and $\alpha \ne \delta^p$ for all $p \in \mathbb{N}$. A holomorphic bivector field on $X$ is induced from a $\langle f\rangle$-invariant holomorphic bivector field on $W$, equivalently an invariant bivector field on $W$ under $f^n:(z,w)\mapsto (z',w')=(\alpha^n z,\delta^n w)$ for all $n\in \mathbb{Z}$. Let $g(z,w)\frac{\partial}{\partial z}\wedge \frac{\partial}{\partial w}$ be an $\langle f\rangle$-invariant holomorphic bivector field on $W$. Since $g(z,w)\frac{\partial}{\partial z}\wedge \frac{\partial}{\partial w}=g(z',w')\frac{\partial}{\partial z'}\wedge \frac{\partial}{\partial w'}$, and $\frac{\partial}{\partial z}\wedge \frac{\partial}{\partial w}=\alpha^{n}\delta^n\frac{\partial }{\partial z'}\wedge \frac{\partial}{\partial w'}$, we get $g(z,w)=\frac{1}{\alpha^{n}\delta^n}g(\alpha^n z,\delta^n w)$. By Hartog's theorem, $g(z,w)$ is extended to a holomorphic function on the whole $\mathbb{C}^2$. Let $g(z,w)=\sum_{h,k\geq 0}^\infty c_{hk}z^hw^k$ be the power series expansion of $g(z,w)$. Then since $0<|\alpha|<|\delta| <1$ and $\alpha\ne \delta^p$ for all $p\in \mathbb{N}$,  we have 
\begin{align*}
g(z,w)=\lim_{n\to \infty}\frac{1}{\alpha^{n}\delta^n} \sum_{h,k \geq 0}c_{hk}\alpha^{hn}\delta^{kn}z^h w^k=c_{11}zw.
\end{align*}
 Hence $H^0(X,\wedge^2 \Theta_X)$ is generated by $\left(zw\frac{\partial}{\partial z}\wedge \frac{\partial}{\partial w}\right)$ which give holomorphic Poisson structures on $X$.
 
 Now let $\Lambda_0=Azw\frac{\partial}{\partial z}\wedge \frac{\partial}{\partial w}$ on a Hopf surface of type $\textnormal{IV}$ defined by $f:(z,w)\to (\alpha z,\delta w)$ for $0<|\alpha|<|\delta | <1$ and $\alpha \ne \delta^p$ for all $p \in \mathbb{N}$ for some constant $A\in \mathbb{C}$. We find infinitesimal Poisson auotmorphisms $\mathbb{H}^0(X, \Theta_X^\bullet)=ker(H^0(X, \Theta_X)\xrightarrow{[\Lambda_0, -]} H^0(X, \wedge^2 \Theta_X))$. We note a basis of $H^0(X, \Theta_X)$ in Table $\ref{h16}$. Then
\begin{align}\label{h65}
&[Azw\frac{\partial}{\partial z}\wedge \frac{\partial}{\partial w}, bz\frac{\partial}{\partial z}+cw\frac{\partial}{\partial w}]=0
\end{align}
Hence we have $\mathbb{H}^0(X,\Theta_X^\bullet)=H^0(X,\Theta_X)$ so that
\begin{align}
\mathbb{H}^0(X, \Theta_X^\bullet)=Span_\mathbb{C} \left\langle z\frac{\partial}{\partial z}, w\frac{\partial}{\partial w} \right\rangle
\end{align}

Lastly we discuss groups of Poisson automorphisms of Hopf surfaces of type $\textnormal{II}_c$.
\begin{lemma}
Let $X$ be a Hopf surface of type $\textnormal{II}_c$. The group of Poisson automorphisms of $(X,\Lambda_0=Azw\frac{\partial}{\partial z}\wedge \frac{\partial}{\partial w})$ is
\begin{align*}
Aut(X,\Lambda_0)=Aut(X)=\{(z,w)\mapsto (az,dw):ad\ne 0\}/\langle f \rangle
\end{align*}
\end{lemma}

We summarize our discussion in this subsection in Table $\ref{h28}$ and Table $\ref{h29}$.

\begin{table}
\begin{center}
\begin{tabular}{| c | c | c | c |c|} \hline
type & $\dim H^0(X, \wedge^2 \Theta_X)$ & basis of $H^0(X, \wedge^2 \Theta_X)$ & $f(z,w)$ & parameters \\ \hline
$\textnormal{IV}$ & 3 & $z^2\frac{\partial}{\partial z}\wedge \frac{\partial}{\partial w},zw\frac{\partial}{\partial z}\wedge \frac{\partial}{\partial w},w^2\frac{\partial}{\partial z}\wedge \frac{\partial}{\partial w}$ & $(\alpha z,\alpha w)$ & $0<|\alpha|<1$ \\ \hline
$\textnormal{III}$ & 2 & $zw\frac{\partial}{\partial z}\wedge \frac{\partial}{\partial w}, w^{p+1}\frac{\partial}{\partial z}\wedge \frac{\partial}{\partial w}$ & $(\delta^p z,\delta w)$ & $p\in \mathbb{N}-\{0,1\},0<|\delta|<1$\\ \hline
$\textnormal{II}_a$ & 1 &  $w^{p+1}\frac{\partial}{\partial z}\wedge \frac{\partial}{\partial w}$   & $(\delta^pz+w^p,\delta w)$ & $p \in \mathbb{N}-\{0,1\},0<|\delta|<1$ \\ \hline
$\textnormal{II}_b$ & 1  &  $w^2\frac{\partial}{\partial z}\wedge \frac{\partial}{\partial w}$  & $(\alpha z +w, \alpha w)$ & $0<|\alpha|<1$ \\ \hline
$\textnormal{II}_c$ & 1 &  $zw\frac{\partial}{\partial z}\wedge \frac{\partial}{\partial w}$   & $(\alpha z,\delta w)$ & $0<|\alpha|<|\delta| <1, \alpha\ne \delta^p$ for all $p\in \mathbb{N}$ \\ \hline
\end{tabular}
\end{center}
\caption{Holomorphic Poisson structures on Hopf surfaces}\label{h28}
\end{table}

\begin{table}\label{aa1}
{\small{\begin{center}
\begin{tabular}{| c | c | c | c |c|} \hline
type & Poisson structure $\Lambda_0$ & basis of $ \mathbb{H}^0(X,\Theta_X^\bullet)$ & $Aut(X,\Lambda_0)$  \\ \hline
$\textnormal{IV}$ & 0 &  $z\frac{\partial}{\partial z}, w\frac{\partial}{\partial z},z\frac{\partial}{\partial w},w\frac{\partial}{\partial w}$ & $GL(2,\mathbb{C})/\langle f \rangle$ \\ \hline
$\textnormal{IV}$ & $(Az^2+Bzw+Cw^2)\frac{\partial}{\partial z}\wedge \frac{\partial}{\partial w}$ &  & $\{(z,w)\mapsto ((a+bB)z+bCw,-bAz+aw)\}/\langle f\rangle$\\ \cline{5-5}
 & $(A,B,C)\ne 0$ & $z\frac{\partial}{\partial z}+w\frac{\partial}{\partial w},(Bz+Cw)\frac{\partial}{\partial z}-Az\frac{\partial}{\partial w}$ &  where $a(a+bB)+b^2AC\ne0$ \\ \hline
$\textnormal{III}$ & 0 &  $z\frac{\partial}{\partial z},w^p\frac{\partial}{\partial z}, w\frac{\partial}{\partial w}$ & $\{(z,w)\mapsto (az+bw^p,dw):ad\ne 0, b\in\mathbb{C}\}/\langle f\rangle$ \\ \hline
$\textnormal{III}$ &  $Bw^{p+1}\frac{\partial}{\partial z}\wedge \frac{\partial}{\partial w},B\ne 0$ & $pz\frac{\partial}{\partial z}+w\frac{\partial}{\partial w}, w^p \frac{\partial}{\partial z}$  & $\{(z,w)\mapsto (d^pz+bw^p , dw): d \in \mathbb{C}^*,b \in \mathbb{C}\}/\langle f\rangle$ \\ \hline
$\textnormal{III}$ & $(Azw+Bw^{p+1})\frac{\partial}{\partial z}\wedge \frac{\partial}{\partial w},A\ne 0$ & $(z+\frac{B}{A}w^p)\frac{\partial}{\partial z}, -\frac{pB}{A}w^p\frac{\partial}{\partial z}+w\frac{\partial}{\partial w}$  & $\{(z,w)\mapsto az+\frac{B}{A}(a-d^p)w^p,dw):ad\ne 0\}/\langle f\rangle$\\ \hline
$\textnormal{II}_a$ & $Aw^{p+1}\frac{\partial}{\partial z}\wedge \frac{\partial}{\partial w}$  & $pz\frac{\partial}{\partial z}+w\frac{\partial}{\partial w},w^p\frac{\partial}{\partial z}$   & $\{(z,w)\mapsto (a^pz+bw^p , aw): a \in \mathbb{C}^*,b \in \mathbb{C}\}/\langle f\rangle$ \\ \hline
$\textnormal{II}_b$ &  $Aw^2\frac{\partial}{\partial z}\wedge \frac{\partial}{\partial w}$ & $z\frac{\partial}{\partial z}+w\frac{\partial}{\partial w}, w\frac{\partial}{\partial z}$ & $\{(z,w)\mapsto (az+bw , aw): a \in \mathbb{C}^*,b \in \mathbb{C}\}/\langle f\rangle$   \\ \hline
$\textnormal{II}_c$ & $Azw\frac{\partial}{\partial z}\wedge \frac{\partial}{\partial w}$ & $z\frac{\partial}{\partial z}, w\frac{\partial}{\partial w}$ & $\{(z,w)\mapsto (az,dw):ad\ne 0\}/\langle f\rangle
$ \\ \hline
\end{tabular}
\end{center}}}
\caption{Infinitesimal Poisson automorphisms and Groups of Poisson automorphisms on Poisson Hopf surfaces} \label{h29}
\end{table}

\begin{remark}
Let $X$ be any Hopf surface. Since the first Chern class $c_1(X)=0$, the first Chern class of $\wedge^2 \Theta_X$ is $c_1(\wedge^2 \Theta_X)=0$, and the second Chern class $c_2(\wedge^2 \Theta_X)=0$. Then from Hirzebruch-Riemann-Roch theorem, we have $\chi(\wedge^2 \Theta_X)=\sum_{j=0}^2 (-1)^j \dim_\mathbb{C} H^j(X,\wedge^2 \Theta_X)=(\bold{ch} \wedge^2 \Theta_X\cdot \bold{td} X)[X]=0$. On the other hand, from Serre duality, we have $H^2(X,\wedge^2 \Theta_X)\cong H^0(X, (\wedge^2 \Theta_X)^*\otimes \wedge^2 \Omega_X^1)$. Since $H^0(X,\wedge^2 \Omega_X^1)=0$ and $H^0(X, \wedge^2 \Theta_X)\ne 0$ from Table $\ref{h28}$, we have $H^2(X,\wedge^2 \Theta_X)=0$. Hence since $\chi(\wedge^2 \Theta_X)=0$, we obtain
\begin{align}\label{h58}
\dim_\mathbb{C} H^0(X,\wedge^2 \Theta_X)=\dim_\mathbb{C} H^1(X,\wedge^2 \Theta_X)
\end{align}
\end{remark}

As in Lemma $\ref{r4}$,  since $H^2(X,\Theta_X)=0$ for any Hopf surface $X$, we have

\begin{lemma}\label{h90}
Let $(X, \Lambda_0)$ be any Poisson Hopf surface.  Then 
\begin{align}
\mathbb{H}^1(X, \Theta_{X}^\bullet)&\cong coker(H^0(X, \Theta_{X})\xrightarrow{[\Lambda_0,-]} H^0(X, \wedge^2 \Theta_{X}))\oplus ker(H^1(X, \Theta_{X})\xrightarrow{[\Lambda_0,-]} H^1(X,\wedge^2 \Theta_{X})) \label{h1}\\
\mathbb{H}^2(X,\Theta_{X}^\bullet)&\cong coker(H^1(X,\Theta_{X})\xrightarrow{[\Lambda_0,-]}H^1(X,\wedge^2 \Theta_X))
\end{align}
 $(X,\Lambda_0)$ is obstructed in Poisson deformations if for some $a , b$ where $a\in H^0(X,\wedge^2 \Theta_X)$ and $b\in ker(H^1(X, \Theta_{X})\xrightarrow{[\Lambda_0,-]} H^1(X,\wedge^2 \Theta_{X}))$, under the following map
\begin{align} \label{h10}
[-,-]:H^0(X, \wedge^2 \Theta_{X})\times H^1(X, \Theta_{X})\to H^1(X, \wedge^2 \Theta_{X})
\end{align}
$[a,b]\in H^1(X,\wedge^2 \Theta_{X})$ is not in the image of $ H^1(X, \Theta_{X})\xrightarrow{[\Lambda_0,-]} H^1(X,\wedge^2 \Theta_{X}) $.
\end{lemma}

\subsection{Calculations of $H^1(X,\Theta_X)$ and $H^1(X, \wedge^2 \Theta_X)$ in terms of the universal covering} \label{h3}\

We will describe $\mathbb{H}^1(X,\Theta_X^\bullet)$ for any Poisson Hopf surface $(X,\Lambda_0)$. Form $(\ref{h1})$, first we will describe
\begin{align}\label{h2}
 ker(H^1(X, \Theta_{X})\xrightarrow{[\Lambda_0,-]} H^1(X,\wedge^2 \Theta_{X}))
\end{align} 

In order to describe $(\ref{h2})$, we describe the basis $H^1(X, \Theta_X)$ and $H^1(X, \wedge^2 \Theta_X)$ in terms of universal covering $W$ of $X$ as elements in $H^0(W, \Theta_W)$ and $H^0(W, \wedge^2 \Theta_W)$, respectively. We use the method presented in \cite{Weh81}. We describe bases explicitly  for Hopf surfaces of any type since we actually need them to compute $(\ref{h2})$.

Let $X$ be Hopf surface. Take a finite covering $\{U_i\}_{i\in I}$ of $X$ where every $U_i$ is an open Stein subset of $X$ such that the inverse image $\tilde{U}_i:=v^{-1}(U_i)$ relative to the canonical projection
\begin{align*}
v:W \to X
\end{align*}
splits up into a disjoint union
\begin{align*}
\tilde{U}_i=\bigcup_{m\in \mathbb{Z}}^\bullet f^m (U_i')
\end{align*}
 and the canonical projection induces a homoemorphism $v|_{U_i'}:U_i'\to U_i$. Then the family $\tilde{\mathcal{U}}:=(\tilde{U}_i)_{i\in I}$ constitutes an open covering of $W$ and the mapping $f:W\to W$ induces morphisms
\begin{align*}
f_*:\Gamma(\tilde{U}_i,\Theta_W) &\to \Gamma(\tilde{U}_i, \Theta_W)\\
f_*:\Gamma(\tilde{U}_i,\wedge^2 \Theta_W) &\to \Gamma(\tilde{U}_i, \wedge^2 \Theta_W)
\end{align*}
for every $i\in I$ so that we have short exact sequences
\begin{align}
0\to C^\bullet(\mathcal{U},\Theta_X) & \to  C^\bullet(\tilde{\mathcal{U}},\Theta_W) \xrightarrow{id-f_*} C^\bullet (\tilde{\mathcal{U}},\Theta_W)\to 0 \label{hh12}\\
0\to C^\bullet(\mathcal{U},\wedge^2 \Theta_X) & \to  C^\bullet(\tilde{\mathcal{U}}, \wedge^2 \Theta_W) \xrightarrow{id-f_*} C^\bullet (\tilde{\mathcal{U}}, \wedge^2 \Theta_W)\to 0
\end{align}
which determines the long exact sequence
\begin{align}
0 \to H^0(X, \Theta_X) \to H^0(W, \Theta_W) & \xrightarrow{id-f_*} H^0(W, \Theta_W) \xrightarrow{\sigma} H^1(X,\Theta_X) \to \cdots \label{h4}\\
0 \to H^0(X, \wedge^2 \Theta_X) \to H^0(W, \wedge^2 \Theta_W) &  \xrightarrow{id-f_*} H^0(W,  \wedge^2 \Theta_W) \xrightarrow{\sigma} H^1(X,   \wedge^2 \Theta_X) \to \cdots \label{h5}
\end{align}

\begin{lemma}\label{h55}
From $(\ref{h4})$ and $(\ref{h5})$, we have isomorphisms
\begin{align}
M_1:=coker(H^0(W, \Theta_W) \xrightarrow{id-f_*} H^0(W, \Theta_W)) & \cong H^1(X,\Theta_X) \label{h8}\\
M_2:=coker(H^0(W, \wedge^2 \Theta_W)\xrightarrow{id-f_*} H^0(W, \wedge^2 \Theta_W))  & \cong H^1(X, \wedge^2 \Theta_X) \label{h9}
\end{align}
and explicitly we have
\begin{enumerate}
\item Type $\textnormal{IV}$
\begin{align}\label{h30}
M_1\cong Span_\mathbb{C}\left\langle z\frac{\partial}{\partial z},w\frac{\partial}{\partial z}, z\frac{\partial}{\partial w}, w\frac{\partial}{\partial w}\right\rangle,\,\,\,\,\,\,\,\,\,M_2\cong Span_\mathbb{C} \left\langle z^2\frac{\partial}{\partial z}\wedge \frac{\partial}{\partial w}, zw\frac{\partial}{\partial z}\wedge \frac{\partial}{\partial w}, w^2\frac{\partial}{\partial z}\wedge \frac{\partial}{\partial w} \right\rangle
\end{align}
\item Type $\textnormal{III}$
\begin{align}\label{h31}
M_1\cong Span_\mathbb{C} \left\langle z\frac{\partial}{\partial z}, w^p \frac{\partial}{\partial z}, w\frac{\partial}{\partial w} \right\rangle,\,\,\,\,\,\,\,\,\,\, M_2\cong Span_\mathbb{C} \left\langle zw\frac{\partial}{\partial z}\wedge \frac{\partial}{\partial w}, w^{p+1}\frac{\partial}{\partial z}\wedge \frac{\partial}{\partial w} \right\rangle
\end{align}
\item Type $\textnormal{II}_a$
\begin{align}\label{h32}
M_1\cong Span_\mathbb{C} \left\langle (\delta^p z-w^p)\frac{\partial}{\partial z}, w\frac{\partial}{\partial w} \right\rangle, \,\,\,\,\,\,\,\,\,\, M_2\cong Span_\mathbb{C} \left\langle  zw\frac{\partial}{\partial z}\wedge \frac{\partial}{\partial w} \right\rangle
\end{align}
\item Type $\textnormal{II}_b$
\begin{align}\label{h33}
M_1\cong Span_\mathbb{C}\left\langle  (\alpha z -w)\frac{\partial}{\partial z}+\alpha w \frac{\partial}{\partial w}, (\alpha z-w)\frac{\partial}{\partial w} \right\rangle,\,\,\,\,\,\,\,\,\,\, M_2\cong Span_\mathbb{C}  \left\langle z^2\frac{\partial}{\partial z}\wedge \frac{\partial}{\partial w}  \right\rangle
\end{align}
\item Type $\textnormal{II}_c$
\begin{align}\label{h34}
M_1\cong Span_\mathbb{C} \left\langle z\frac{\partial}{\partial z}, w\frac{\partial}{\partial w} \right\rangle,\,\,\,\,\,\,\,\,\,\,M_2\cong Span_\mathbb{C}   \left\langle zw\frac{\partial}{\partial z}\wedge \frac{\partial}{\partial w}  \right\rangle
\end{align}
\end{enumerate}
\end{lemma}

\begin{proof}
We note that by Hartog's theorem every element of $H^0(W, \Theta_W)$ and  $H^0(W, \wedge^2 \Theta_W)$ extends to a holomorphic vector field and a bivector field on $\mathbb{C}^2$ respectively. Taking 
\begin{align*}
\theta&=g(z,w)\frac{\partial}{\partial z}+h(z,w) \frac{\partial}{\partial w}\\
\Pi&=G(z,w)\frac{\partial}{\partial z}\wedge \frac{\partial}{\partial w}
\end{align*}
where
\begin{align*}
g(z,w)=\sum_{\mu, v\geq 0} a_{\mu v} z^\mu w^v,\,\,\,\,\,h(z,w)=\sum_{\mu, v\geq 0} b_{\mu v}z^\mu w^v,\,\,\,\,\,G(z,w)=\sum_{\mu, v \geq 0}c_{\mu v}z^\mu w^v.
\end{align*}

We will describe $(id-f_*)\theta$ and $(id-f_*)\Pi$ in each type of Hopf surfaces. We set
\begin{align}
f_*\theta&=\tilde{g}(z,w)\frac{\partial}{\partial z}+\tilde{h}(z,w)\frac{\partial}{\partial w} \label{h56}\\
f_* \Pi&=\tilde{G}(z,w)\frac{\partial}{\partial z}\wedge \frac{\partial}{\partial w}\label{h57}
\end{align}
where
\begin{align*}
\tilde{g}(z,w)=\sum_{\mu, v\geq 0} \tilde{a}_{\mu v} z^\mu w^v,\,\,\,\,\, \tilde{h}(z,w)=\sum_{\mu, v\geq 0} \tilde{b}_{\mu v}z^\mu w^v,\,\,\,\,\,\tilde{G}(z,w)=\sum_{\mu, v \geq 0}\tilde{c}_{\mu v}z^\mu w^v.
\end{align*}

\subsubsection{Proof of $(\ref{h30})$}\

Let us consider a Hopf surface of type $\textnormal{IV}$ defined by $f(z,w)=(\alpha z,\alpha w)$ with $0<|\alpha |<1$. First let us consider $f_* \theta$ in $(\ref{h56})$. Then we have
\begin{align*}
g(z,w)\alpha=\tilde{g}(\alpha z,\alpha w),\,\,\,\,\,\,\,\, h(z,w)\alpha=\tilde{h}(\alpha z,\alpha w),
\end{align*}
equivalently, we have
\begin{align*}
\sum_{\mu, v \geq 0}\alpha a_{\mu v}z^\mu w^v=\sum_{\mu, v\geq 0} \alpha^{\mu+v}\tilde{a}_{\mu v}z^\mu w^v,\,\,\,\,\,\,\,\,\,\,\sum_{\mu, v \geq 0}\alpha b_{\mu v} z^\mu w^v=\sum_{\mu, v\geq 0} \alpha^{\mu+v} \tilde{b}_{\mu v} z^\mu w^v.
\end{align*}
so that we obtain
\begin{align*}
\tilde{a}_{\mu v}=\alpha^{1-\mu-v}a_{\mu v},\,\,\,\,\,\,\,\,\tilde{b}_{\mu v}=\alpha^{1-\mu-v}b_{\mu v}.
\end{align*}
Then we have
\begin{align*}
(id-f_*)\theta= \left(\sum_{\mu,v\geq 0} a_{\mu v}(1-\alpha^{1-\mu-v})z^\mu w^v    \right)\frac{\partial}{\partial z} +   \left( \sum_{\mu,v\geq 0} b_{\mu v}(1-\alpha^{1-\mu-v})z^\mu w^v           \right) \frac{\partial}{\partial w}
\end{align*}
We note that if $(\mu, v)=(0,1),(1,0)$, the coefficients of $z^\mu w^v$ is zero so that $\frac{\partial}{\partial z},w\frac{\partial}{\partial z},z\frac{\partial}{\partial w}, w\frac{\partial}{\partial w}$ are linearly independent modulo $im(id-f_*)$. Since we have $(\ref{h8})$ (see \cite{Weh81} Lemma 2 p.27), we obtain 
\begin{align*}
M_1\cong Span_\mathbb{C} \left\langle z\frac{\partial}{\partial z}, w \frac{\partial}{\partial z}, z\frac{\partial}{\partial w }, w\frac{\partial}{\partial w} \right\rangle
\end{align*}

Next let us consider $f_* \Pi $ in $(\ref{h57})$. Then we have
\begin{align*}
G(z,w)\alpha^2=\tilde{G}(\alpha z,\alpha w),
\end{align*}
equivalently, we have
\begin{align*}
\sum_{\mu, v \geq 0} \alpha^2 c_{\mu v} z^\mu w^v=\sum_{\mu, v\geq 0} \alpha^{\mu+v}\tilde{c}_{\mu v} z^\mu w^v
\end{align*}
so that we obtain
\begin{align*}
\tilde{c}_{\mu v}=\alpha^{2-\mu-v} c_{\mu v}
\end{align*}
Then we have
\begin{align*}
(id-f_*)\Pi= \left(  \sum_{\mu, v \geq 0} c_{\mu v}(1-\alpha^{2-\mu -v})z^\mu w^v        \right) \frac{\partial}{\partial z}\wedge \frac{\partial}{\partial w}
\end{align*}
We note that if $(\mu,v)=(2,0),(1,1),(0,2)$, the coefficients of $z^\mu w^v$ is zero so that $z^2\frac{\partial}{\partial z}\wedge \frac{\partial}{\partial w}, zw\frac{\partial}{\partial z}\wedge \frac{\partial}{\partial w}, w^2\frac{\partial}{\partial z}\wedge \frac{\partial}{\partial w}$ are linearly independent modulo $im(id-f_*)$. Hence since $\dim_\mathbb{C} H^1(X, \wedge^2 \Theta_X)=3$ by $(\ref{h58})$, we have $\dim_\mathbb{C} M_2\geq \dim_\mathbb{C} H^1(X, \wedge^2 \Theta_X)$. On the other hand, from $(\ref{h5})$, $\dim_\mathbb{C} M_2=\dim_\mathbb{C} (im \,\sigma)\leq \dim_\mathbb{C}  H^1(X, \wedge^2 \Theta_X)$ so that we obtain $(\ref{h9})$ and
\begin{align*}
M_2\cong Span_\mathbb{C} \left\langle z^2\frac{\partial}{\partial z}\wedge \frac{\partial}{\partial w}, zw\frac{\partial}{\partial z}\wedge \frac{\partial}{\partial w}, w^2\frac{\partial}{\partial z}\wedge \frac{\partial}{\partial w} \right\rangle
\end{align*}

\subsubsection{Proof of $(\ref{h31})$}\

Let us consider a Hopf surface of type $\textnormal{III}$ defined by $f(z,w)=(\delta^p z, \delta w)$ with $p\in \mathbb{N}-\{0,1\}$ and $0<|\delta |<1$. First we note that we have $(\ref{h8})$ and $M_1\cong Span_\mathbb{C} \left\langle z\frac{\partial}{\partial z}, w^p \frac{\partial}{\partial z}, w\frac{\partial}{\partial w} \right\rangle$ (see \cite{Weh81} Lemma 2 p.27).

Next let us consider $f_* \Pi$ in $(\ref{h57})$. Then we have
\begin{align*}
G(z,w)\delta^{p+1}=\tilde{G}(\delta^p z, \delta w),
\end{align*}
equivalently, we have
\begin{align*}
\sum_{\mu, v\geq 0} \delta^{p+1} c_{\mu v} z^\mu w^v=\sum_{\mu, v \geq 0} \delta^{p\mu+v}\tilde{c}_{\mu v}z^\mu w^v
\end{align*}
so that we obtain
\begin{align*}
\tilde{c}_{\mu v}=\delta^{p(1-\mu)-v+1} c_{\mu v}
\end{align*}
Then we have
\begin{align*}
(id-f_*)\Pi=\left(\sum_{\mu, v\geq 0} c_{\mu v}(1-\delta^{p(1-\mu)-v+1}) z^\mu w^v \right) \frac{\partial}{\partial z}\wedge \frac{\partial}{\partial w}
\end{align*}
We note that if  $(\mu,v)=(1,1),(0,p+1)$, then the coefficients of $z^\mu w^v$ are zero so that $zw\frac{\partial}{\partial z}\wedge \frac{\partial}{\partial w}, w^{p+1}\frac{\partial}{\partial z}\wedge \frac{\partial}{\partial w}$ are linearly independent modulo $im(id-f_*)$. Hence since $\dim_\mathbb{C} H^1(X, \wedge^2 \Theta_X)=2$ by $(\ref{h58})$, we have $\dim_\mathbb{C} M_2\geq \dim_\mathbb{C} H^1(X, \wedge^2 \Theta_X)$. On the other hand, from $(\ref{h5})$, $\dim_\mathbb{C} M_2=\dim_\mathbb{C} (im \,\sigma)\leq \dim_\mathbb{C}  H^1(X, \wedge^2 \Theta_X)$ so that we obtain $(\ref{h9})$ and 
\begin{align*}
M_2 \cong\langle zw\frac{\partial}{\partial z}\wedge \frac{\partial}{\partial w}, w^{p+1}\frac{\partial}{\partial z}\wedge \frac{\partial}{\partial w} \rangle
\end{align*}
\subsubsection{Proof of $(\ref{h32})$}\

Let us consider a Hopf surface of type $\textnormal{II}_a$ defined by $f(z,w)=(\delta^p z+w^p, \delta w )$ with $p\in \mathbb{N}-\{0,1\}$ and $ 0< |\delta |<1 $. First we note that we have $(\ref{h8})$ and $M_1\cong Span_\mathbb{C} \left\langle (\delta^p z-w^p)\frac{\partial}{\partial z},    w\frac{\partial}{\partial w} \right\rangle$ (see \cite{Weh81} Lemma 2 p.27).

Next let us consider $f_* \Pi$ in $(\ref{h57})$. We note that $f^{-1}:(z,w)\mapsto (z',w')= (\delta^{-p} z-\delta^{-2p} w^p, \delta^{-1}w)$. Since $\frac{\partial}{\partial z}=\delta^{-p} \frac{\partial}{\partial z'}$, and $\frac{\partial}{\partial w}=-p\delta^{-2p} w^{p-1}\frac{\partial}{\partial z'}+\delta^{-1}\frac{\partial}{\partial w'}$ so that $\frac{\partial}{\partial z}\wedge \frac{\partial}{\partial w}=\delta^{-p-1}\frac{\partial}{\partial z'}\wedge \frac{\partial}{\partial w'}$. Then we have
\begin{align*}
\tilde{G}(z,w)\delta^{-p-1}=G(\delta^{-p} z-\delta^{-2p} w^p, \delta^{-1}w),
\end{align*}
equivalently, we have
\begin{align*}
\sum_{\mu, v\geq 0} \tilde{a}_{\mu v}\delta^{-p-1} z^\mu w^v&=\sum_{m,n\geq 0} a_{mn}(\delta^{-p}z-\delta^{-2p}w^p)^m (\delta^{-1} w)^n\\
&=\sum_{m,n\geq 0} \sum_{i=0}^m a_{mn} \binom{m}{i} (\delta^{-p}z)^i(-\delta^{-2p}w^p)^{m-i}(\delta^{-1}w)^n\\
\iff \sum_{\mu, v\geq 0} \tilde{a}_{\mu v} z^\mu w^v&=\sum_{m,n\geq 0} \sum_{i=0}^m a_{mn} \binom{m}{i} (-1)^{m-i}\delta^{-2pm+pi-n+p+1} z^iw^{p(m-i)+n} 
\end{align*}
so that we obtain
\begin{align*}
\tilde{a}_{\mu v}
=\sum_{m\geq 0}a_{m, v+p(\mu-m)}\binom{m}{\mu}(-1)^{m-\mu}\delta^{-pm-v+p+1}
\end{align*}
Hence we get
\begin{align*}
(id-f_*)\Pi=\left(\sum_{\mu, v\geq 0}\left(a_{\mu v}-\sum_{m\geq 0}a_{m, v+p(\mu-m)}\binom{m}{\mu}(-1)^{m-\mu}\delta^{-pm-v+p+1}\right) z^\mu w^v\right)\frac{\partial}{\partial z}\wedge \frac{\partial}{\partial w}
\end{align*}

We note that if $(\mu,v)=(1,1)$, then since $m \geq \mu, v+p(\mu-m)=1+p(1-m)\geq 0$ so that $m=1$, the coefficient of $zw$ is zero. Hence $zw\frac{\partial}{\partial z}\wedge \frac{\partial}{\partial w}$ is not in $im(id-f_*)$. Since $\dim_\mathbb{C} H^1(X,\wedge^2\Theta_X)=1$ by $(\ref{h58})$, we have  $\dim_\mathbb{C} M_2\geq \dim_\mathbb{C} H^1(X, \wedge^2 \Theta_X)$. On the other hand, from $(\ref{h5})$, $\dim_\mathbb{C} M_2=\dim_\mathbb{C} (im \,\sigma)\leq \dim_\mathbb{C}  H^1(X, \wedge^2 \Theta_X)$ so that we obtain $(\ref{h9})$ and
\begin{align*}
M_2 \cong Span_\mathbb{C} \left\langle zw\frac{\partial}{\partial z}\wedge \frac{\partial}{\partial w}  \right\rangle
\end{align*}

\subsubsection{Proof of $(\ref{h33})$}\

Let us consider a Hopf surface of type $\textnormal{II}_b$ defined by $f(z,w)=(\alpha z +w, \alpha w)$ with $0<|\alpha |<1$. We note that $f^{-1}:(z,w)\mapsto (z',w')=(\alpha^{-1}z-\alpha^{-2}w, \alpha^{-1}w)=(\alpha^{-1}(z-\alpha^{-1}w), \alpha^{-1}w)$. Let us consider $f_*\theta$ in $(\ref{h56})$. Since $\frac{\partial}{\partial z}=\alpha^{-1}\frac{\partial}{\partial z'}$, and $\frac{\partial}{\partial w}=-\alpha^{-2} \frac{\partial}{\partial z'}+\alpha^{-1}\frac{\partial}{\partial w'}$, we have
\begin{align*}
\tilde{g}(z,w)\alpha^{-1}+\tilde{h}(z,w)(-\alpha^{-2})&=g(\alpha^{-1}z-\alpha^{-2}w,\alpha^{-1}w)\\
\tilde{h}(z,w)\alpha^{-1}&=h(\alpha^{-1}z-\alpha^{-2}w,\alpha^{-1}w),
\end{align*}
equivalently, we have
\begin{align}
\sum_{\mu,v \geq 0} \alpha^{-1} \tilde{a}_{\mu v}z^\mu w^v+\sum_{\mu, v\geq 0} (-\alpha^{-2})\tilde{b}_{\mu v} z^\mu w^v&=\sum_{m,n\geq 0} a_{mn}\alpha^{-m}(z-\alpha^{-1}w)^m \alpha^{-n}w^n \label{h59}\\
\sum_{\mu, v\geq 0}\alpha^{-1}\tilde{b}_{\mu v}z^\mu w^v&=\sum_{m,n\geq 0} b_{mn}\alpha^{-m}(z-\alpha^{-1}w)^m \alpha^{-n}w^n \label{h60}
\end{align}

Let us consider $(\ref{h60})$. Then
\begin{align*}
\sum_{\mu, v\geq 0}\alpha^{-1} \tilde{b}_{\mu v}z^\mu w^v&=\sum_{m,n\geq 0}\sum_{i=0}^m b_{mn}\binom{m}{i}z^i(-\alpha^{-1}w)^{m-i}\alpha^{-n-m}w^n\\
&=\sum_{m,n\geq 0}\sum_{i=0}^m b_{mn}\binom{m}{i}(-1)^{m-i}\alpha^{-2m+i-n}z^iw^{n+m-i}
\end{align*}
so that we obtain
\begin{align*}
\tilde{b}_{\mu v}=\sum_{m\geq 0} b_{m, v-m+\mu}\binom{m}{\mu}(-1)^{m-\mu} \alpha^{-m-v+1}
\end{align*}
On the other hand, let us consider $(\ref{h59})$. Then
\begin{align*}
\sum_{\mu,v \geq 0} \alpha^{-1} \tilde{a}_{\mu v}z^\mu w^v+\sum_{\mu, v\geq 0} (-\alpha^{-2})\tilde{b}_{\mu v} z^\mu w^v=\sum_{m,n\geq 0}\sum_{i=0}^m a_{mn}\binom{m}{i}(-1)^{m-i}\alpha^{-2m+i-n}z^iw^{n+m-i}
\end{align*}
so that we obtain
\begin{align*}
\tilde{a}_{\mu v}&=\alpha^{-1} \tilde{b}_{\mu v}+\sum_{m\geq 0} a_{m, v-m+\mu}\binom{m}{\mu}(-1)^{m-\mu} \alpha^{-m-v+1}\\
 &=\sum_{m\geq 0} b_{m, v-m+\mu}\binom{m}{\mu}(-1)^{m-\mu} \alpha^{-m-v}+\sum_{m\geq 0} a_{m, v-m+\mu}\binom{m}{\mu}(-1)^{m-\mu} \alpha^{-m-v+1}
\end{align*}
Hence we get
\begin{align*}
(id-f_*)\theta=\left(\sum_{\mu, v\geq 0} \left( a_{\mu v}-   \sum_{m\geq 0} b_{m, v-m+\mu}\binom{m}{\mu}(-1)^{m-\mu} \alpha^{-m-v}-\sum_{m\geq 0} a_{m, v-m+\mu}\binom{m}{\mu}(-1)^{m-\mu} \alpha^{-m-v+1}    \right)z^\mu w^v     \right)\frac{\partial}{\partial z}\\
 +    \left(\sum_{\mu, v \geq 0}\left(  b_{\mu v}- \sum_{m\geq 0} b_{m, v-m+\mu}\binom{m}{\mu}(-1)^{m-\mu} \alpha^{-m-v+1}     \right) z^\mu w^v   \right)\frac{\partial}{\partial w}
\end{align*}
We note that if $(\mu, v)=(1,0)$ in term $\frac{\partial}{\partial w}$, then since $m\geq \mu$ and $v-m+\mu=-m+1\geq 0$, we get $m=1$ so that the coefficient of $z$ is $0$.  Hence $z\frac{\partial}{\partial w}$ is not in $im(id-f_*)$. On the other hand, if $(\mu, v)=(1,0)$ in term $\frac{\partial}{\partial z}$, then $m=1$ so that the coefficient of $z$ in the term $\frac{\partial}{\partial z}$ is $-b_{10}\alpha^{-1}$. If $(\mu, v)=(0,1)$ in the term $\frac{\partial}{\partial w}$, then since $m\geq \mu$ and $v-m+\mu=1-m\geq 0$, we get $m=0,1$ so that the coefficient of $w$ is $b_{10}\alpha^{-1}$. Hence $w\frac{\partial}{\partial w}$ is not in $im(id-f_*)$. Then $z\frac{\partial}{\partial w}$ and $w\frac{\partial}{\partial w}$ are linearly independent modulo $im(id-f_*)$. We note that since $f_*(\alpha^2 z\frac{\partial}{\partial w})=(\alpha z-w)\frac{\partial}{\partial z}+(\alpha^2 z-\alpha w)\frac{\partial}{\partial w}$, we see that $(\alpha z-w)\frac{\partial}{\partial z}-\alpha w\frac{\partial}{\partial w}\equiv 0 \mod (id-f_*)$.   

Since we have $(\ref{h8})$ (see \cite{Weh81} Lemma 2 p.27), we obtain 
\begin{align*}
M_1\cong Span_\mathbb{C} \left\langle w\frac{\partial}{\partial w}, z \frac{\partial}{\partial w} \right\rangle=Span_\mathbb{C}\left\langle  (\alpha z -w)\frac{\partial}{\partial z}+\alpha w \frac{\partial}{\partial w}, (\alpha z-w)\frac{\partial}{\partial w} \right\rangle
\end{align*}

Next let us consider $f_* \Pi$ in  $(\ref{h57})$. Then we have
\begin{align*}
\tilde{G}(z,w)\alpha^{-2}=G(\alpha^{-1}(z-\alpha^{-1}w),\alpha^{-1}w),
\end{align*}
equivalently, we have
\begin{align*}
\sum_{\mu, v\geq 0} \alpha^{-2} \tilde{c}_{\mu v} z^\mu w^v=\sum_{m,n\geq 0} c_{mn}\alpha^{-m}(z-\alpha^{-1}w)^m \alpha^{-n}w^n
\end{align*}
so that we obtain
\begin{align*}
\tilde{c}_{\mu v}=\sum_{m\geq 0} c_{m, v-m+\mu}\binom{m}{\mu}(-1)^{m-\mu} \alpha^{-m-v+2}
\end{align*}
Hence we get
\begin{align*}
(id-f_*)\Pi=\left( \sum_{\mu,v\geq 0} \left( c_{\mu v}-\sum_{m\geq 0} c_{m,v-m+\mu}\binom{m}{\mu}(-1)^{m-\mu}\alpha^{-m-v+2}       \right)z^\mu w^v  \right)\frac{\partial}{\partial z}\wedge \frac{\partial}{\partial w}
\end{align*}
We note that if $(\mu, v)=(2,0)$, then since $m\geq \mu$ and $v-m+\mu=-m+2\geq 0$, we have $m=2$ so that the coefficient of $z^2$ is zero. Hence $z^2\frac{\partial}{\partial z}\wedge \frac{\partial}{\partial w}$ is not in $im(id-f_*)$. Since $\dim_\mathbb{C} H^1(X,\wedge^2\Theta_X)=1$ by $(\ref{h58})$, we have  $\dim_\mathbb{C} M_2\geq \dim_\mathbb{C} H^1(X, \wedge^2 \Theta_X)$. On the other hand, from $(\ref{h5})$, $\dim_\mathbb{C} M_2=\dim_\mathbb{C} (im \,\sigma)\leq \dim_\mathbb{C}  H^1(X, \wedge^2 \Theta_X)$ so that we obtain $(\ref{h9})$ and
\begin{align*}
M_2 \cong Span_\mathbb{C} \left\langle z^2\frac{\partial}{\partial z}\wedge \frac{\partial}{\partial w}  \right\rangle
\end{align*}

\subsubsection{Proof of $(\ref{h34})$}\

Let us consider a Hopf surface of type $\textnormal{II}_c$ defined by $f(z,w)=(\alpha z, \delta w)$ with $0<|\alpha |<|\delta |<1$ and $\alpha\ne \delta^p $ for all $p\in \mathbb{N}$. First let us consider $f_* \theta$ in $(\ref{h56})$. Then we have
\begin{align*}
g(z,w)\alpha=\tilde{g}(\alpha z,\delta w),\,\,\,\,\,\,\,\, h(z,w)\delta=\tilde{h}(\alpha z,\delta w),
\end{align*}
equivalently, we have
\begin{align*}
\sum_{\mu, v \geq 0}\alpha a_{\mu v}z^\mu w^v=\sum_{\mu, v\geq 0} \alpha^{\mu}\delta^{v}\tilde{a}_{\mu v}z^\mu w^v,\,\,\,\,\,\,\,\,\,\,\sum_{\mu, v \geq 0}\delta b_{\mu v} z^\mu w^v=\sum_{\mu, v\geq 0} \alpha^{\mu}\delta^{v} \tilde{b}_{\mu v} z^\mu w^v.
\end{align*}
so that we obtain
\begin{align*}
\tilde{a}_{\mu v}=\alpha^{1-\mu}\delta^{-v}a_{\mu v},\,\,\,\,\,\,\,\,\tilde{b}_{\mu v}=\alpha^{-\mu}\delta^{1-v}b_{\mu v}.
\end{align*}
Then we have
\begin{align*}
(id-f_*)\theta= \left(\sum_{\mu,v\geq 0} a_{\mu v}(1-\alpha^{1-\mu}\delta^{-v})z^\mu w^v    \right)\frac{\partial}{\partial z} +   \left( \sum_{\mu,v\geq 0} b_{\mu v}(1-\alpha^{-\mu}\delta^{-v})z^\mu w^v           \right) \frac{\partial}{\partial w}
\end{align*}
We see that $z\frac{\partial}{\partial z}, w\frac{\partial}{\partial w}$ are linearly independent modulo $im(id-f_*)$. Since we have $(\ref{h8})$ (see \cite{Weh81} Lemma 2 p.27), we obtain 
\begin{align*}
M_1\cong Span_\mathbb{C} \left\langle z\frac{\partial}{\partial z},  w\frac{\partial}{\partial w} \right\rangle
\end{align*}

Next let us consider $f_* \Pi $ in $(\ref{h57})$. Then we have
\begin{align*}
G(z,w)\alpha\delta=\tilde{G}(\alpha z,\delta w),
\end{align*}
equivalently, we have
\begin{align*}
\sum_{\mu, v \geq 0} \alpha\delta c_{\mu v} z^\mu w^v=\sum_{\mu, v\geq 0} \alpha^\mu \delta^v \tilde{c}_{\mu v} z^\mu w^v
\end{align*}
so that we obtain
\begin{align*}
\tilde{c}_{\mu v}=\alpha^{1-\mu}\delta^{1-v} c_{\mu v}
\end{align*}
Then we have
\begin{align*}
(id-f_*)\Pi= \left(  \sum_{\mu, v \geq 0} c_{\mu v}(1-\alpha^{1-\mu}\delta^{1-v})z^\mu w^v        \right) \frac{\partial}{\partial z}\wedge \frac{\partial}{\partial w}
\end{align*}
We note that the coefficient of $zw$ is zero so that $ zw\frac{\partial}{\partial z}\wedge \frac{\partial}{\partial w}$ is not in $im(id-f_*)$. Hence since $\dim_\mathbb{C} H^1(X, \wedge^2 \Theta_X)=1$ by $(\ref{h58})$, we have $\dim_\mathbb{C} M_2\geq \dim_\mathbb{C} H^1(X, \wedge^2 \Theta_X)$. On the other hand, from $(\ref{h5})$, $\dim_\mathbb{C} M_2=\dim_\mathbb{C} (im \,\sigma)\leq \dim_\mathbb{C}  H^1(X, \wedge^2 \Theta_X)$ so that we obtain $(\ref{h9})$ and
\begin{align*}
M_2\cong Span_\mathbb{C} \left\langle zw\frac{\partial}{\partial z}\wedge \frac{\partial}{\partial w}  \right\rangle
\end{align*}

This completes proof of Lemma $\ref{h55}$.
\end{proof}

\subsection{Descriptions of $ ker(H^1(X, \Theta_{X})\xrightarrow{[\Lambda_0,-]} H^1(X,\wedge^2 \Theta_{X})) $}\

Let $X$ be a Hopf surface. We keep the notations in the section $\ref{h3}$ for an open covering $\mathcal{U}=\{U_i\}$ of $X$, and an open covering $\tilde{\mathcal{U}}=\{\tilde{U}_i\}$, and so on.

\begin{lemma}\label{h6}
Let $(X,\Lambda_0)$ be any Poisson Hopf surface. From $(\ref{h4})$ and $(\ref{h5})$, we have a commutative diagram
\begin{center}
$\begin{CD}
0@>>> H^0(X,\wedge^2 \Theta_X) @>>> H^0(W,\wedge^2 \Theta_W) @>id-f_*>> H^0(W, \wedge^2 \Theta_W)@>\sigma>>H^1(X,\wedge^2 \Theta_X)@>>>\\
@. @A[\Lambda_0,-]AA @A[\Lambda_0,-]AA @A[\Lambda_0,-]AA @A[\Lambda_0,-]AA\\
0@>>> H^0(X,\Theta_X) @>>> H^0(W,\Theta_W)@>id-f_*>> H^0(W,\Theta_W)@>\sigma >> H^1(X,\Theta_X)@>>>
\end{CD}$
\end{center}

\end{lemma}
\begin{proof}
We show that second diagram commutes. Indeed, let $a\in H^0(W,\Theta_W)$. Then 
\begin{align*}
&(id-f_*)([\Lambda_0,a])=[\Lambda_0,a]-f_*([\Lambda_0,a])=[\Lambda_0,a]-[f_*\Lambda_0, f_*a]=[\Lambda_0, a]-[\Lambda_0, f_*a]=[\Lambda_0,(id-f_*)a]
\end{align*}

We show that third diagram commutes.
First we recall how the connecting homomorphism $\sigma $ is constructed. Let $a\in H^0(W,\Theta_W)$. Then there exists $\{a_i\}\in C^0(\tilde{\mathcal{U}},\Theta_W)$ such that $(id-f_* ) (\{a_i\})=a$. Then $\delta(\{a_i\})=\{a_j-a_i\}\in C^1(\tilde{\mathcal{U}},\Theta_W)$ defines $\sigma(a)$ in $H^1(X,\Theta_X)$. Then $[\Lambda_0,\sigma(a)]\in H^1(X,\wedge^2 \Theta_X)$ is defined by $[\Lambda_0, \{a_j-a_i\}]=\{[\Lambda_0, a_j-a_i]\}\in C^1(\tilde{\mathcal{U}}, \wedge^2 \Theta_W)$. On the other hand, since $(id-f_*)([\Lambda_0, \{a_i\}])=[\Lambda_0,(id-f_*)(\{a_i\})]=[\Lambda_0,a]$, we see that $\sigma([\Lambda_0,a])\in H^1(X,\wedge^2 \Theta_X)$ is defined by $\delta([\Lambda_0,\{a_i\}])= \{[\Lambda_0,a_j-a_i]\}\in C^1(\tilde{\mathcal{U}},\wedge^2 \Theta_W)$ so that the third diagram commutes.
\end{proof}

By Lemma \ref{h6}, $(\ref{h8})$ and $(\ref{h9})$, we have the following commutative diagram 
\begin{center}
$\begin{CD}
coker(H^0(W, \wedge^2 \Theta_W)\xrightarrow{id-f_*}H^0(W, \wedge^2 \Theta_W))@> \cong>> H^1(X,\wedge^2 \Theta_X)\\
@A[\Lambda_0,-]AA @A[\Lambda_0,-]AA \\
coker(H^0(W,  \Theta_W)\xrightarrow{id-f_*}H^0(W,  \Theta_W))@>\cong >> H^1(X, \Theta_X)\\
\end{CD}$
\end{center}

Then by Lemma $\ref{h90}$, we obtain

\begin{lemma}
Let $(X,\Lambda_0)$ be any Poisson Hopf surface. Then we have
\begin{align}\label{h37}
&\mathbb{H}^1(X, \Theta_X^\bullet)\\
&\cong coker( H^0(X,\Theta_X)\xrightarrow{[\Lambda_0,-]} H^0(X,\wedge^2 \Theta_X))\oplus ker\left(H^0(W,\Theta_W)/im(id-f_*)\xrightarrow{[\Lambda_0,-]} H^0(W,\wedge^2 \Theta_W)/im(id-f_*)\right) \notag
\end{align}
On the other hand, 
we have
\begin{align}\label{h38}
\mathbb{H}^2(X, \Theta_X^\bullet) \cong coker\left(H^0(W,\Theta_W)/im(id-f_*)\xrightarrow{[\Lambda_0,-]} H^0(W,\wedge^2 \Theta_W)/im(id-f_*)\right)
\end{align}
\end{lemma}

Next we describe $(\ref{h10})$ in terms of the universal covering.
\begin{lemma}

The map $H^0(W,\wedge^2 \Theta_W)^f\times  H^0(W,\Theta_W) \xrightarrow{[-,-]} H^0(W,\wedge^2 \Theta_W)$ induces 
\begin{align}\label{h11}
H^0(W,\wedge^2 \Theta_W)^f \times H^0(W,\Theta_W)/im(id-f_*) \xrightarrow{[-,-]} H^0(W,\wedge^2 \Theta_W)/im(id-f_*)
\end{align}
where $H^0(W,\wedge^2 \Theta_W)^f\cong H^0(X, \wedge^2 \Theta_X)$ is the invariant bivector fields on $W$ by the action generated by $f$.

\end{lemma}
\begin{proof}
Let us show that $(\ref{h11})$ is well-defined. Let $a\in H^0(W,\wedge^2 \Theta_W)^f$ and $b,c\in H^0(W,\Theta_W)$ with $b-c=(id-f_*)(d)$ for some $d\in H^0(W, \Theta_W)$. Then $[a,b]-[a,c]=[a,(id-f_*)d]=(id-f_*)([a,d])$.
\end{proof}

\begin{lemma}\label{h91}
We have a commutative diagram
\begin{center}
$\begin{CD}
H^0(X,\wedge^2 \Theta_X)\,\,\,\,\,\,\,\,\,\,\,\,\,\times @.\,\,\,\,\, H^1(X,\Theta_X) @>[-,-]>> H^1(X,\wedge^2 \Theta_X)\\
@V\cong VV @V\cong VV\ @V\cong VV\\
H^0(W,\wedge^2 \Theta_W)^f\,\,\,\,\,@.\times \,\,\,\,\,\, H^0(W,\Theta_W)/im(id-f_*)@>[-,-]>> H^0(W,\wedge^2 \Theta_W)/im(id-f_*)
\end{CD}$
\end{center}
\end{lemma}

\begin{proof}
Let $a\in H^0(W,\wedge^2\Theta_W)^f$ and $b\in H^0(W,\Theta_W)$. Then there exists $\{b_i\}\in C^0(\tilde{\mathcal{U}}, \Theta_W)$ such that $(id-f_*)(\{b_i\})=b$ 
and the class $\bar{b}\in H^0(W,\Theta_W)/im(id-f_*)$ corresponds to $\{b_j-b_i\}\in C^1(\mathcal{U}, \Theta_X)$. Since $[a, \{b_j-b_i\}]=\{[a,b_j]-[a,b_i]\}$ and $(id-f_*)([a,\{b_i\}])=[a,b]$, the diagram commutes.
\end{proof}

\subsection{$\mathbb{H}^1(X,\Theta_X^\bullet)$ and $\mathbb{H}^2(X, \Theta_X^\bullet)$ of Poisson Hopf surfaces $(X,\Lambda_0)$} 

\subsubsection{Hopf surfaces of type $\textnormal{IV}$}\

Let $X$ be a Hopf surface of type $\textnormal{IV}$, and $\Lambda_0=(Az^2+Bzw+Cw^2)\frac{\partial}{\partial z}\wedge \frac{\partial}{\partial w}$ be a Poisson structure on $X$.

\begin{enumerate}
\item if $\Lambda_0=0$, then $\mathbb{H}^1(X, \Theta_X^\bullet)=H^0(X, \Theta_X)\oplus H^1(X, \Theta_X)$, and $\mathbb{H}^2(X, \Theta_X^\bullet)=H^1(X, \wedge^2 \Theta_X)$.
\item if $\Lambda_0=(Az^2+Bzw+Cw^2)\frac{\partial}{\partial z}\wedge \frac{\partial}{\partial w} \ne 0$, then from $(\ref{h30})$, $(\ref{h78})$ and $(\ref{h35})$, we obtain
\begin{align}\label{h70}
ker\left(H^0(W,\Theta_W)/im(id-f_*)\xrightarrow{[\Lambda_0,-]} H^0(W,\wedge^2 \Theta_W)/im(id-f_*)\right)\cong Span_\mathbb{C}\left\langle z\frac{\partial}{\partial z}+w\frac{\partial}{\partial w}, (Bz+Cw)\frac{\partial}{\partial z}-Az\frac{\partial}{\partial w} \right\rangle
\end{align}
From $(\ref{h35})$ and $(\ref{h23})$, we see that $\dim_\mathbb{C} coker(H^0(X, \Theta_X)\xrightarrow{[\Lambda_0, -]}H^0(X, \wedge^2 \Theta_X))=1$. Hence from $(\ref{h37})$, we have $\dim_\mathbb{C} \mathbb{H}^1(X, \Theta_X^\bullet)=3$. On the other hand, from $(\ref{h38}),(\ref{h30})$, $(\ref{h78})$ and $(\ref{h35})$, we have $\dim_\mathbb{C} \mathbb{H}^2(X, \Theta_X^\bullet)=1$.
\end{enumerate}

\begin{lemma}
Let $X$ be a Hopf surface of type $\textnormal{IV}$, and $\Lambda_0=(Az^2+Bzw+Cw^2)\frac{\partial}{\partial z}\wedge \frac{\partial}{\partial w}\ne 0$ be a Poisson structure on $X$. Assume that $B^2-4AC\ne 0$, then we have
\begin{align}\label{h61}
coker(H^0(X, \Theta_X)\xrightarrow{[\Lambda_0, -]}H^0(X, \wedge^2 \Theta_X))\cong Span_\mathbb{C} \left\langle  (Az^2+Bzw+Cw^2)\frac{\partial}{\partial z}\wedge \frac{\partial}{\partial w} \right\rangle
\end{align}
\end{lemma}
\begin{proof}
Let $B^2-4AC\ne 0$. Assume that $A\ne 0$. Then from $(\ref{h78})$, $(\ref{h35})$ and $(\ref{h23})$, and
\begin{align*}
\det\left(
\begin{matrix}
1& B &0&1\\
0& C & 1& 0\\
0 &-A & 0& 0\\
1& 0&0&0
\end{matrix}
\right)\ne 0
\end{align*}
Hence the image of $H^0(X, \Theta_X)\xrightarrow{[\Lambda_0, -]}H^0(X, \wedge^2 \Theta_X)$ is generated by 
\begin{align*}
[\Lambda_0, z \frac{\partial}{\partial z}]=(-Az^2+Cw^2)\frac{\partial}{\partial z}\wedge \frac{\partial}{\partial w},\,\,\,\,\,\,\,\,\,\,\, [\Lambda_0, w\frac{\partial}{\partial z}] = (-2Azw-Bw^2)\frac{\partial}{\partial z}\wedge \frac{\partial}{\partial w}
\end{align*}
By considering
\begin{align*}
\det\left(
\begin{matrix}
-A& 0 & C\\
0& -2A & -B\\
A& B& C
\end{matrix}
\right)=-A(-2AC+B^2)+C(2A^2)=4A^2 C-AB^2=A(4AC-B^2)\ne 0 ,
\end{align*}
we obtain $(\ref{h61})$.

On the other hand, assume that $A=0$. From $B^2-4AC\ne 0$, we have $B\ne 0$. Considering
\begin{align*}
\det\left(
\begin{matrix}
1& B &0&0 \\
0& C & 0& 1\\
0 &A=0& 1& 0\\
1& 0&0&0
\end{matrix}
\right)\ne  0,
\end{align*}
we see that the image of $H^0(X, \Theta_X)\xrightarrow{[\Lambda_0, -]}H^0(X, \wedge^2 \Theta_X)$ is generated by 
\begin{align*}
[\Lambda_0, z \frac{\partial}{\partial w}]=(-Bz^2-2Czw)\frac{\partial}{\partial z}\wedge \frac{\partial}{\partial w},\,\,\,\,\,\,\,\,\,\,\, [\Lambda_0, w\frac{\partial}{\partial z}] = -Bw^2\frac{\partial}{\partial z}\wedge \frac{\partial}{\partial w}
\end{align*}
Since we have
\begin{align*}
\det\left(
\begin{matrix}
0& 0 & -B\\
-B& -2C & 0\\
A=0& B& C
\end{matrix}
\right)\ne 0 ,
\end{align*}
we obtain $(\ref{h61})$.

\end{proof}

We summarize our computation in Table $\ref{h52}$.

\subsubsection{Hopf surfaces of type $\textnormal{III}$}\

Let $X$ be a Hopf surface of type $\textnormal{III}$, and $\Lambda_0=(Azw+Bw^{p+1})\frac{\partial}{\partial z}\wedge \frac{\partial}{\partial w}$ be a Poisson structure on $X$.
\begin{enumerate}
\item if $\Lambda_0=0$, then $\mathbb{H}^1(X, \Theta_X^\bullet)=H^0(X, \wedge^2 \Theta_X)\oplus H^1(X, \Theta_X)$, and $\mathbb{H}^2(X, \Theta_X^\bullet)=H^1(X, \wedge^2 \Theta_X)$ so that $\dim_\mathbb{C} \mathbb{H}^1(X, \Theta_X^\bullet)=5$, and $\dim_\mathbb{C} \mathbb{H}^2(X, \Theta_X^\bullet)=2$.
\item if $\Lambda_0=Bw^{p+1}\frac{\partial}{\partial z}\wedge \frac{\partial}{\partial w}, B\ne 0$, then from $(\ref{h31})$, and $(\ref{h62})$, we obtain
\begin{align}
ker\left(H^0(W,\Theta_W)/im(id-f_*)\xrightarrow{[\Lambda_0,-]} H^0(W,\wedge^2 \Theta_W)/im(id-f_*)\right)&\cong Span_\mathbb{C} \left\langle pz\frac{\partial}{\partial z}+w\frac{\partial}{\partial w}, w^p\frac{\partial}{\partial z} \right\rangle\\
coker(H^0(X, \Theta_X)\xrightarrow{[\Lambda_0, -]}H^0(X, \wedge^2 \Theta_X))&\cong Span_\mathbb{C} \left\langle  zw\frac{\partial}{\partial z}\wedge \frac{\partial}{\partial w} \right\rangle
\end{align}
\item if $\Lambda_0=(Azw+Bw^{p+1})\frac{\partial}{\partial z}\wedge \frac{\partial}{\partial w}, A\ne 0$, then from $(\ref{h31})$ and $(\ref{h62})$, we obtain
\begin{align}
ker\left(H^0(W,\Theta_W)/im(id-f_*)\xrightarrow{[\Lambda_0,-]} H^0(W,\wedge^2 \Theta_W)/im(id-f_*)\right)&\cong Span_\mathbb{C} \left\langle \left( z+\frac{B}{A} w^p \right)\frac{\partial}{\partial z}  , -\frac{pB}{A}w^p\frac{\partial}{\partial z}+w\frac{\partial}{\partial w}  \right\rangle \label{h72}\\
coker(H^0(X, \Theta_X)\xrightarrow{[\Lambda_0, -]}H^0(X, \wedge^2 \Theta_X))&\cong Span_\mathbb{C} \left\langle  (Azw+Bw^{p+1})\frac{\partial}{\partial z}\wedge \frac{\partial}{\partial w} \right\rangle \label{h73}
\end{align}
\end{enumerate}

We summarize our computation in Table $\ref{h52}$.

\subsubsection{Hopf surfaces of type $\textnormal{II}_a$}\

Let $X$ be a Hopf surface of type $\textnormal{II}_a$, and $\Lambda_0=Aw^{p+1}\frac{\partial}{\partial z}\wedge \frac{\partial}{\partial w}$ be a Poisson structure on $X$. 
 From $(\ref{h32})$, and $(id-f_*)(zw\frac{\partial}{\partial z}\wedge \frac{\partial}{\partial w})=\delta^{-p}w^{p+1}\frac{\partial}{\partial z}\wedge \frac{\partial}{\partial w}$,
{\small{\begin{align*}
[Aw^{p+1}\frac{\partial}{\partial z}\wedge \frac{\partial}{\partial w},(\delta^pz-w^p)\frac{\partial}{\partial z}]&=A\delta^p w^{p+1}\frac{\partial}{\partial z}\wedge \frac{\partial}{\partial w}\equiv 0 \,\,\,\,\,\mod( id-f_*)\\
[w^{p+1}\frac{\partial}{\partial z}\wedge \frac{\partial}{\partial w}, w\frac{\partial}{\partial w}]&=-pw^{p+1} \frac{\partial}{\partial z}\wedge \frac{\partial}{\partial w}\equiv 0 \,\,\,\,\,\mod im(id-f_*)
\end{align*}}}
so that we have
\begin{align}
ker\left(H^0(W,\Theta_W)/im(id-f_*)\xrightarrow{[\Lambda_0,-]} H^0(W,\wedge^2 \Theta_W)/im(id-f_*)\right)\cong Span_\mathbb{C} \left\langle (\delta^pz-w^p)\frac{\partial}{\partial z}, w\frac{\partial}{\partial w} \right\rangle \label{h74}
\end{align}
From $(\ref{h63})$, we see that 
\begin{align}
coker(H^0(X, \Theta_X)\xrightarrow{[\Lambda_0,-]} H^0(X, \wedge^2 \Theta_X))\cong Span_\mathbb{C} \left\langle w^{p+1}\frac{\partial}{\partial z}\wedge \frac{\partial}{\partial w} \right\rangle \label{h75}
\end{align}
 Hence from $(\ref{h37})$, we have $\dim_\mathbb{C} \mathbb{H}^1(X, \Theta_X^\bullet)=3$. On the other hand, from $(\ref{h32}),(\ref{h38})$, we have $\dim_\mathbb{C} \mathbb{H}^2(X, \Theta_X^\bullet)=1$.

We summarize our computation in Table $\ref{h52}$.

\subsubsection{Hopf surfaces of type $\textnormal{II}_b$}\

Let $X$ be a Hopf surface of type $\textnormal{II}_b$, and $\Lambda_0=Aw^2\frac{\partial}{\partial z}\wedge \frac{\partial}{\partial w}$ be a Poisson structure on $X$. Then from $(\ref{h33})$, $(\ref{h64})$ and $(id-f_*)(z^2\frac{\partial}{\partial z}\wedge \frac{\partial}{\partial w})=(2\alpha^{-1}zw-\alpha^{-2}w^2)\frac{\partial}{\partial z}\wedge \frac{\partial}{\partial w}$, we have
{\small{\begin{align*}
&[Aw^2\frac{\partial}{\partial z}\wedge \frac{\partial}{\partial w},(\alpha z-w)\frac{\partial}{\partial z}+\alpha w\frac{\partial}{\partial w}]=0\\
&[Aw^2\frac{\partial}{\partial z}\wedge \frac{\partial}{\partial w},(\alpha z-w)\frac{\partial}{\partial w}]=A(-2\alpha zw+w^2)\frac{\partial}{\partial z}\wedge \frac{\partial}{\partial w}\equiv 0 \,\,\,\,\,\,\mod (id-f_*)
\end{align*}}}
so that we have
\begin{align}
ker\left(H^0(W,\Theta_W)/im(id-f_*)\xrightarrow{[\Lambda_0,-]} H^0(W,\wedge^2 \Theta_W)/im(id-f_*)\right)&\cong Span_\mathbb{C}\left\langle  (\alpha z -w)\frac{\partial}{\partial z}+\alpha w \frac{\partial}{\partial w}, (\alpha z-w)\frac{\partial}{\partial w} \right\rangle \label{h76}\\
coker(H^0(X, \Theta_X)\xrightarrow{[\Lambda_0,-]} H^0(X, \wedge^2 \Theta_X))&\cong Span_\mathbb{C} \left\langle w^2\frac{\partial}{\partial z}\wedge \frac{\partial}{\partial w} \right\rangle \label{h77}
\end{align}
 Hence from $(\ref{h37})$, we have $\dim_\mathbb{C} \mathbb{H}^1(X, \Theta_X^\bullet)=3$. On the other hand, from $(\ref{h33}),(\ref{h38})$, we have $\dim_\mathbb{C} \mathbb{H}^2(X, \Theta_X^\bullet)=1$.

We summarize our computation in Table $\ref{h52}$.

\subsubsection{Hopf surfaces of type $\textnormal{II}_c$}\

Let $X$ be a Hopf surface of type $\textnormal{II}_c$, and $\Lambda_0=Azw\frac{\partial}{\partial z}\wedge \frac{\partial}{\partial w}$ be a Poisson structure on $X$. Then from $(\ref{h34})$ and $(\ref{h65})$, we have
\begin{align}
ker\left(H^0(W,\Theta_W)/im(id-f_*)\xrightarrow{[\Lambda_0,-]} H^0(W,\wedge^2 \Theta_W)/im(id-f_*)\right)&\cong Span_\mathbb{C} \left\langle z\frac{\partial}{\partial z}, w\frac{\partial}{\partial w} \right\rangle \label{h80}\\
coker(H^0(X, \Theta_X)\xrightarrow{[\Lambda_0,-]} H^0(X, \wedge^2 \Theta_X))&\cong Span_\mathbb{C} \left\langle zw\frac{\partial}{\partial z}\wedge \frac{\partial}{\partial w} \right\rangle  \label{h81}
\end{align}
 Hence from $(\ref{h37})$, we have $\dim_\mathbb{C} \mathbb{H}^1(X, \Theta_X^\bullet)=3$. On the other hand, from $(\ref{h34}),(\ref{h38})$, we have $\dim_\mathbb{C} \mathbb{H}^2(X, \Theta_X^\bullet)=1$.

We summarize our computation in Table $\ref{h52}$.

\begin{table}
{\small{\begin{center}
\begin{tabular}{| c | c | c | c |c|c|} \hline
Type of Hopf surface $X$ &  Poisson Structure $\Lambda_0$ &    $\dim_\mathbb{C}\mathbb{H}^0(X,\Theta_X^\bullet)$ & $\dim_\mathbb{C}\mathbb{H}^1(X,\Theta_X^\bullet)$ & $\dim_\mathbb{C}\mathbb{H}^2(X,\Theta_X^\bullet)$    \\ \hline
$\textnormal{IV}$   &   0 & 4  &   7        &  3  \\ \hline
$\textnormal{IV}$   &   $(Az^2+Bzw+Cw^2)\frac{\partial}{\partial z}\wedge \frac{\partial}{\partial w}$ &2 &    3  &  1 \\ \cline{6-6}
   & $(A,B,C)\ne 0 $ &  &       &   \\ \hline
$\textnormal{III}$  & 0     & 3  &   5       & 2 \\ \hline
$\textnormal{III}$  & $Bw^{p+1}\frac{\partial}{\partial z}\wedge \frac{\partial}{\partial w}, B\ne 0$     &  2 &   3      & 1 \\ \hline
$\textnormal{III}$  & $(Azw+Bw^{p+1})\frac{\partial}{\partial z} \wedge \frac{\partial}{\partial w}, A\ne 0 $  &  2 &   3       & 1 \\ \hline
$\textnormal{II}_a$  &  $Aw^{p+1}\frac{\partial}{\partial z}\wedge \frac{\partial}{\partial w}$ & 2          &      3    & 1 \\ \hline
$\textnormal{II}_b$  &  $Aw^2\frac{\partial}{\partial z}\wedge \frac{\partial}{\partial w}$   &    2     &     3    & 1 \\ \hline
$\textnormal{II}_c$  &  $Azw\frac{\partial}{\partial z}\wedge \frac{\partial}{\partial w}$  &  2            &   3     & 1 \\ \hline
\end{tabular}
\end{center}}}
\caption{} \label{h52}
\end{table}

\subsection{Obstructed and unobstructed Poisson deformations of Poisson Hopf surfaces}\

In this subsection, we determine obstructedness or unobstructedness of Poisson Hopf surfaces except for two classes of Poisson Hopf surfaces, namely $(X, \Lambda_0=(Az^2+Bzw+Cw^2))\frac{\partial}{\partial z}\wedge \frac{\partial}{\partial w}, 4AC-B^2=0$ where $X$ is a Hopf surface of type $\textnormal{IV}$, and $(X,\Lambda_0=Bw^{p+1}\frac{\partial}{\partial z}\wedge \frac{\partial}{\partial w}), B\ne 0$ where $X$ is a Hopf surface of type $\textnormal{III}$. The author could not determine obstructedness or unobstructedness of those two classes of Poisson Hopf surfaces in Poisson deformations (see Remark $\ref{h95}$).

First we discuss obstructed Poisson Hopf surfaces.
\begin{theorem}
$(X,\Lambda_0=0)$ is obstructed in Poisson deformations if $X$ is a Hopf surface of type $\textnormal{IV}$ or type $\textnormal{III}$.
\end{theorem}
\begin{proof}
Let $(X,\Lambda_0=0)$ be a Poisson Hopf surface of type $\textnormal{IV}$. Then $\mathbb{H}^1(X,\Theta_X^\bullet)=H^0(X,\wedge^2 \Theta_X)\oplus H^1(X,\Theta_X)$. Choose $A,B,C,d,e,f,g\in \mathbb{C}$ such that
\begin{align*}
&[(Az^2+Bzw+Cw^2)\frac{\partial}{\partial z}\wedge \frac{\partial}{\partial w}, (dz+ew)\frac{\partial}{\partial z}+ (fz+gw)\frac{\partial}{\partial w} ]\\
&=\left((-Ad-Bf+gA)z^2+(-2Ae-2Cf)zw+(Cd-Be-Cg)w^2\right)\frac{\partial}{\partial z}\wedge \frac{\partial}{\partial w}\ne 0\\
\end{align*}
Then from Table $\ref{h28}$, $(\ref{h30})$, Lemma $\ref{h90}$ and Lemma $\ref{h91}$, $(X,\Lambda_0)$ is obstructed in Poisson deformations.

On the other hand, let $(X,\Lambda_0=0)$ be a Poisson Hopf surface of type $\textnormal{III}$. Then $H^1(X,\Theta_X^\bullet)=H^0(X,\wedge^2 \Theta_X)\oplus H^1(X,\Theta_X)$. Choose $A,B,d,e,f\in \mathbb{C}$ such that
\begin{align*}
[(Azw+Bw^{p+1})\frac{\partial}{\partial z}\wedge\frac{\partial}{\partial w},(dz+ew^p)\frac{\partial}{\partial z}+fw\frac{\partial}{\partial w}]=(Bd-Ae-pBf)w^{p+1}\frac{\partial}{\partial z}\wedge \frac{\partial}{\partial w}\ne 0
\end{align*}
Then from Table $\ref{h28}$, $(\ref{h31})$, and Lemma $\ref{h90}$ and Lemma $\ref{h91}$, $(X,\Lambda_0)$ is obstructed in Poisson deformations.
\end{proof}

Next we discuss unobstructed Poisson Hopf surfaces. We prove the following theorem which is an extension of \cite{Weh81} Theorem 2 p.28 in the context of Poisson deformations.

\begin{theorem}[compare \cite{Weh81} Theorem 2 p.28] \label{h53}
Let $(X,\Lambda_0)$  be a Poisson Hopf surface except for $(X,\Lambda_0=0), (X, (Az^2+Bzw+Cw^2)\frac{\partial}{\partial z}\wedge \frac{\partial}{\partial w}), 4AC-B^2=0$, where $X$ is a Hopf surface of type $\textnormal{IV}$, and $(X,\Lambda_0=0), (X, \Lambda_0=Bw^{p+1}\frac{\partial}{\partial z}\wedge \frac{\partial}{\partial w} ), B\ne 0$, where $X$ is a Hopf surface of type $\textnormal{III}$. Then $(X,\Lambda_0)$ is unobstructed in Poisson deformations$:$

Explicitly a Poisson analytic family 
\begin{align*}
\pi:(Y,\Lambda_0)\to (S,s_0)
\end{align*}
of deformations of an unobstructed Poisson Hopf surface $\pi^{-1}(s_0)=(X=W/\langle f\rangle,\Lambda_0)$ such that the Poisson Kodaira-Spencer map $\varphi_{s_0}: T_{s_0} S\to \mathbb{H}^1(X, \Theta_X^\bullet)$ is an isomorphism at the distinguished point $s_0$ can be constructed as follows: The base $S$ is a smooth manifold. There exist a holomorphic Poisson structure $\Lambda\in H^0(W\times S, \wedge^2 \Theta_{W\times S/S})$ on $W\times S$ and a biholomorphic Poisson map
\begin{align*}
(W\times S,\Lambda)\to ( W\times S,\Lambda),\,\,\,\,\,(x,s)\mapsto (F(x,s),s),\,\,\,\,\,f(x)=F(x,s_0)
\end{align*}
generating an infinite cyclic group $G$. This group acts properly discontinuous and without fixed points on $W\times S$ and induces
\begin{align*}
\pi:(Y,\Lambda)\to (S,s_0)
\end{align*}
as canonical projection from the factor space $(Y,\Lambda)=((W\times S)/G,\Lambda)$. According to the type of $X$ as defined by Theorem $\ref{h85}$, and Poisson structures as given by Table $\ref{h28}$, the base $S$, the Poisson structure on $W\times S$, and the map
\begin{align*}
F:W\times S\to W
\end{align*}
are given by the explicit formulas in Table $\ref{h50}$ and Table $\ref{h51}$:
\begin{table}
{\tiny{\begin{center}
\begin{tabular}{| c | c | c | c |c|} \hline
Type  &  Poisson &    $S$ &  $F$   \\ \cline{5-5}
       &  structures&     &   \\ \hline
$\textnormal{IV}$   &   0  &  &               \\ \hline       
$\textnormal{IV}$   &   $(Az^2+Bzw+Cw^2)\frac{\partial}{\partial z}\wedge \frac{\partial}{\partial w}$  & $S=\{(\alpha,\beta, t)\in\mathbb{C}^3 :\left(\begin{matrix} \alpha+\beta B & \beta C \\ -\beta A & \alpha \end{matrix}\right)\in GL(2,\mathbb{C})$&  $F(z,w,\alpha,\beta,t)=$         \\ \cline{5-5}
       &  $4AC-B^2\ne 0$ &  $ \text{ has all eigenvalues of modulus}<1\}$   & $((\alpha+\beta B)z+\beta C w,-\beta Az+\alpha w,\alpha,\beta,t)$  \\ \hline
$\textnormal{IV}$   & $(Az^2+Bzw+Cw^2)\frac{\partial}{\partial z}\wedge \frac{\partial}{\partial w}$    &   & \\ \cline{5-5} 
       & $4AC-B^2=0$ &     &   \\ \hline    
$\textnormal{III}$   &  0   &  &                \\ \hline
$\textnormal{III}$   &  $Bw^{p+1}\frac{\partial}{\partial z}\wedge \frac{\partial}{\partial w},B\ne0 $   &  &                \\ \hline
$\textnormal{III}$  &   $(Azw+Bw^{p+1})\frac{\partial}{\partial z}\wedge \frac{\partial}{\partial w},A\ne 0$  &$S=\{(\alpha,\delta,t)\in \mathbb{C}^3:0<|\alpha|<|\delta|<1\}$           & $F(z,w,\alpha,\delta,t)$       \\ \cline{5-5}
 &      &   & $=(\alpha z+\frac{B}{A}(\alpha-\delta^p)w^p,\delta w, \alpha,\delta,t)$      \\ \hline
$\textnormal{II}_a$  & $  Aw^{p+1}\frac{\partial}{\partial z}\wedge \frac{\partial}{\partial w}$   & $S=\{(\alpha,\delta,t)\in \mathbb{C}^3:0<|\alpha|<|\delta|<1\}$          & $F(z,w,\alpha,\delta,t)$   \\ \cline{5-5}
  &      &   & $=(\alpha z+w^p, \delta w,\alpha , \delta,t)   $      \\ \hline
$\textnormal{II}_b$  &  $Aw^2\frac{\partial}{\partial z}\wedge \frac{\partial}{\partial w}$   &$S=\{(\alpha,\beta, t)\in \mathbb{C}^3:\left( \begin{matrix} \alpha & 1\\ \beta &\alpha \end{matrix}   \right)\in GL(2,\mathbb{C})$ has all eigenvalues of modulus $<1\}$        & $F(z,w,\alpha,\beta,t)$   \\ \cline{5-5}
  &      &   &   $=(\alpha z+w,\beta z+\alpha w, \alpha, \beta,t) $    \\ \hline
$\textnormal{II}_c$  & $Azw\frac{\partial}{\partial z}\wedge \frac{\partial}{\partial w}$    &$S=\{(\alpha,\delta,t)\in \mathbb{C}^3:\left(\begin{matrix} \alpha & 0\\ 0& \delta\end{matrix}\right)\in GL(2,\mathbb{C}),0<|\alpha|<|\delta|<1\}$              & $F(z,w,\alpha, \delta, t)=(\alpha z,\delta w,\alpha,\delta,t)$ \\ \hline
\end{tabular}
\end{center}}}
\caption{} \label{h50}
\end{table}

\begin{table}
{\small{\begin{center}
\begin{tabular}{| c | c | c | c |c|} \hline
type &  Poisson &    Poisson structure $\Lambda$ &   Poisson   \\ \cline{5-5}
       &  structures&    on $W\times S$  &   Deformations \\ \hline
$\textnormal{IV}$   &   0  &              &obstructed  \\ \hline       
$\textnormal{IV}$   &   $(Az^2+Bzw+Cw^2)\frac{\partial}{\partial z}\wedge \frac{\partial}{\partial w}$ & $(1+t)(Az^2+Bzw+Cw^2)\frac{\partial}{\partial z}\wedge \frac{\partial}{\partial w}$  & unobstructed \\ \cline{5-5}
 & $4AC-B^2\ne 0$  &      &     \\ \hline
$\textnormal{IV}$   &   $(Az^2+Bzw+Cw^2)\frac{\partial}{\partial z}\wedge \frac{\partial}{\partial w}$ & & The author could not determine \\ \cline{5-5}
 & $4AC-B^2= 0$  &      &     \\ \hline
$\textnormal{III}$  &   0   &     & obstructed   \\ \hline
$\textnormal{III}$  & $Bw^{p+1}\frac{\partial}{\partial z}\wedge \frac{\partial}{\partial w},B\ne0$     &       & The author could not determine   \\ \hline
$\textnormal{III}$  &   $(Azw+Bw^{p+1})\frac{\partial}{\partial z}\wedge \frac{\partial}{\partial w}, A\ne 0$   & $(1+t)(Azw+Bw^{p+1})\frac{\partial}{\partial z}\wedge \frac{\partial}{\partial w}$    & unobstructed   \\ \hline
$\textnormal{II}_a$  &  $Aw^{p+1}\frac{\partial}{\partial z}\wedge \frac{\partial}{\partial w}$ & $(A+t)((\alpha-\delta^p)zw+w^{p+1})\frac{\partial}{\partial z}\wedge \frac{\partial}{\partial w}$ & unobstructed  \\ \hline
$\textnormal{II}_b$  & $Aw^2\frac{\partial}{\partial z}\wedge \frac{\partial}{\partial w}$    & $(A+t)(-\beta z^2+w^2)\frac{\partial}{\partial z}\wedge \frac{\partial}{\partial w}$  & unobstructed   \\ \hline
$\textnormal{II}_c$  & $Azw\frac{\partial}{\partial z}\wedge \frac{\partial}{\partial w}$   & $(A+t)zw\frac{\partial}{\partial z}\wedge \frac{\partial}{\partial w}$   &unobstructed \\ \hline
\end{tabular}
\end{center}}}
\caption{continuation of Table $\ref{h50}$, obstructed and unobstructed Poisson Hopf surfaces}\label{h51}
\end{table}
\end{theorem}

\begin{remark}
We show that the Poisson structure $\Lambda$ on $W\times S$ given by Table $\ref{h51}$ in Theorem $\ref{h53}$ is invariant under the action $G$ so that $\Lambda $ is well-defined on $Y=W\times S/ G$.
\subsubsection{Hopf surfaces of  $\textnormal{IV}$ with $\Lambda_0=(Az^2+Bzw+Cw^2)\frac{\partial}{\partial z}\wedge \frac{\partial}{\partial w}, 4AC-B^2\ne 0$}\

We show that $\Lambda=(1+t)(Az^2+Bzw+Cw^2)\frac{\partial}{\partial z}\wedge \frac{\partial}{\partial w}, 4AC-B^2\ne 0$ is invariant under the action generated by $F(z,w,\alpha,\beta, t)=((\alpha+\beta B)z+\beta C w,-\beta Az+\alpha w,\alpha,\beta,t)$. This follows from $(\ref{h21})$.

\subsubsection{Hopf surfaces of $\textnormal{III}$ with $\Lambda_0=(Azw+Bw^{p+1})\frac{\partial}{\partial z}\wedge \frac{\partial}{\partial w},A\ne 0$}\

We show that $\Lambda=(1+t)(Azw+Bw^{p+1})\frac{\partial}{\partial z}\wedge \frac{\partial}{\partial w}$ is invariant under the action generated by $F(z,w,\alpha,\delta,t)=(\alpha z+\frac{B}{A}(\alpha-\delta^p)w^p,\delta w, \alpha,\delta, t)$. We note that $F^n(z,w,\alpha,\delta, t)=(z',w',\alpha,\delta, t)=(\alpha^n z + \frac{B}{A}(\alpha^n -(\delta^p)^n)w^p,\delta^n w, \alpha, \delta, t)$. Since $\frac{\partial}{\partial z}\wedge \frac{\partial}{\partial w}=\alpha^n \delta^n \frac{\partial}{\partial z'}\wedge \frac{\partial}{\partial w'}$, and
\begin{align*}
(Azw+Bw^{p+1})\alpha^n \delta^n=A\alpha^n \delta^n zw+B(\alpha^n -(\delta^p)^n)\delta^n w^{p+1}+B\delta^{n(p+1)}w^{p+1}
\end{align*}
we see that $\Lambda$ is an invariant bivector field under the action generated by $F$.

\subsubsection{Hopf surfaces of $\textnormal{II}_a$ with $\Lambda_0=Aw^{p+1}\frac{\partial}{\partial z}\wedge \frac{\partial}{\partial w}$}\

We show that $\Lambda=(A+t)((\alpha-\delta^p)zw+w^{p+1})\frac{\partial}{\partial z}\wedge \frac{\partial}{\partial w}$ is invariant under the action generated by $F(z,w,\alpha,\delta,t)=(\alpha z+w^{p},\delta w, \alpha, \delta, t)$. We note that $F^n(z,w, \alpha, \delta, t)=(z',w',\alpha,\delta, t)=(\alpha^n z +\frac{\alpha^n -(\delta^p)^n}{\alpha-\delta^p}w^p, \delta^n w)$. Since $\frac{\partial}{\partial z}\wedge \frac{\partial}{\partial w}=\alpha^n \delta^n \frac{\partial}{\partial z'}\wedge \frac{\partial}{\partial w'}$, and
\begin{align*}
(\alpha-\delta^p)\alpha^n \delta^n zw+(\alpha^n-(\delta^p)^n)\delta^n w^{p+1}+\delta^{n(p+1)}w^{p+1}=((\alpha-\delta^p)zw+w^{p+1})\alpha^n \delta^n,
\end{align*}
we see that $\Lambda$ is an invariant bivector field under the action generated by $F$.

\subsubsection{Hopf surface of $\textnormal{II}_b$ with $\Lambda_0=Aw^2\frac{\partial}{\partial z}\wedge \frac{\partial}{\partial w}$}\

We show that $\Lambda=(A+t)(-\beta z^2+w^2)\frac{\partial}{\partial z}\wedge \frac{\partial}{\partial w}$ is invariant under the action  generated by $F(z,w,\alpha,\beta, t)=(\alpha z+w, \beta z+\alpha w,\alpha, \beta ,t)=(\left(\begin{matrix} \alpha & 1 \\ \beta & \alpha  \end{matrix}\right)\left(\begin{matrix} z \\ w\end{matrix}\right), \alpha, \beta, t)$. We note that

\begin{equation*}
\left(\begin{matrix}
\alpha &1\\
 \beta &\alpha
\end{matrix}\right)
=
-\frac{1}{2\beta}\left(\begin{matrix}
\sqrt{\beta}&1\\
 \beta & -\sqrt{\beta}
\end{matrix}\right)
\left(\begin{matrix}
\alpha+\sqrt{\beta} &0\\
 0 &\alpha-\sqrt{\beta}
\end{matrix}\right)
\left(\begin{matrix}
-\sqrt{\beta} &-1\\
 -\beta & \sqrt{\beta}
\end{matrix}\right)
\end{equation*}

Hence we have
\begin{align*}
\left(\begin{matrix}
\alpha &1\\
 \beta &\alpha
\end{matrix}\right)^n
&=
-\frac{1}{2\beta}\left(\begin{matrix}
\sqrt{\beta}&1\\
 \beta & -\sqrt{\beta}
\end{matrix}\right)
\left(\begin{matrix}
(\alpha+\sqrt{\beta})^n &0\\
 0 &(\alpha-\sqrt{\beta})^n
\end{matrix}\right)
\left(\begin{matrix}
-\sqrt{\beta} &-1\\
 -\beta & \sqrt{\beta}
\end{matrix}\right)\\
&=-\frac{1}{2\beta}
\left(\begin{matrix}
-\beta(\alpha+\sqrt{\beta})^n-\beta(\alpha-\sqrt{\beta})^n &-\sqrt{\beta}(\alpha+\sqrt{\beta})^n+\sqrt{\beta}(\alpha-\sqrt{\beta})^n\\
-\beta\sqrt{\beta}(\alpha+\sqrt{\beta})^n+\beta\sqrt{\beta}(\alpha-\sqrt{\beta})^n & -\beta(\alpha+\sqrt{\beta})^n-\beta(\alpha-\sqrt{\beta})^n
\end{matrix}\right)
\end{align*}

Let $C=-\frac{1}{2\beta}(-\beta(\alpha+\sqrt{\beta})^n-\beta(\alpha-\sqrt{\beta})^n)$ and $D=-\frac{1}{2\beta}(-\sqrt{\beta}(\alpha+\sqrt{\beta})^n+\sqrt{\beta}(\alpha-\sqrt{\beta})^n)$.

Then we have
\begin{equation*}
\left(\begin{matrix}
\alpha & 1\\
\beta & \alpha
\end{matrix}\right)^n
=
\left(\begin{matrix}
C & D\\
\beta D & C
\end{matrix}\right)
\end{equation*}

Then $F^n(z,w,\alpha,\beta, t)=(z',w',\alpha, \beta , t )=(Cz+Dw,\beta D z+Cw,\alpha, \beta, t)$. Since $\frac{\partial}{\partial z}=C\frac{\partial}{\partial z'}+\beta D\frac{\partial}{\partial w'}$ and $\frac{\partial}{\partial w}=D\frac{\partial}{\partial z'}+C\frac{\partial}{\partial w'}$ so that $\frac{\partial}{\partial z}\wedge \frac{\partial}{\partial w}=(C^2-\beta D^2)\frac{\partial}{\partial z'}\wedge \frac{\partial}{\partial w'}$, and
\begin{align*}
-\beta(Cz+Dw)^2+(\beta Dz+Cw)^2&=-\beta(C^2z^2+2CDzw+D^2w^2)+\beta^2D^2z^2+2\beta CDzw+C^2w^2\\
&=-\beta (C^2-\beta D^2)z^2+(C^2-\beta D^2)w^2=(-\beta z^2+w^2)(C^2-\beta D^2)  ,
\end{align*}
we see that $\Lambda$ is an invariant bivector field under the action generated by $F$.

\subsubsection{Hopf surfaces of type $\textnormal{II}_c$ with $\Lambda_0=Azw\frac{\partial}{\partial z}\wedge \frac{\partial}{\partial w}$}\

It is clear that $\Lambda_0=(A+t)zw\frac{\partial}{\partial z}\wedge \frac{\partial}{\partial w}$ is invariant under the action generated by $F(z,w,\alpha,\delta, t)=(\alpha z, \delta w, \alpha, \delta, t)$.
\end{remark}

\subsubsection{ Descriptions of Poisson Kodaira-Spencer maps}\

We show how we can describe the Poisson Kodaira-Spencer map in the Poisson analytic family constructed in Theorem $\ref{h53}$ (see Table $\ref{h50}$ and Table $\ref{h51}$). We prove the following lemma which is an extension of \cite{Weh81} Lemma 3 p.29 in the context of Poisson analytic families.

\begin{lemma}[compare \cite{Weh81} Lemma 3 p.29] \label{h54}
Under the hypotheses of Theorem $\ref{h53}$, assume $s\in S$, let
\begin{align*}
f:\mathbb{C}^2 \to \mathbb{C}^2
\end{align*}
be given by $f(x):=F(x,s)=(F_1(x,s),F_2(x,s))$, where $x=(z,w)\in \mathbb{C}^2$ and set $X:=W/\langle f\rangle$. Let $\Lambda=\Lambda(z,w,s)\frac{\partial }{\partial z}\wedge \frac{\partial}{\partial w}$ be the $G$-invariant bivector field on $W\times S$ with $\Lambda \in H^0(W\times S, \wedge^2 \Theta_{W\times S/S})$ in Theorem $\ref{h53}$, and set $\Lambda_s:=\Lambda (z,w,s)\frac{\partial}{\partial z}\wedge \frac{\partial}{\partial w}$ to be the Poisson structure on $X$ for $s\in S$.  Denote by 
\begin{align*}
\rho:T_s S\to \mathbb{H}^1(X,\Theta_X^\bullet)\cong coker(H^0(X,\Theta_X)\xrightarrow{[\Lambda_s,-]} H^0(W,\wedge^2 \Theta_X))\oplus ker (H^1(X,\Theta_X)\xrightarrow{[\Lambda_s,-]} H^1(X,\Theta_X))
\end{align*}
the Poisson Kodaira-Spencer map of the Poisson analytic family
\begin{align*}
\pi:(Y,\Lambda)\to S
\end{align*}
at the base point $s\in S$. Then a linear map
\begin{align*}
\tau:T_s S\to D\subset H^0(W,\wedge^2\Theta_W)\oplus H^0(W,\Theta_W)
\end{align*}
can be defined by 
\begin{align*}
\tau(v)(x):=\left(a\frac{\partial \Lambda}{\partial s}, \,\,\,\,\,a\cdot {}^t Q (f^{-1}(x))\cdot \frac{\partial}{\partial x}\right)
\end{align*}
where $v=a\frac{\partial}{\partial s}\in T_s S$ and $Q(y):=\frac{\partial F}{\partial s}(y,s)=\left(\begin{matrix}
 \frac{\partial F_1}{\partial s}(y,s) \\
 \frac{\partial F_2}{\partial s}(y,s) \\
 \end{matrix}
 \right)$, which renders the following diagram commutative
\[\xymatrix{
D\subset H^0(W,\wedge^2 \Theta_W) \oplus H^0(W,\Theta_W)   \ar[rrrr]^\sigma
 & & && \mathbb{H}^1(X,\Theta_X^\bullet) \\
& &  \\
T_s S \ar[uu]^{\tau} \ar[rrrruu]^\rho \\      
}\]
Here $D:=\{(B,A)\in H^0(W,\wedge^2 \Theta_W)\oplus H^0(W, \Theta_W): (id-f_*)(B)=[\Lambda_s, A]\}$ a vector space, and $\sigma$ is defined in the following way$:$ for $(B,A)\in D$, from $(\ref{hh12})$ we choose $\{\beta_i\}\in C^0(\tilde{\mathcal{U}}, \Theta_W)$ such that $(id-f_*)(\{\beta_i\})=A$. Then we define
{\small{\begin{align*}
D\subset H^0(W,\wedge^2 \Theta_W) \oplus H^0(W,\Theta_W)&\xrightarrow{\sigma} \mathbb{H}^1(X,\Theta_X^\bullet)\\
(B,A)&\mapsto (B-[\Lambda_s, \{\beta_i\}],\{\beta_j-\beta_i\})\in C^0(\mathcal{U}, \wedge^2 \Theta_X)\oplus C^1(\mathcal{U},\Theta_X)
\end{align*}} }
\end{lemma}

\begin{remark}
We show that $\sigma:D\to \mathbb{H}^1(X,\Theta_X^\bullet)$ in Lemma $\ref{h54}$ is well-defined. First we note that $B-[\Lambda_s,\{\beta_i\}]$ defines an element in $C^0(\mathcal{U},\wedge^2 \Theta_X)$ since $(id-f_*)(B-[\Lambda_s,\{\beta_i\}])=0$, and $(B-[\Lambda_s, \{\beta_i\}],\{\beta_j-\beta_i\})$ defines $1$-cocycle of the \v{C}ech resolusion $(\ref{o6})$ of $\Theta_X^\bullet$.  We show that $\sigma$ is independent of choice of $\{\beta_i\}$. Let $\{\beta_i'\}\in C^0(\tilde{\mathcal{U}},\Theta_W)$ such that $(id-f_*)(\{\beta_i'\})=A$. Since $(id-f_*)(\{\beta_i-\beta_i'\})=0$, we see that $\{\alpha_i:=\beta_i-\beta_i'\}$ defines an element in $C^0(\mathcal{U},\Theta_X)$. Then $[\Lambda_s,\{\alpha_i\}]=B-[\Lambda_s,\{\beta_i'\}]-(B-[\Lambda_s,\{\beta_i\}])$ and $-\delta(\{\alpha_i\})=\{\alpha_i-\alpha_j\}=\{\beta_j'-\beta_i'-(\beta_j-\beta_i)\}$.
\end{remark}

\begin{proof}[Proof of Lemma $\ref{h54}$]

We use the coverings in the subsection $\ref{h3}$ so that 
\begin{align*}
\mathcal{U}=(U_i)_{i\in I} \,\,\,\,\,\text{of $X$},\,\,\,\,\,\,\text{and}\,\,\,\,\,\,\text{$ \tilde{\mathcal{U}}=(\tilde{U}_i)_{i\in I}$ of $W$}
\end{align*}
where 
\begin{align*}
\tilde{U}_i=\bigcup_{m\in \mathbb{Z}}^\bullet f^m (U_i').
\end{align*}

 As in \cite{Weh81} Lemma 3 p.29, if $v=a\frac{\partial}{\partial s}\in T_s S$ is given, we construct a family of projectable vector fields $\tilde{\eta}_i\in \Gamma(\tilde{U}_i,\Theta_W),i\in I$ setting
 $\tilde{\eta}_i:=a\frac{\partial}{\partial s}$ on $U_i'$, and $\tilde{\eta}_i:=G_*^m \tilde{\eta}_i$ on $f^m ( U_i' ) $, $m\in \mathbb{Z}$, where $G(x,s):=(F(x,s),s)$, equivalently
\begin{align*}
G(z,w,s)=(F(z,w,s),s)=(F_1,F_2,s)=(F_1(z,w,s),F_2(z,w,s),s)
\end{align*}

The explicit notation
\begin{equation}
\tilde{\eta}_i
=(\beta_i,a)
\left(
\begin{matrix}
\frac{\partial}{\partial x}\\
\frac{\partial}{\partial s}
\end{matrix}
\right)
=(\beta_i^1,\beta_i^2,a)
\left(
\begin{matrix}
\frac{\partial}{\partial z}\\
\frac{\partial}{\partial w}\\
\frac{\partial}{\partial s}
\end{matrix}
\right)
\end{equation}
shows that we have defined a cochain
\begin{align*}
\beta&=\left(\beta_i\frac{\partial}{\partial x}\right)=\left( \beta_i^1\frac{\partial}{\partial z}+\beta_i^2\frac{\partial}{\partial w}          \right)\in C^0(\tilde{\mathcal{U}},\Theta_W )
\end{align*}
From $G_*\tilde{\eta}_i=\tilde{\eta}_i$ and
\begin{equation}
\partial G:=
\left(\begin{matrix}
\frac{\partial f}{\partial x} & Q\\
0 & 1
\end{matrix}
\right)
=\left(
\begin{matrix}
\frac{\partial F_1}{\partial z} &\frac{\partial F_1}{\partial w} & \frac{\partial F_1}{\partial s}      \\
\frac{\partial F_2}{\partial z} & \frac{\partial F_2}{\partial w} & \frac{\partial F_2}{\partial s}\\
0 & 0& 1
\end{matrix}
\right)
\end{equation}
where 
\begin{align*}
Q:=\frac{\partial F}{\partial s}
=
\left(
\begin{matrix}
\frac{\partial F_1}{\partial s}\\
 \frac{\partial F_2}{\partial s}
\end{matrix}
\right)
\end{align*}
follows the transformation law
\begin{align}
(id-f_*)\beta&=a\cdot {}^t Q(f^{-1})\cdot \frac{\partial}{\partial x} \label{h87}\\
\iff (id-f_*)\left(\beta_i^1\frac{\partial}{\partial z}+\beta_i^2\frac{\partial}{\partial w}\right)&=a\left(  \frac{\partial F_1}{\partial s}(f^{-1})\frac{\partial }{\partial z}+\frac{\partial F_2}{\partial s}(f^{-1})\frac{\partial}{\partial w}      \right) \notag
\end{align}
On the other hand, since $\Lambda=\Lambda(z,w,s)\frac{\partial}{\partial z}\wedge \frac{\partial}{\partial w}$ is invariant under $F(z,w,s)=(F_1,F_2, s)$, we have 
\begin{align}\label{h86}
\Lambda(z,w,s)\left(\frac{\partial F_1}{\partial z}\frac{\partial F_2}{\partial w}-\frac{\partial F_2}{\partial z}\frac{\partial F_1}{\partial w} \right)=\Lambda(F_1,F_2, s).
\end{align}
By applying $a\frac{\partial}{\partial s}$ to both sides of $(\ref{h86})$, we obtain 
\begin{align*}
(id-f_*)\left(a\frac{\partial \Lambda}{\partial s}\right)&=a\frac{\partial \Lambda}{\partial s}-f_*\left(a\frac{\partial \Lambda}{\partial s}\right)=[\Lambda_s, a\cdot {}^t Q(f^{-1})\cdot \frac{\partial}{\partial x}]
\end{align*}

 From $(\ref{h87})$ and $(\ref{h86})$, the $1$-cocycle in \v{C}ech resolution of $\Theta_W^\bullet$
\begin{align*}
\tilde{\theta}&=(\tilde{\theta}_{ij})\in Z^1(\tilde{\mathcal{U}},\Theta_W)\,\,\,\,\,\,\,\,\,\,\,\,\,\text{given by}\,\,\,\,\,\tilde{\theta}_{ij}:= (\tilde{\eta}_j-\tilde{\eta}_i)\frac{\partial}{\partial x}=(\beta_j-\beta_i)\frac{\partial}{\partial x}\\
\tilde{\lambda}&=(\tilde{\lambda}_i)\in C^0(\mathcal{\tilde{U}},\wedge^2 \Theta_W) \,\,\,\,\,\,\,\,\text{given by }\,\,\,\,\tilde{\lambda}_i:=a\frac{\partial \Lambda}{\partial s}-[\Lambda_s,\beta_i\frac{\partial}{\partial x}]
\end{align*}
is invariant with respect to $id-f_*$ so that $(\tilde{\lambda}, \tilde{\theta})$ defines the $1$-cocycle $(\lambda,\theta)\in C^0(\mathcal{U},\wedge^2 \Theta_X)\oplus C^1(\mathcal{U},\Theta_X)$ in \v{C}ech resolution of $\Theta_X^\bullet$ which represents the image of $v=a\frac{\partial}{\partial s}$ under $\rho$
\begin{align*}
(\lambda,\theta)=\rho(v)\in \mathbb{H}^1(X,\Theta_X^\bullet)
\end{align*}
of the Poisson Kodaira-Spencer map at $s$. By the definition of  $\sigma$, we see that
\begin{align*}
\rho(v)=\rho \left(a\frac{\partial}{\partial s}\right)=\sigma\left( a\frac{\partial \Lambda}{\partial s},\,\,\,\,\, a\cdot {}^t Q(f^{-1} ) \cdot \frac{\partial}{\partial x}     \right)\in \mathbb{H}^1(X,\Theta_X^\bullet)
\end{align*}

\end{proof}

\subsection{Proof of Theorem $\ref{h53}$ }\

Now we prove Theorem $\ref{h53}$ for each class of unobstructed Poisson Hopf surfaces given by Table $\ref{h51}$.

\subsection{Hopf surfaces of type $\textnormal{IV}$ with $\Lambda_0=(Az^2+Bzw+Cw^2)\frac{\partial}{\partial z}\wedge \frac{\partial}{\partial w}$ with $4AC-B^2\ne 0$}\

Denote by $(X,\Lambda_0)=(W/\langle f\rangle, (Az^2+Bzw+Cw^2)\frac{\partial}{\partial z}\wedge \frac{\partial}{\partial w})$ where $4AC-B^2\ne 0$ a Poisson Hopf surface of type $\textnormal{IV}$ defined by $f(z,w)=(\alpha z,\alpha w)$ with $0<|\alpha |<1$.

We will show that the Poisson Kodaira-Spencer map of  the Poisson analytic family in Table $\ref{h50}$ and Table $\ref{h51}$ defined by
\begin{align*}
&(W\times S/G, (1+t)(Az^2+Bzw+Cw^2)\frac{\partial}{\partial z}\wedge \frac{\partial}{\partial w}),\\
\text{where}\,\,\,\,\,S=\{(\alpha,\beta, t)\in\mathbb{C}^3 :&\left(\begin{matrix} \alpha+\beta B & \beta C \\ -\beta A & \alpha \end{matrix}\right)\in GL(2,\mathbb{C})\,\,\,\,\,\text{has all eigenvalues of modulus}<1\},\\
 &F(z,w,\alpha,\beta,t)=((\alpha+\beta B) z+\beta C w,-\beta Az+\alpha w,\alpha,\beta,t)
\end{align*}
 is an isomorphism at $s_0=(\alpha, 0,0)$. Indeed, we note that $f^{-1}(z,w)=(\alpha^{-1}z, \alpha^{-1}w)$  and 
 \begin{align*}
 {}^tQ=\left(\begin{matrix}
 \frac{\partial F_1}{\partial \alpha} & \frac{\partial F_2}{\partial \alpha}\\
 \frac{\partial F_1}{\partial \beta} & \frac{\partial F_2}{\partial \beta}\\
 \frac{\partial F_1}{\partial t} & \frac{\partial F_1}{\partial t}
 \end{matrix}
 \right)=
 \left(\begin{matrix}
 z & w\\
 Bz+Cw& -Az\\
 0 & 0
 \end{matrix}
 \right)
 \end{align*}
 so that by Lemma \ref{h54}, $\frac{\partial}{\partial \alpha}\in T_{s_0} S$ is mapped to $\left(0, \alpha^{-1}z\frac{\partial}{\partial z}+\alpha^{-1}w\frac{\partial}{\partial w} \right)\in D$, $\frac{\partial}{\partial \beta}$ is sent to $(0,  (B\alpha^{-1}z+C\alpha^{-1}w)\frac{\partial}{\partial z}-A\alpha^{-1}z\frac{\partial}{\partial z}  )\in D$, and $\frac{\partial}{\partial t}$ is sent to $((Az^2+Bzw+Cw^2,0)\frac{\partial}{\partial z}\wedge \frac{\partial}{\partial w},0)\in D$. Hence $\tau$ induces a biholomorphic map
{\small{ \begin{align*}
 \tau:T_{s_0}S\to M:=Span_\mathbb{C} \left\langle \left((Az^2+Bzw+Cw^2)\frac{\partial}{\partial z}\wedge \frac{\partial}{\partial w} ,0\right) \right\rangle \oplus \left\langle \left(0,z\frac{\partial }{\partial z}+w\frac{\partial}{\partial w}\right),\left(0, (Bz+Cw)\frac{\partial}{\partial z}-Az\frac{\partial}{\partial w}\right) \right\rangle\subset D
 \end{align*}}}
 
 By Lemma $\ref{h54}$, $(\ref{h70})$ and $(\ref{h61})$, the restriction
 \begin{align*}
 \sigma: M\to \mathbb{H}^1(X,\Theta_X^\bullet)\cong coker(H^0(X,\Theta_X)\xrightarrow{[\Lambda_0, -]}H^0(X, \wedge^2 \Theta_X))\oplus ker(H^1(X,\Theta_X)\xrightarrow{[\Lambda_0,-]} H^1(X,\wedge^2 \Theta_X)
 \end{align*}
is an isomorphism so that the Poisson Kodaira-Spencer map is an isomorphism at $s_0$. Hence $(X,\Lambda_0)$ is unobstructed in Poisson deformations .

\subsection{Hopf surfaces of type $\textnormal{III}$ with $\Lambda_0=(Azw+Bw^{p+1})\frac{\partial}{\partial z}\wedge \frac{\partial}{\partial w},A\ne 0$}\

Denote by $(X,\Lambda_0)=(W/\langle f\rangle,(Azw+Bw^{p+1})\frac{\partial}{\partial z}\wedge \frac{\partial}{\partial w}, A\ne 0$ a Hopf surface of type III defined by $f(z,w)=(\delta^p z,\delta w)$ with constants $0<|\delta|<1$ and $p\in \mathbb{N}-\{0,1\}$. We will show that the Poisson Kodaira-Spencer map of the Poisson analytic family in Table $\ref{h50}$ and Table $\ref{h51}$ defined by
\begin{align*}
&(W\times S/G,\Lambda=(1+t)(Azw+Bw^{p+1})\frac{\partial}{\partial z}\wedge \frac{\partial}{\partial w})\\
\text{where}\,\,\,\,\,\,\,S=\{(\alpha,\delta,t)&\in \mathbb{C}^3:|0<|\alpha|<|\delta|<1\},\,\,\,\,\,F((z,w),\alpha,\delta, t)=(\alpha z+\frac{B}{A}(\alpha-\delta^p)w^p,\delta w,\alpha, \delta ,t)
\end{align*}
is an isomorphism at $s_0=(\delta^p,\delta, 0)\in S$. Indeed,  $f^{-1}(z,w)=(\delta^{-p}z,\delta^{-1}w)$ and
\begin{align*}
^{t}Q=\left(\begin{matrix}
 \frac{\partial F_1}{\partial \alpha} & \frac{\partial F_2}{\partial \alpha}\\
 \frac{\partial F_1}{\partial \delta} & \frac{\partial F_2}{\partial \delta}\\
 \frac{\partial F_1}{\partial t} & \frac{\partial F_1}{\partial t}
 \end{matrix}
 \right)=\left(
\begin{matrix}
z+\frac{B}{A}w^p & 0\\
-\frac{pB}{A}\delta^{p-1}w^p & w\\
0 & 0
\end{matrix}
\right)
\end{align*}
so that by Lemma \ref{h54}, $\frac{\partial}{\partial \alpha}$ is sent to $(0,((\delta^{-p}z)+\frac{B}{A}(\delta^{-1}w)^p)\frac{\partial}{\partial z})=(0,\delta^{-p}(z+\frac{B}{A}w^p)\frac{\partial}{\partial z})\in D$, $\frac{\partial}{\partial \delta}$ is sent to $(0,-\frac{pB}{A}\delta^{p-1}(\delta^{-1}w)^p \frac{\partial}{\partial z}+\delta^{-1} w\frac{\partial}{\partial w})=(0,\delta^{-1}(-\frac{pB}{A}w^p\frac{\partial}{\partial z}+w\frac{\partial}{\partial w}))\in D$ and $\frac{\partial}{\partial t}$ is sent to $((Azw+Bw^{p+1})\frac{\partial}{\partial z}\wedge \frac{\partial}{\partial w},0)\in D$.
Hence  $\tau$ induces a biholomorphic map
{\small{\begin{align*}
\tau:T_{s_0}S\to M:=Span_\mathbb{C}\left\langle \left(  (Azw+Bw^{p+1})\frac{\partial}{\partial z} \wedge \frac{\partial}{\partial w} ,0\right)  \right\rangle\oplus \left\langle \left(0,\left(z+\frac{B}{A}w^p\right)\frac{\partial}{\partial z}\right),\left( 0,-\frac{pB}{A}w^p\frac{\partial}{\partial z}+w \frac{\partial}{\partial w} \right) \right\rangle \subset D
\end{align*}}}
 By Lemma $\ref{h54}$, $(\ref{h72})$ and $(\ref{h73})$, the restriction
 \begin{align*}
 \sigma:  M  \to \mathbb{H}^1(X,\Theta_X^\bullet)\cong coker(H^0(X,\Theta_X)\xrightarrow{[\Lambda_0, -]}H^0(X, \wedge^2 \Theta_X))\oplus ker(H^1(X,\Theta_X)\xrightarrow{[\Lambda_0,-]} H^1(X,\wedge^2 \Theta_X))
 \end{align*}
is an isomorphism so that the Poisson Kodaira-Spencer map is an isomorphism at $s_0$. Hence $(X,\Lambda_0)$ is unobstructed in Poisson deformations. 

\subsection{Hopf surfaces of type $\textnormal{II}_a$ with $\Lambda_0=Aw^{p+1}\frac{\partial}{\partial z}\wedge \frac{\partial}{\partial w}$}\

Denote by $(X,\Lambda_0)=(W/\langle f\rangle, Aw^{p+1}\frac{\partial}{\partial z}\wedge \frac{\partial}{\partial w})$ a Poisson Hopf surface of type $\textnormal{II}_a$ defined by $f(z,w)=(\delta^p z+w^p, \delta w)$ with constants $0<|\delta|<1$ and $p\in \mathbb{N}-\{0,1\}$.

We will show that the Poisson Kodaira-Spencer map of the Poisson analytic family in Table $\ref{h50}$ and Table $\ref{h51}$ defined by
\begin{align*}
&(W\times S/G, (A+t)((\alpha-\delta^p)zw+w^{p+1})\frac{\partial}{\partial z}\wedge \frac{\partial}{\partial w})\\
\text{where}\,\,\,\,\, S=\{(\alpha,\delta, t)&\in \mathbb{C}^2:  0<|\alpha|<|\delta |<1\},\,\,\,\,\,\,\,\,
F((z,w),(\alpha,\delta,t))=(\alpha z+w^p,\delta w, \alpha,\delta, t)
\end{align*}
is an isomorphism at $s_0=(\delta^p, \delta, 0)\in S$. Indeed, we have $f^{-1}(z,w)=(\delta^{-p}z-\delta^{-2p}w^p,\delta^{-1}w)$ and
\begin{align*}
{}^tQ=\left(\begin{matrix}
 \frac{\partial F_1}{\partial \alpha} & \frac{\partial F_2}{\partial \alpha}\\
 \frac{\partial F_1}{\partial \delta} & \frac{\partial F_2}{\partial \delta}\\
 \frac{\partial F_1}{\partial t} & \frac{\partial F_1}{\partial t}
 \end{matrix}
 \right)=\left(\begin{matrix}
z &0 \\
0 & w \\
0 & 0
\end{matrix}\right)
\end{align*}
By Lemma $\ref{h54}$, $\frac{\partial}{\partial \alpha}$ is sent to $\left(Azw\frac{\partial}{\partial z}\wedge \frac{\partial}{\partial w}, (\delta^{-p}z-\delta^{-2p}w^p)\frac{\partial}{\partial z}\right)\in D$, $\frac{\partial}{\partial \delta}$ is sent to $\left(-pA\delta^{p-1} zw\frac{\partial}{\partial z}\wedge \frac{\partial}{\partial w},\delta^{-1}w\frac{\partial}{\partial w}     \right)\in D$, and $\frac{\partial}{\partial t}$ is sent to $(w^{p+1}\frac{\partial}{\partial z}\wedge \frac{\partial}{\partial w} ,0)\in D $.
Hence $\tau$ induces a biholomorphic map
{\tiny{\begin{align*}
\tau: T_{s_0} S\to M:= Span_\mathbb{C}\left\langle \left(Azw\frac{\partial}{\partial z}\wedge \frac{\partial}{\partial w}, (\delta^{-p}z-\delta^{-2p}w^p)\frac{\partial}{\partial z}\right) \right\rangle \oplus \left\langle \left(-pA\delta^{p-1} zw\frac{\partial}{\partial z}\wedge \frac{\partial}{\partial w},\delta^{-1}w\frac{\partial}{\partial w}     \right)   \right\rangle \oplus  \left\langle \left(w^{p+1}\frac{\partial}{\partial z}\wedge \frac{\partial}{\partial w} ,0 \right) \right\rangle\subset D 
\end{align*}}}
By Lemma $\ref{h54}$, $(\ref{h74})$ and $(\ref{h75})$, the restriction
\begin{align*}
\sigma |_M \to \mathbb{H}^1(X,\Theta_X^\bullet)\cong coker(H^0(X,\Theta_X)\xrightarrow{[\Lambda_0, -]}H^0(X, \wedge^2 \Theta_X))\oplus ker(H^1(X,\Theta_X)\xrightarrow{[\Lambda_0,-]} H^1(X,\wedge^2 \Theta_X)).
\end{align*}
is an isomorphism so that the Poisson Kodaira-Spencer map is an isomorphism at $s_0$. Hence $(X,\Lambda_0)$ is unobstructed in Poisson deformations.

\subsection{Hopf surfaces of type $\textnormal{II}_b$ with $\Lambda_0=Aw^2\frac{\partial}{\partial z}\wedge \frac{\partial}{\partial w}$ }\

Denote by $(X,\Lambda_0)=(W/\langle f\rangle, Aw^2\frac{\partial}{\partial z}\wedge \frac{\partial}{\partial w})$ a Poisson Hopf surface of type $\textnormal{II}_b$ defined by $f(z,w)=(\alpha z+w,\alpha w)$.

We will show that the Poisson Kodaira-Spencer map of the Poisson analytic family in Table $\ref{h50}$ and Table $\ref{h51}$ defined by
\begin{align*}
&(W\times S/G,(A+t)(-\beta z^2+w^2)\frac{\partial}{\partial z}\wedge \frac{\partial}{\partial w}) \\
\text{where}\,\,\,\,\,\,\,S=\{(\alpha,\beta,t )&\in \mathbb{C}^3:\left(\begin{matrix} \alpha  & 1 \\ \beta & \alpha \end{matrix}\right)\in GL(2,\mathbb{C}) \text{ has all eigenvalues of modulus $<1$} \} \\  
&F(z,w,\alpha,\beta, t)=(\alpha z+w,\beta z+\alpha w, \alpha,\beta, t)
\end{align*}
is an isomorphism at $s_0=(\alpha, 0, 0)$. Indeed, we have $f^{-1}(z,w)=( \alpha^{-1} z-\alpha^{-2} w , \alpha^{-1}w)$ and
\begin{align*}
^{t}Q=\left(\begin{matrix}
 \frac{\partial F_1}{\partial \alpha} & \frac{\partial F_2}{\partial \alpha}\\
 \frac{\partial F_1}{\partial \beta} & \frac{\partial F_2}{\partial \beta}\\
 \frac{\partial F_1}{\partial t} & \frac{\partial F_1}{\partial t}
 \end{matrix}
 \right)=\left( 
\begin{matrix}
z & w \\
0 & z \\
0 & 0
\end{matrix}
\right)
\end{align*}
so that by Lemma \ref{h54}, $\frac{\partial}{\partial \alpha}$ is sent to $(0,(\alpha^{-1}z-\alpha^{-2}w)\frac{\partial}{\partial z}+\alpha^{-1} w\frac{\partial}{\partial w})\in D$, $\frac{\partial}{\partial \beta}$ is sent to $(-Az^2\frac{\partial}{\partial z}\wedge \frac{\partial}{\partial w}, (\alpha^{-1}z-\alpha^{-2}w)\frac{\partial}{\partial w})\in D$, and $\frac{\partial}{\partial t}$ is sent to $(w^2\frac{\partial}{\partial z}\wedge \frac{\partial}{\partial w},0)\in D$. Hence $\tau$ induces a biholomorphic map
{\tiny{\begin{align*}
\tau:T_{s_0} S\to M:=Span_\mathbb{C}\left\langle \left(0,\left(\alpha^{-1}z-\alpha^{-2}w\right)\frac{\partial}{\partial z}+\alpha^{-1} w\frac{\partial}{\partial w}\right) \right\rangle \oplus \left\langle \left(-Az^2\frac{\partial}{\partial z}\wedge \frac{\partial}{\partial w}, \left(\alpha^{-1}z-\alpha^{-2}w\right)\frac{\partial}{\partial w}\right) \right\rangle \oplus \left\langle \left(w^2\frac{\partial}{\partial z}\wedge \frac{\partial}{\partial w},0\right) \right\rangle \subset D
\end{align*}}}
By Lemma $\ref{h54}$, $(\ref{h76})$ and $(\ref{h77})$, the restriction
\begin{align*}
\sigma: M\to \mathbb{H}^1(X,\Theta_X^\bullet) \cong  coker(H^0(X,\Theta_X)\xrightarrow{[\Lambda_0, -]}H^0(X, \wedge^2 \Theta_X))\oplus ker(H^1(X,\Theta_X)\xrightarrow{[\Lambda_0,-]} H^1(X,\wedge^2 \Theta_X))
\end{align*}
is an isomorphism so that the Poisson Kodaira-Spencer map is an isomorphism at $s_0$. Hence $(X,\Lambda_0)$ is unobstructed in Poisson deformations.

\subsection{Hopf surfaces of type $\textnormal{II}_c$ with $\Lambda_0=Azw\frac{\partial}{\partial z}\wedge \frac{\partial}{\partial w}$ }\

Denote by $(X,\Lambda_0)=(W/\langle f \rangle, Azw\frac{\partial}{\partial z}\wedge \frac{\partial}{\partial w})$ a Poisson Hopf surface $\textnormal{II}_c$ defined by $f(z,w)=(\alpha z, \delta w)$ with $0<|\alpha| <|\delta |<1$. We will show that the Poisson Kodaira-Spencer map of the Poisson Kodaira-Spencer map of the Poisson analytic family in Table $\ref{h50}$ and Table $\ref{h51}$ defined by
\begin{align*}
&(W\times S/G, (A+t)zw\frac{\partial}{\partial z}\wedge \frac{\partial}{\partial w})\\
\text{where}\,\,\,\,\,S=\{(\alpha,\delta, t)&\in \mathbb{C}:0<|\alpha| <|\delta| <1\},\,\,\,\,\,F(z,w,\alpha,\delta, t)=(\alpha z,\delta w, \alpha,\delta, t)
\end{align*}
is an isomorphism at $s_0=(\alpha,\delta, 0)$. Indeed, we have $f^{-1}(z,w)=(\alpha^{-1}z , \delta^{-1}w)$ and
\begin{align*}
^{t}Q=\left(\begin{matrix}
 \frac{\partial F_1}{\partial \alpha} & \frac{\partial F_2}{\partial \alpha}\\
 \frac{\partial F_1}{\partial \delta} & \frac{\partial F_2}{\partial \delta}\\
 \frac{\partial F_1}{\partial t} & \frac{\partial F_1}{\partial t}
 \end{matrix}
 \right)=\left(\begin{matrix}
z & 0\\
0 & w \\
0 & 0
\end{matrix}\right)
\end{align*}
so that by Lemma \ref{h54},  $\frac{\partial}{\partial \alpha}$ is sent to $(0,\alpha^{-1}z\frac{\partial}{\partial z})\in D$, $\frac{\partial}{\partial \delta}$ is sent to $(0,\delta^{-1}w \frac{\partial}{\partial w})\in D$, and $\frac{\partial}{\partial t}$ is sent to $(zw\frac{\partial}{\partial z}\wedge \frac{\partial}{\partial w},0)\in D$. Hence $\tau$ induces a biholomorphic map
\begin{align*}
\tau:T_{s_0} S\to M:=span_\mathbb{C}\left\langle \left(0,\alpha^{-1}z\frac{\partial}{\partial z}\right) \right\rangle \oplus \left\langle \left(0,\delta^{-1}w \frac{\partial}{\partial w}\right) \right\rangle \oplus \left\langle \left(zw\frac{\partial}{\partial z}\wedge \frac{\partial}{\partial w},0\right) \right\rangle \subset D
\end{align*}
By Lemma $\ref{h54}$, $(\ref{h80})$, and $(\ref{h81})$,  the restriction
\begin{align*}
\sigma:M\to \mathbb{H}^1(X,\Theta_X^\bullet) \cong  coker(H^0(X,\Theta_X)\xrightarrow{[\Lambda_0, -]}H^0(X, \wedge^2 \Theta_X))\oplus ker(H^1(X,\Theta_X)\xrightarrow{[\Lambda_0,-]} H^1(X,\wedge^2 \Theta_X))
\end{align*}
is an isomorphism so that the Poisson Kodaira-Spencer map is an isomorphism at $s_0$. Hence $(X,\Lambda_0)$ is unobstructed in Poisson deformations.

\begin{remark}\label{h95}
The author could not determine unobstructedness or obstructedness in Poisson deformations for the cases $(\textnormal{IV}, \Lambda_0=(Az^2+Bzw+Cw^2)\frac{\partial}{\partial z}\wedge \frac{\partial}{\partial w}\ne 0)$ with $4AC-B^2=0$ , and $(\textnormal{III},\Lambda_0=Bw^{p+1}\frac{\partial}{\partial z}\wedge \frac{\partial}{\partial w})$ with $B\ne 0$. As one can check, we cannot use Lemma $\ref{h90}$ in order to show obstructedness.

In each case, when we ignore Poisson structures, we can construct a complex analytic family such that the Kodaira-Spencer map is an isomorphism to $ker(H^1(X,\Theta_X)\xrightarrow{[\Lambda_0,-]} H^1(X,\wedge^2 \Theta_X))$ in the following:
\begin{enumerate}
\item Let $(\textnormal{IV}, \Lambda_0=(Az^2+Bzw+Cw^2)\frac{\partial}{\partial z}\wedge \frac{\partial}{\partial w}\ne 0)$ with $4AC-B^2=0$. Consider $(W\times S', S')/\langle F\rangle$ where $F(z,w,\alpha,\beta)= ( (\alpha+\beta B)z+\beta Cw, -\beta Az+\alpha w,\alpha,\beta)$ and $S'=\{(\alpha,\beta)\in \mathbb{C}^2:\left(\begin{matrix} \alpha+\beta B & \beta C \\ -\beta A & \alpha \end{matrix}\right)\in GL(2,\mathbb{C}) \text{ has all eigenvalues of modulus}<1\}$. Then the Kodaira-Spencer map of the complex analytic family is isomorphic to $ker(H^1(X,\Theta_X)\xrightarrow{[\Lambda_0,-]} H^1(X,\wedge^2 \Theta_X))$. The question is whether we can construct a Poisson analytic family using $F(z,w,\alpha,\beta)$ in a way that the Poisson Kodaira-Spencer map is an isomorphism at the distinguished point. But this is impossible. For example let us consider $\Lambda_0=z^2\frac{\partial}{\partial z}\wedge \frac{\partial}{\partial w}$. We may assume that $F(z,w)=(\alpha z,\beta z+\alpha w)$. We note that $\left(\begin{matrix} \alpha & 0 \\ \beta & \alpha \end{matrix}\right)^n=\left(\begin{matrix} \alpha^n & 0\\ n \alpha^{n-1}\beta  &\alpha^n \end{matrix} \right)$. Hence the invariant bivector field is of the form $Dz^2\frac{\partial}{\partial z}\wedge \frac{\partial}{\partial w}$ under the action generated by $F$. Let us consider the Poisson analytic family defined by
\begin{align}
&(W\times S/\langle F\rangle,  g(t)z^2\frac{\partial}{\partial z}\wedge \frac{\partial}{\partial w},S),\,\,\,\,\, \notag\\
&F:(z,w,\alpha,\beta,t)\mapsto ( (\alpha z,\beta z+ \alpha w,\alpha,\beta,t)\label{a1}\\
&S=\{(\alpha,\beta,t)\in \mathbb{C}^3:\left(\begin{matrix} \alpha & 0 \\ \beta & \alpha \end{matrix}\right)\in GL(2,\mathbb{C}) \text{ has all eigenvalues of modulus}<1\}, g(0)=A \notag
\end{align}
But in this case, the Poisson Kodaira-Spencer map is not an isomorphism at $0$ since $\frac{\partial}{\partial t}$ is sent to $(g'(0)z^2\frac{\partial}{\partial z}\wedge \frac{\partial}{\partial w},0)\in D$ which defines the $0$ class in $\mathbb{H}^1(X,\Theta_X^\bullet)$. We cannot construct a Poisson analytic family of deformations of $(X,\Lambda_0)$ such that the Poisson Kodaira-Spencer map is an isomorphism by using the $F(z,w,\alpha,\beta,t)$ in $(\ref{a1})$.
\item Let $(\textnormal{III},\Lambda_0=Bw^{p+1}\frac{\partial}{\partial z}\wedge \frac{\partial}{\partial w})$ with $B\ne 0$. Consider $(W\times S',S')/\langle F \rangle $ where $F(z,w,\delta,\lambda)=(\delta^pz+\lambda w^p, \delta w,\delta, \lambda)$ and $S'=\{(\delta,\lambda)\in \mathbb{C}^2:0<|\delta|<1\}$. Then the Kodaira-Spencer map of the complex analytic family is isomorphic to $ker(H^1(X,\Theta_X)\xrightarrow{[\Lambda_0,-]} H^1(X,\wedge^2 \Theta_X))$. However we cannot construct a Poisson analytic family using $F(z,w,\delta, \lambda)$ such that the Poisson Kodaira-Spencer map is an isomorphism at $(\delta, 0)$. We note that  $F^n$ is given by
\begin{align*}
(z,w, \delta, \lambda)\mapsto (z',w',\delta, \lambda)=(\delta^{np}z+\lambda n \delta^{(n-1)p}w^p,\delta^n w,\delta, \lambda)
\end{align*}
Since $\frac{\partial}{\partial z}=\delta^{np}\frac{\partial}{\partial z'}$, and $\frac{\partial}{\partial w}=p\lambda n \delta^{(n-1)p}w^{p-1}\frac{\partial}{\partial z'}+\delta^n \frac{\partial}{\partial w'}$ and so $\frac{\partial}{\partial z}\wedge \frac{\partial}{\partial w}= \delta^{p(n+1)}\frac{\partial}{\partial z'}\wedge \frac{\partial}{\partial w'}$, the invariant bivector field under the action generated by $F$ is of the form $Dw^{p+1}\frac{\partial}{\partial z}\wedge \frac{\partial}{\partial w}$. Let us consider the Poisson analytic family
\begin{align}
&(W\times S/\langle F\rangle,  g(t)w^{p+1}\frac{\partial}{\partial z}\wedge \frac{\partial}{\partial w},S),\,\,\,\,\, \notag\\
&F:(z,w,\delta,\lambda,t)\mapsto ( \delta^p z+\lambda w^p, \delta w,\delta,\lambda, t) \label{a10} \\
&S=\{(\delta,\lambda,t)\in \mathbb{C}^3: 0<|\delta| <1\}, g(0)=A \notag
\end{align}
But in this case, the Poisson Kodaira-Spencer map is not isomorphism at $s=(\delta,1 ,0)$ since $\frac{\partial}{\partial t}$ is sent to $(g'(0)w^{p+1}\frac{\partial}{\partial z}\wedge \frac{\partial}{\partial w},0)\in D$ which defines the $0$ class in $\mathbb{H}^1(X,\Theta_X^\bullet)$. We cannot construct a Poisson analytic family of deformations of $(X,\Lambda_0)$ such that the Poisson Kodaira-Spencer map is an isomorphism by using the $F(z,w,\delta,\lambda, t)$ in $(\ref{a10})$.

\end{enumerate}
\end{remark}

\section{Product of two nonsingular projective curves}\label{section5}

In this section, we study Poisson deformations of $X=C_1\times C_2$ where $C_1$ and $C_2$ are nonsingular projective curves with genera $g(C_1)=g_1$ and $g(C_2)=g_2$ respectively. Since $X$ has only trivial Poisson structure for $g_1\geq 2$ or $g_2\geq 2$, we only consider $g_1\leq 1$ and $g_2\leq 1$. In this case, we show that $(X=C_1\times C_2,\Lambda_0)$ are unobstructed in Poisson deformations except for $(E\times \mathbb{P}_\mathbb{C}^1,\Lambda_0=0)$ where $E$ is an elliptic curve.

\subsection{Description of cohomology groups $H^i(X,\wedge^j \Theta_X)$ on $X$}\

Let $C_1$ and $C_2$ be two nonsingular projective curves with genera $g(C_1)=g_1$ and $g(C_2)=g_2$ respectively. Let $X:=C_1\times C_2$. We describe the cohomology groups $H^i(X,\wedge^j \Theta_X),i=0,1,2,j=1,2$. Let $\pi_1:X\to C_1$ and $\pi_2:X\to C_2$ be two natural projections. Then we have $\Theta_X=\pi_1^*\Theta_{C_1}\oplus \pi_2^* \Theta_{C_2}$ and $\wedge^2 \Theta_X=\pi_1^*\Theta_{C_1}\otimes_{\mathcal{O}_X} \pi_2^*\Theta_{C_2}$. By K\"{u}nneth formula, we have
\begin{align}\label{formula}
H^0(X,\Theta_X)&\cong H^0(C_1,\Theta_{C_1})\oplus H^0(C_2,\Theta_{C_2})\\
H^1(X,\Theta_X)&\cong H^1(C_1,\Theta_{C_1})\oplus \left( H^0(C_1,\Theta_{C_1})\otimes H^1(C_2,\mathcal{O}_{C_2})\right)\oplus\left(H^1(C_1,\mathcal{O}_{C_1})\otimes H^0(C_2,\Theta_{C_2}) \right)\oplus H^1(C_2,\Theta_{C_2}) \notag\\
H^2(X,\Theta_X)&\cong \left( H^1(C_1,\Theta_{C_1})\otimes H^1(C_2,\mathcal{O}_{C_2}) \right)\oplus \left(H^1(C_1,\mathcal{O}_{C_1})\otimes H^1(C_2,\Theta_{C_2})\right)\notag\\
H^0(X,\wedge^2 \Theta_X)&\cong  H^0(C_1,\Theta_{C_1})\otimes H^0(C_2,\Theta_{C_2})  \notag     \\
H^1(X,\wedge^2 \Theta_X)&\cong  \left( H^1(C_1,\Theta_{C_1})\otimes H^0(C_2,\Theta_{C_2})    \right) \oplus \left( H^0(C_1,\Theta_{C_1})\otimes H^1(C_2,\Theta_{C_2}) \right)  \notag    \\
H^2(X,\wedge^2 \Theta_X)&\cong  H^1(C_1,\Theta_{C_1})\otimes H^1(C_2,\Theta_{C_2})          \notag
\end{align}

\subsection{Poisson deformations of $X=C_1\times C_2$}\

If $g_1\geq 2$ or $g_2\geq 2$, we have $H^0(X,\wedge^2 \Theta_X)=0$ so that we only consider $(g_1,g_2)=(0,0),(0,1),(1,1)$.
\begin{enumerate}
\item When $g_1=0,g_2=0$, we have $X=\mathbb{P}_\mathbb{C}^1\times \mathbb{P}_\mathbb{C}^1\cong F_0$. In this case,  $(X,\Lambda_0)$ is unobstructed in Poisson deformations for any holomorphic Poisson structure as we showed in Table $\ref{ruled}$.

\item When $g_1=1,g_2=1$, we have $X=E_1\times E_2$, where $E_1,E_2$ are elliptic curves. In this case, $(X,\Lambda_0)$ is unobstructed in Poisson deformations by Example $\ref{torus}$.

\item When $g_1=1,g_2=0$, we have $X=E\times \mathbb{P}_\mathbb{C}^1$ where $E$ is an elliptic curve. In the next subsection, we show that $(X,\Lambda_0)$ is unobstructed in Poisson deformations for $\Lambda_0\ne 0$, and $(X,\Lambda_0=0)$ is obstructed in Poisson deformations.
\end{enumerate}

\subsection{Poisson deformations of $E\times \mathbb{P}_\mathbb{C}^1$ where $E$ is an elliptic curve}\

Let $E=\mathbb{C}/\mathcal{G}$ be an elliptic curve, where $\mathcal{G}$ is the free abelian group generated by non-degenerate period $\omega_j,j=1,2$ and where,  if $z$ is the coordinate of $\mathbb{C}$, $\omega_j\in \mathcal{G}$ operates on $\mathcal{G}$ by sending $z$ into $z+\omega_j$. Let $X:=E\times \mathbb{P}_\mathbb{C}^1$, and $\xi$ is the inhomogenous coordinate of $\mathbb{P}_\mathbb{C}^1$. Then from $(\ref{formula})$ we have
\begin{enumerate}
\item $\dim_\mathbb{C} H^0(X,\Theta_X)=4$ and any element of $H^0(X,\Theta_X)$ has an unique representation of the form 
\begin{align}\label{c51}
a_0\frac{\partial}{\partial z}+(a_1+a_2\xi+a_3\xi^2)\frac{\partial}{\partial \xi},\,\,\,\,\,\text{where}\,\,\,\,\,a_0,a_1,a_2,a_3\in \mathbb{C}.
\end{align} 
\item $\dim_\mathbb{C} H^1(X,\Theta_X)=4$ and any element of $H^1(X,\Theta_X)$ has an unique representation of the form
\begin{align}\label{c52}
x_0\frac{\partial}{\partial z}d\bar{z}+(x_1+x_2\xi+x_3\xi^2)\frac{\partial}{\partial \xi}d\bar{z},\,\,\,\,\,\text{where}\,\,\,\,\,x_0,x_1,x_2,x_3\in \mathbb{C}.
\end{align}
\item $\dim_\mathbb{C} H^2(X,\Theta_X)=0.$
\item $\dim_\mathbb{C} H^0(X,\wedge^2 \Theta_X)=3$ and any element of $H^0(X, \wedge^2 \Theta_X)$ has an unique representation of the form
\begin{align} \label{c1}
(b_0+b_1\xi+b_2\xi^2)\frac{\partial}{\partial z}\wedge \frac{\partial}{\partial \xi},\,\,\,\,\,\text{where}\,\,\,\,\,b_0,b_1,b_2\in \mathbb{C}.
\end{align}
\item $\dim_\mathbb{C} H^1(X,\wedge^2 \Theta_X)=3$ and any element of $H^1(X, \wedge^2 \Theta_X)$ has an unique representation of the form
\begin{align} \label{c2}
(x_0+x_1\xi+x_2\xi^2)\frac{\partial}{\partial z}\wedge \frac{\partial}{\partial \xi}d\bar{z},\,\,\,\,\,\text{where}\,\,\,\,\, x_0,x_1,x_2\in \mathbb{C}.
\end{align}
\item $\dim_\mathbb{C} H^2(X,\wedge^2 \Theta_X)=0$.
\end{enumerate}

By $(\ref{c1})$, let us consider a holomorphic Poisson structure on $X=E\times \mathbb{P}_\mathbb{C}^1$ given by
\begin{align*}
\Lambda_0=(A+B\xi+C\xi^2)\frac{\partial}{\partial z}\wedge \frac{\partial}{\partial \xi}
\end{align*}
for some constant $A,B,C\in \mathbb{C}$. As in Lemma $\ref{r4}$, we have

\begin{lemma}\label{c9}
Let $(X=E\times \mathbb{P}_\mathbb{C}^1,\Lambda_0)$ where $E$ is an elliptic curve.  Then 
\begin{align*}
\mathbb{H}^0(X,\Theta_X^\bullet)&\cong ker(H^0(X,\Theta_X)\xrightarrow{[\Lambda_0,-]} H^0(X,\wedge^2 \Theta_X))\\
\mathbb{H}^1(X, \Theta_{X}^\bullet)&\cong coker(H^0(X, \Theta_{X})\xrightarrow{[\Lambda_0,-]} H^0(X, \wedge^2 \Theta_{X}))\oplus ker(H^1(X, \Theta_{X})\xrightarrow{[\Lambda_0,-]} H^1(X,\wedge^2 \Theta_{X}))\\
\mathbb{H}^2(X,\Theta_{X}^\bullet)&\cong coker(H^1(X,\Theta_{X})\xrightarrow{[\Lambda_0,-]}H^1(X,\wedge^2 X))
\end{align*}
 $(X,\Lambda_0)$ is obstructed in Poisson deformations if for some $a , b$ where $a\in H^0(X,\wedge^2 \Theta_X)$ and $b\in ker(H^1(X, \Theta_{X})\xrightarrow{[\Lambda_0,-]} H^1(X,\wedge^2 \Theta_{X}))$, under the following map
\begin{align*}
[-,-]:H^0(X, \wedge^2 \Theta_{X})\times H^1(X, \Theta_{X})\to H^1(X, \wedge^2 \Theta_{X})
\end{align*}
$[a,b]\in H^1(X,\wedge^2 \Theta_{X})$ is not in the image of $ H^1(X, \Theta_{X})\xrightarrow{[\Lambda_0,-]} H^1(X,\wedge^2 \Theta_{X}) $.
\end{lemma}

We will describe the first cohomology group 
\begin{align}\label{c7}
\mathbb{H}^1(X,\Theta_X^\bullet)\cong coker(H^0(X,\Theta_X)\xrightarrow{[\Lambda_0,-]} H^0(X,\wedge^2 \Theta_X))\oplus ker(H^1(X,\Theta_X)\xrightarrow{[\Lambda_0,-]}H^1(X,\wedge^2 \Theta_X)).
\end{align} 

Let us describe $H^1(X,\Theta_X)\xrightarrow{[\Lambda_0,-]} H^1(X,\wedge^2 \Theta_X)$ in $(\ref{c7})$. We compute, from $(\ref{c52})$ and $(\ref{c2})$,
\begin{center}
\small{\begin{align*}
&[\Lambda_0, x_0\frac{\partial}{\partial z}d\bar{z}+ (x_1+x_2\xi+x_3\xi^2) \frac{\partial}{\partial \xi}d\bar{z}]=[(A+B\xi+C\xi^2)\frac{\partial}{\partial z}\wedge \frac{\partial}{\partial \xi},x_0\frac{\partial}{\partial z}+ (x_1+x_2\xi+x_3\xi^2) \frac{\partial}{\partial \xi}]d\bar{z}\\
  &=( -Bx_1+Ax_2+2(-Cx_1+Ax_3 )\xi   +( -Cx_2+Bx_3  )\xi^2   )\frac{\partial}{\partial z}\wedge \frac{\partial}{\partial \xi} d\bar{z}
\end{align*}}
\end{center}
We represent the above computation in the matrix form with respect to bases $(\ref{c52})$ and $(\ref{c2})$
\begin{equation}\label{c5}
M\bold{x}:=\left(\begin{matrix}
0& -B & A& 0\\
0& -2C & 0& 2A\\
0& 0& -C & B
\end{matrix}\right)
\left(
\begin{matrix}
x_0\\
x_1\\
x_2\\
x_3
\end{matrix}
\right)\\
=\left(
\begin{matrix}
-Bx_1+Ax_2\\
-2Cx_1+2Ax_3\\
-Cx_2+Bx_3
\end{matrix}
\right)
\end{equation}

On the other hand, let us describe $H^0(X,\Theta_X)\xrightarrow{[\Lambda_0,-]} H^0(X,\wedge^2 \Theta_X)$ in $(\ref{c7})$, we compute, from $(\ref{c51})$ and $(\ref{c1})$
\begin{center}
\small{\begin{align*}
&[\Lambda_0, a_0\frac{\partial}{\partial z}+ (a_1+a_2\xi+a_3\xi^2) \frac{\partial}{\partial \xi}]=[(A+B\xi+C\xi^2)\frac{\partial}{\partial z}\wedge \frac{\partial}{\partial \xi},a_0\frac{\partial}{\partial z}+ (a_1+a_2\xi+a_3\xi^2) \frac{\partial}{\partial \xi}] \\
  &=( -Ba_1+Aa_2+2(-Ca_1+Aa_3 )\xi   +( -Ca_2+Ba_3  )\xi^2   )\frac{\partial}{\partial z}\wedge \frac{\partial}{\partial \xi}
\end{align*}}
\end{center}
which is represented by the matrix form with respect to bases $(\ref{c51})$ and $(\ref{c1})$
\begin{equation}\label{c6}
M\bold{a}:=\left(\begin{matrix}
0& -B & A& 0\\
0& -2C & 0& 2A\\
0& 0& -C & B
\end{matrix}\right)
\left(
\begin{matrix}
a_0\\
a_1\\
a_2\\
a_3
\end{matrix}
\right)\\
=\left(
\begin{matrix}
-Ba_1+Aa_2\\
-2Ca_1+2Aa_3\\
-Ca_2+Ba_3
\end{matrix}
\right)
\end{equation}
We note that since $det\left( \begin{matrix}  -B & A& 0\\
-2C & 0& 2A\\
 0& -C & B \end{matrix} \right)=0$, $rank(M)$ cannot be $3$, and we see that if $A=B=C=0$, $rank(M)=0$, and if not, $rank(M)=2$.  
When $\Lambda_0\ne 0$, the kernel of $M$ in $(\ref{c5})$ and $(\ref{c6})$ is given by
\begin{equation*}
\left(\begin{matrix}
x_0\\
x_1\\
x_2\\
x_3
\end{matrix}
\right)
\text{or}
\left(\begin{matrix}
a_0\\
a_1\\
a_2\\
a_3
\end{matrix}
\right)
=t_1\left(\begin{matrix}
1\\
0\\
0\\
0
\end{matrix}\right)
+t_2\left(\begin{matrix}
0\\
A\\
B\\
C
\end{matrix}\right)
\,\,\,\,\,\,\,\,\,t_1,t_2\in \mathbb{C}.
\end{equation*}

We are ready to describe $\mathbb{H}^1(X,\Theta_X^\bullet)$ in terms of $(\ref{c7})$ and determine obstructedeness or unobstructedness of Poisson deformations for the following two cases: (1) $(A,B,C)\ne 0$, and (2) $A=B=C=0$. We note that for the case $(1)$, since $rank(M)=2$, we have $\dim_\mathbb{C} coker(H^0(X,\Theta_X)\xrightarrow{[\Lambda_0,-]} H^0(X,\wedge^2 \Theta_X))=1$, $\dim_\mathbb{C} ker(H^1(X,\Theta_X)\xrightarrow{[\Lambda_0,-]} H^1(X,\wedge^2 \Theta_X))=2 $ so that $\mathbb{H}^1(X,\Theta_X^\bullet)=3$, and $\mathbb{H}^2(X,\Theta_X^\bullet)=1$ by Lemma \ref{c9}. We will show that for the case $(1)$, Poisson deformations of $(E\times \mathbb{P}^1, \Lambda_0= (A+B\xi+C\xi^2)\frac{\partial}{\partial z}\wedge \frac{\partial}{\partial \xi})$ is unobstructed even though we have $\mathbb{H}^2(X,\Theta_X^\bullet)=1\ne 0$. We will show the unobstructedness by finding $\beta(t)\in A^{0,0}(X,\wedge^2\Theta_X)\oplus A^{0,1}(X,\Theta_X)$ satisfying $L\beta(t)+\frac{1}{2}[\beta(t), \beta(t)]=0$, where $L=\bar{\partial}- +[\Lambda_0,-]$, which defines a Poisson analytic family of deformations of $(X,\Lambda_0)$ such that the Poisson Kodaira-Spencer map is an isomorphism at $t=0$ (see Remark $\ref{o9}$). On the other hand, we show that the case $(2)$ $(X,\Lambda_0=0)$ is obstructed in Poisson deformations.

\subsubsection{The case of $(A,B,C)\ne 0$}\

We note that $ker(H^1(X,\Theta_X) \xrightarrow{[\Lambda_0,-]} H^1(X,\wedge^2 \Theta_X))$ is given by
\begin{align}\label{c8}
t_1\frac{\partial}{\partial z}d\bar{z}+t_2(A+B\xi+C\xi^2)\frac{\partial}{\partial \xi}d\bar{z},\,\,\,\,\, t_1,t_2\in \mathbb{C}
\end{align}
Since $coker(H^0(X,\Theta_X)\xrightarrow{[\Lambda_0,-]} H^0(X,\wedge^2 \Theta_X))=1$, choose $(F_0,F_1,F_2)\ne 0\in \mathbb{C}^3$ such that
\begin{align}
(F_0+F_1\xi+F_2\xi^2)\frac{\partial}{\partial z}\wedge \frac{\partial}{\partial \xi}
\end{align}
is not in the image of $H^0(X,\Theta_X)\xrightarrow{[\Lambda_0,-]} H^0(X,\wedge^2 \Theta_X)$ so that $(F_0+F_1\xi+F_2\xi^2)\frac{\partial}{\partial z}\wedge \frac{\partial}{\partial \xi}$ is a representative of the $coker(H^0(X,\Theta_X)\xrightarrow{[\Lambda_0,-]} H^0(X,\wedge^2 \Theta_X))$. Then by Lemma $\ref{c9}$, 
\begin{align*}
\mathbb{H}^1(X,\Theta_X)=\left\langle (F_0+F_1\xi+F_2 \xi^2)\frac{\partial}{\partial z}\wedge \frac{\partial}{\partial \xi} \right\rangle \oplus  \left\langle \frac{\partial}{\partial z}d\bar{z}, (A+B\xi+C\xi^2)\frac{\partial}{\partial \xi}d\bar{z} \right\rangle
\end{align*}
We set $\alpha(t):=t_0(F_0+F_1\xi+F_2\xi^2)\frac{\partial}{\partial z}\wedge \frac{\partial}{\partial \xi}+t_1\frac{\partial}{\partial z}d\bar{z}+t_2(A+B\xi+C\xi^2)\frac{\partial}{\partial \xi}d\bar{z}\in A^{0,0}(X,\wedge^2 \Theta_X)\oplus A^{0,1}(X,\Theta_X)$. Then we have
\begin{align*}
&L \alpha(t)+\frac{1}{2}[\alpha(t),\alpha(t)]=-t_0t_2[\Lambda_0,(F_0+F_1\xi+F_2 \xi^2)\frac{\partial}{\partial \xi} d\bar{z}]
\end{align*}
where $L=\bar{\partial}-+[\Lambda_0,-]$.

Now we take $\beta(t):= \alpha(t) + t_0t_2(F_0+F_1\xi+F_2 \xi^2)\frac{\partial}{\partial \xi} d\bar{z}$. Then we have
\begin{align*}
&L \beta(t)+\frac{1}{2}[\beta(t),\beta(t)]=0
\end{align*}

Hence $\beta(t)$ defines a Poisson analytic family of deformations of $(X,\Lambda_0)$ such that the Poisson Kodaira Spencer map is an isomorphism at $t=0$ so that $(X,\Lambda_0)$ is unobstructed in Poisson deformations.

\subsubsection{The case of $A=B=C=0$}\

We show that $(E\times \mathbb{P}_\mathbb{C}^1,\Lambda_0=0)$ is obstructed in Poisson deformations. Since $\Lambda_0=0$, we have $\mathbb{H}^1(X,\Theta_X^\bullet)=H^0(X,\wedge^2 \Theta_X)\oplus H^1(X,\Theta_X)$. Consider $( \frac{\partial}{\partial z}\wedge \frac{\partial}{\partial \xi},\xi\frac{\partial}{\partial \xi}d\bar{z} )\in H^0(X,\wedge^2 \Theta_X)\oplus H^1(X,\Theta_X)$. Then
\begin{align*}
[ \frac{\partial}{\partial z}\wedge \frac{\partial}{\partial \xi},\xi\frac{\partial}{\partial \xi}d\bar{z}]=\frac{\partial}{\partial z}\wedge \frac{\partial}{\partial \xi}d\bar{z}\ne 0
\end{align*}
so that $(E\times \mathbb{P}_\mathbb{C}^1,\Lambda_0=0)$ is obstructed in Poisson deformations by Lemma $\ref{c9}$. 

We summarize Poisson deformations of product of two projective nonsingular curves in Table \ref{p32}.

\begin{table}
\begin{center}
\begin{tabular}{| c | c | c | c} \hline
Type & Poisson structure $\Lambda_0$ & Poisson deformations \\ \hline
$\mathbb{P}_\mathbb{C}^1\times \mathbb{P}_\mathbb{C}^1$ & any Poisson structure &  unobstructed  \\ \hline
$E_1\times E_2$ where $g(E_1)=g(E_2)=1$ & any Poisson structure  & unobstructed \\ \hline
$E\times \mathbb{P}_\mathbb{C}^1$ where $g(E)=1$ & $\Lambda_0\ne 0$ &  unobstructed  \\ \hline
$E\times \mathbb{P}_\mathbb{C}^1$ where $g(E)=1$ & $\Lambda_0=0$  & obstructed \\ \hline
\end{tabular}
\end{center}
\caption{Poisson deformations of product of two projective nonsingular curves} \label{p32}
\end{table}

\section{$T\times \mathbb{P}_\mathbb{C}^1$, where $T$ is a complex torus with dimension $2$}\label{section6}\
In this section, we study Poisson deformations of $T\times \mathbb{P}_\mathbb{C}^1$ where $T$ is a complex torus with dimension $2$. It is known that $X=T\times \mathbb{P}_\mathbb{C}^1$ is obstructed in complex deformations (see \cite{Kod58} p.436). In this section, we determine obstructedness and unobstructedness in Poisson deformations for any Poisson structure on $X$.  In particular, we show that there exist holomorphic Poisson structures $\Lambda_0$ on $X$ such that $(X,\Lambda_0)$ are unobstructed in Poisson deformations. $T\times \mathbb{P}_\mathbb{C}^1$ provides examples which are
\begin{enumerate}
\item obstructed in complex deformations, but unobstructed in Poisson deformations.
\item obstructed in complex deformations, and obstructed in Poisson deformations.
\end{enumerate}

\subsection{Descriptions of cohomology groups $H^i(X,\wedge^j \Theta_X)$ on $X$}\

Let $T=\mathbb{C}^2/\mathcal{G}$ is a complex torus with dimension $2$, where $\mathcal{G}$ is the free abelian group generated by non-degenerate periods $\omega_j=(\omega_{j1},\omega_{j2}), j=1,2,3,4$ and where, if $(z_1,z_2)$ are the coordinates of $\mathbb{C}^2$, $\omega_j\in \mathcal{G}$ operates on $\mathbb{C}^2$ by sending $z=(z_1,z_2)$ into $z+\omega_j=(z_1+\omega_{j1},z_2+\omega_{j2})$. Let $X:=T\times \mathbb{P}_\mathbb{C}^1$, and $\xi$ is the inhomogenous coordinate of $\mathbb{P}_\mathbb{C}^1$.

We describe cohomology groups $H^i(X,\wedge^j \Theta_X),(i,j)=(0,1),(1,1),(2,1),(0,2),(1,2),(2,2),(0,3),(1,3)$ on $X$. Let $\pi_1:X\to T$ and $\pi_2:X\to \mathbb{P}_\mathbb{C}^1$ be the natural projections. Then we have $\Theta_X=\pi_1^{*}\Theta_T\oplus \pi_2^{*}\Theta_{\mathbb{P}_\mathbb{C}^1}$, $\wedge^2 \Theta_X=\pi_1^*\wedge^2 \Theta_T\oplus \pi_1^*\Theta_T\otimes_{\mathcal{O}_X}\pi_2^*\Theta_{\mathbb{P}_\mathbb{C}^1}$ and $\wedge^3 \Theta_X=\pi_1^*\wedge^2\Theta_T\otimes_{\mathcal{O}_X} \pi_2^* \Theta_{\mathbb{P}_\mathbb{C}^1}$. By K\"{u}nneth formula, we have
\begin{align*}
H^0(X,\Theta_X)&\cong H^0(T,\Theta_T)\oplus H^0(\mathbb{P}_\mathbb{C}^1,\Theta_{\mathbb{P}_\mathbb{C}^1})\\
H^1(X,\Theta_X) &\cong H^1(T,\Theta_T)\oplus \left(H^0(\mathbb{P}_\mathbb{C}^1,\Theta_{\mathbb{P}_\mathbb{C}^1})         \otimes H^1(T,\mathcal{O}_T) \right)\\
H^2(X,\Theta_X) &\cong   H^2(T,\Theta_T)\oplus \left(H^0(\mathbb{P}_\mathbb{C}^1,\Theta_{\mathbb{P}_\mathbb{C}^1}) \otimes H^2(T,\mathcal{O}_T)  \right)    \\
H^0(X,\wedge^2\Theta_X)&\cong H^0(T,\wedge^2 \Theta_T)\oplus  \left(H^0(\mathbb{P}_\mathbb{C}^1,\Theta_{\mathbb{P}_\mathbb{C}^1})\otimes H^0(T,\Theta_T) \right)\\
H^1(X,\wedge^2 \Theta_X)&\cong     H^1(T,\wedge^2 \Theta_T)\oplus  \left(H^0(\mathbb{P}_\mathbb{C}^1,\Theta_{\mathbb{P}_\mathbb{C}^1}) \otimes H^1(T,\Theta_T) \right)        \\
H^2(X,\wedge^2 \Theta_X)&\cong   H^2(T,\wedge^2 \Theta_T)\oplus \left( H^0(\mathbb{P}_\mathbb{C}^1,\Theta_{\mathbb{P}_\mathbb{C}^1})\otimes H^2(T,\Theta_T) \right)   \\
H^0(X,\wedge^3 \Theta_X) & \cong H^0(\mathbb{P}_\mathbb{C}^1,\Theta_{\mathbb{P}_\mathbb{C}^1}) \otimes H^0 (  T,\wedge^2 \Theta_T)      \\
H^1(X,\wedge^3 \Theta_X)&\cong H^0(\mathbb{P}_\mathbb{C}^1,\Theta_{\mathbb{P}_\mathbb{C}^1})\otimes H^1(T,\wedge^2 \Theta_T)
\end{align*}
Hence we have 
\begin{enumerate}
\item $\dim_\mathbb{C} H^0(X,\Theta_X)=5$ and any element of $H^0(X,\Theta_X)$ has an unique representation of the form
\begin{align}\label{p19}
\alpha\frac{\partial}{\partial z_1}+\beta\frac{\partial}{\partial z_2}+(\gamma_1+\gamma_2\xi+\gamma_2\xi^2)\frac{\partial}{\partial \xi},\,\,\,\,\,\text{where}\,\,\,\,\,\alpha,\beta,\gamma_1,\gamma_2,\gamma_3\in \mathbb{C}
\end{align}
\item $\dim_\mathbb{C} H^1(X,\Theta_X)=10$ and any element of $H^1(X,\Theta_X)$ has an unique representation of the form
\begin{align}\label{p24}
\left(r_0\frac{\partial}{\partial z_1}+r_1\frac{\partial}{\partial z_2}\right)d\bar{z}_1+\left(s_0\frac{\partial}{\partial z_1}+s_1\frac{\partial}{\partial z_2}\right)d\bar{z}_2+&(w_0+w_1\xi+w_2\xi^2)\frac{\partial}{\partial \xi}d\bar{z}_1+(v_0+v_1\xi+v_2\xi^2)\frac{\partial}{\partial \xi}d\bar{z}_2,\\
&\text{where}\,\,\,\,\,r_0,r_1,s_0,s_1,w_0,w_1,w_2,v_0,v_1,v_2\in \mathbb{C}. \notag
\end{align}
\item $\dim_\mathbb{C} H^2(X,\Theta_X)=5$ and any element of $H^2(X,\Theta_X)$ has an unique representation of the form
\begin{align}
\left( f_0\frac{\partial}{\partial z_1}+f_1\frac{\partial}{\partial z_2}+(g_0+g_1\xi+g_2\xi^2)\frac{\partial}{\partial \xi}\right)d\bar{z}_1\wedge d\bar{z}_2,\,\,\,\,\,\text{where}\,\,\,\,\,f_0,f_1,g_0,g_1,g_2\in \mathbb{C}.
\end{align}
\item $\dim_\mathbb{C} H^0(X,\wedge^2 \Theta_X)=7$ and any element of $H^0(X, \wedge^2 \Theta_X)$ has an unique representation of the form
\begin{align}
a\frac{\partial}{\partial z_1}\wedge \frac{\partial}{\partial z_2}+(b_0+b_1\xi+b_2\xi^2)\frac{\partial}{\partial z_2}\wedge \frac{\partial}{\partial \xi}+(c_0+c_1\xi+c_2\xi^2)\frac{\partial}{\partial \xi}\wedge \frac{\partial}{\partial z_1},\,\,\,\,\,a,b_0,b_1,b_2,c_0,c_1,c_2\in \mathbb{C}. \label{p1}
\end{align}
\item $\dim_\mathbb{C} H^1(X,\wedge^2 \Theta_X)=14$ and any element of $H^1(X, \wedge^2 \Theta_X)$ has an unique representation of the form
\begin{align}
x_0\frac{\partial}{\partial z_1}\wedge \frac{\partial}{\partial z_2}d\bar{z}_1+x_1\frac{\partial}{\partial z_1}\wedge \frac{\partial}{\partial z_2}d\bar{z}_2+(x_2+x_3\xi+x_4\xi^2)\frac{\partial}{\partial z_1}\wedge \frac{\partial}{\partial \xi}d\bar{z}_1+(x_5+x_6\xi+x_7\xi^2)\frac{\partial}{\partial z_2}\wedge \frac{\partial}{\partial \xi}d\bar{z}_1\\
+ (x_8+x_9\xi+x_{10}\xi^2)\frac{\partial}{\partial z_2}\wedge \frac{\partial}{\partial \xi}d\bar{z}_2    + (x_{11}+x_{12}\xi+x_{12}\xi^2)\frac{\partial}{\partial z_1}\wedge \frac{\partial}{\partial \xi}d\bar{z}_2    \notag \\
\text{where}\,\,\,\,\, x_0,x_1,x_2,x_3,x_4,x_5,x_6,x_7,x_8,x_9,x_{10},x_{11},x_{12},x_{13}\in \mathbb{C}   \notag
\end{align}
\item $\dim_\mathbb{C} H^2(X,\wedge^2 \Theta_X)=7$ and any element of $H^1(X,\wedge^2 \Theta_X)$ has an unique representation of the form
\begin{align}
\left(h_0\frac{\partial}{\partial z_1}\wedge \frac{\partial}{\partial z_2}+(p_0+p_1\xi+p_2\xi^2)\frac{\partial}{\partial z_1}\wedge \frac{\partial}{\partial \xi}+(q_0+q_1\xi+q_2\xi^2)\frac{\partial}{\partial z_2}\wedge \frac{\partial}{\partial \xi}\right)d\bar{z}_1\wedge d\bar{z}_2\\
\text{where}\,\,\,\,\,h_0,p_0,p_1,p_2,q_0,q_1,q_2\in \mathbb{C}. \notag
\end{align}
\item $\dim_\mathbb{C} H^0(X, \wedge^3 \Theta_X)=3$ and any element of $H^0(X,  \wedge^3 \Theta_X)$ has an unique representation of the form
\begin{align}\label{p50}
(e_0+e_1\xi+e_2\xi^2)\frac{\partial}{\partial z_1}\wedge \frac{\partial}{\partial z_2}\wedge \frac{\partial}{\partial \xi},\,\,\,\,\,\text{where}\,\,\,\,\,e_0,e_1,e_2\in \mathbb{C}.
\end{align}
\item $\dim_\mathbb{C}H^1(X,\wedge^3 \Theta_X)=6$ and any element of $H^1(X,\wedge^3 \Theta_X)$ has an unique representation of the form
\begin{align}
(y_0+y_1\xi+y_2\xi^2)\frac{\partial}{\partial z_1}\wedge \frac{\partial}{\partial z_2}\wedge \frac{\partial}{\partial \xi}d\bar{z}_1+(y_3+y_4\xi+y_5\xi^2)\frac{\partial}{\partial z_1}\wedge \frac{\partial}{\partial z_2}\wedge \frac{\partial}{\partial \xi}d\bar{z}_2\\
\text{where}\,\,\,\,\,y_0,y_1,y_2,y_3,y_4,y_5\in \mathbb{C}. \notag
\end{align}
\end{enumerate}

\subsection{Holomorphic Poisson structures on $X=T\times \mathbb{P}_\mathbb{C}^1$}\

We describe holomorphic Poisson structures on $X$. As in $(\ref{p1})$, an element $\Lambda_0\in H^0(X,\wedge^2\Theta_X)$ is of the form
\begin{align*}
\Lambda_0=a\frac{\partial}{\partial z_1}\wedge \frac{\partial}{\partial z_2}+(b_0+b_1\xi+b_2\xi^2)\frac{\partial}{\partial z_2}\wedge \frac{\partial}{\partial \xi}+(c_0+c_1\xi+c_2\xi^2)\frac{\partial}{\partial \xi}\wedge \frac{\partial}{\partial z_1}
\end{align*}
where $a,b_0,b_1,b_2,c_0,c_1,c_2\in\mathbb{C}$. We set
\begin{align*}
 P_1&:=a\\
 P_2&:=b_0+b_1\xi+b_2\xi^2\\
 P_3&:=c_0+c_1\xi+c_2\xi^2.
\end{align*}
 Then $\Lambda_0$ defines a holomorphic Poisson structure on $X$ if and only if $[\Lambda_0,\Lambda_0]=0$ if and only if
\begin{align*}
&P_1\left( \frac{\partial P_3}{\partial z_1}-\frac{\partial P_2}{\partial z_2} \right)+P_2\left( \frac{\partial P_1}{\partial z_2}-\frac{\partial P_3}{\partial \xi}  \right)+P_3\left( \frac{\partial P_2}{\partial \xi}-\frac{\partial P_1}{\partial z_1}  \right)=0\\
&\iff  b_1c_0-b_0c_1=0,\,\,\,\,\, b_2c_0-b_0c_2=0,\,\,\,\,\,b_2c_1-b_1c_2=0
\end{align*}

Hence we can divide holomorphic Poisson structures on $X=T\times \mathbb{P}_\mathbb{C}^1$ into three classes
\begin{enumerate}
\item $\Lambda_0=D\frac{\partial}{\partial z_1}\wedge \frac{\partial}{\partial z_2},D\in \mathbb{C}$ 
\item $\Lambda_0=D\frac{\partial}{\partial z_1}\wedge \frac{\partial}{\partial z_2}+(A+B\xi+C\xi^2)\frac{\partial}{\partial z_2}\wedge \frac{\partial}{\partial \xi}+ k(A+B\xi+C\xi^2)\frac{\partial}{\partial \xi}\wedge \frac{\partial}{\partial z_1}$, $D\in \mathbb{C}, (A,B,C)\ne 0\in \mathbb{C}^3$, $k\in \mathbb{C}$
\item $\Lambda_0=D\frac{\partial}{\partial z_1}\wedge \frac{\partial}{\partial z_2}+(A+B\xi+C\xi^2)\frac{\partial}{\partial \xi}\wedge \frac{\partial}{\partial z_1}$, $D\in \mathbb{C}, (A,B,C)\ne 0\in \mathbb{C}^3$.
\end{enumerate}

In the next subsections, we will determine obstructedness or unobstructedness for each class, and show that $(X,\Lambda_0)$ is obstructed for the case $(1)$, and $(X,\Lambda_0)$ is unobstructed for the case $(2)$ and $(3)$.

Next we note that by considering spectral sequence associated with the double complex $(\ref{o7})$ induced from $\Theta_X^\bullet$, we get the following lemma.
\begin{lemma}
Let us consider $(X=T\times \mathbb{P}_\mathbb{C}^1,\Lambda_0)$, where $T$ is a complex torus with dimension $2$. Then we have
\begin{align}
\mathbb{H}^0(X,\Theta_X^\bullet)\cong & ker(H^0(X,\Theta_X)\xrightarrow{[\Lambda_0,-]}H^0(X,\wedge^2 \Theta_X))  \\
\mathbb{H}^1(X,\Theta_X^\bullet)\cong & ker(H^0(X,\wedge^2 \Theta_X)\xrightarrow{[\Lambda_0,-]} H^0(X,\wedge^3 \Theta_X))/im(H^0(X,\Theta_X)\xrightarrow{[\Lambda_0,-]}H^0(X,\wedge^2 \Theta_X)) \label{p2}\\
 &\oplus ker(H^1(X,\Theta_X)\xrightarrow{[\Lambda_0,-]}H^1(X,\wedge^2 \Theta_X))\notag\\
 \mathbb{H}^2(X,\Theta_X^\bullet)\cong & coker(H^0(X,\wedge^2 \Theta_X)\xrightarrow{[\Lambda_0,-]}H^0(X,\wedge^3 \Theta_X))\\
  &\oplus ker(H^1(X,\wedge^2 \Theta_X)\xrightarrow{[\Lambda_0,-]} H^1(X,\wedge^3 \Theta_X))/im(H^1(X,\Theta_X)\xrightarrow{[\Lambda_0,-]}H^1(X,\wedge^2 \Theta_X)) \notag\\
   &\oplus ker(H^2(X,\Theta_X)\xrightarrow{[\Lambda_0,-]}H^2(X,\wedge^2 \Theta_X)) \notag
\end{align}
\end{lemma}

\subsection{Obstructed Poisson deformations}\

Let us consider a holomorphic Poisson structure on $X=T\times \mathbb{P}_\mathbb{C}^1$ given by
\begin{align*}
\Lambda_0=D\frac{\partial}{\partial z_1}\wedge \frac{\partial}{\partial z_2},\,\,\,\,\,D\in \mathbb{C}.
\end{align*}
In this case, we have $\mathbb{H}^1(X,\Theta_X^\bullet)=H^0(X,\wedge^2 \Theta_X)\oplus H^1(X,\Theta_X)$. Hence since $X$ is obstructed in complex deformations,  $(X,\Lambda_0)$ is obstructed in Poisson deformations.

\subsection{Unobstructed Poisson deformations}\

Let us consider a holomorphic Poisson structure on $X=T\times \mathbb{P}_\mathbb{C}^1$ given by
\begin{align*}
\Lambda_0=D\frac{\partial}{\partial z_1}\wedge \frac{\partial}{\partial z_2}+(A+B\xi+C\xi^2)\frac{\partial}{\partial z_2}\wedge \frac{\partial}{\partial \xi}+ k(A+B\xi+C\xi^2)\frac{\partial}{\partial \xi}\wedge \frac{\partial}{\partial z_1}
\end{align*}
where $D\in \mathbb{C}$ and $(A,B,C)\ne 0\in \mathbb{C}^3$, and $k \in \mathbb{C}$. We will show that $(X,\Lambda_0)$ is unobstructed in Poisson deformations.

We will describe $\mathbb{H}^1(X,\Theta_X^\bullet)$ explicitly in terms of $(\ref{p2})$. 

First let us find $im(H^0(X,\Theta_X)\xrightarrow{[\Lambda_0,-]} H^0(X,\wedge^2 \Theta_X) )$. We compute, from $(\ref{p19})$ and $(\ref{p1})$

{\small{\begin{align*}
&[D\frac{\partial}{\partial z_1}\wedge \frac{\partial}{\partial z_2}+(A+B\xi+C\xi^2)\frac{\partial}{\partial z_2}\wedge \frac{\partial}{\partial \xi}+k(A+B\xi+C\xi^2)\frac{\partial}{\partial \xi}\wedge \frac{\partial}{\partial z_1}, \alpha\frac{\partial}{\partial z_1}+\beta\frac{\partial}{\partial z_2}+(\gamma_0+\gamma_1\xi+\gamma_2\xi^2)\frac{\partial}{\partial \xi}]\\
&=(A\gamma_1-B\gamma_0+2(A\gamma_2-C\gamma_0)\xi+(B\gamma_2-C\gamma_1)\xi^2)\frac{\partial}{\partial z_2}\wedge \frac{\partial}{\partial \xi}+k(A\gamma_1-B\gamma_0+2(A\gamma_2-C\gamma_0)\xi+(B\gamma_2-C\gamma_1)\xi^2)\frac{\partial}{\partial \xi}\wedge \frac{\partial}{\partial z_1}
\end{align*}}}
Hence $\dim_\mathbb{C} \mathbb{H}^0(X,\Theta_X^\bullet)=\dim_\mathbb{C} ker(H^0(X,\Theta_X)\xrightarrow{[\Lambda_0,-]} H^0(X,\wedge^2 \Theta_X))=3$ so that $\dim_\mathbb{C} im(H^0(X,\Theta_X)\xrightarrow{[\Lambda_0,-]} H^0(X,\wedge^2 \Theta_X))=2$.

Second let us find $ker(H^0(X,\wedge^2 \Theta_X)\xrightarrow{[\Lambda_0,-]}H^0(X,\wedge^3 \Theta_X))$. We compute, from $(\ref{p1})$ and $(\ref{p50})$,
{\tiny{\begin{align*}
&[D\frac{\partial}{\partial z_1}\wedge \frac{\partial}{\partial z_2}+(A+B\xi+C\xi^2)\frac{\partial}{\partial z_2}\wedge \frac{\partial}{\partial \xi}+k(A+B\xi+C\xi^2)\frac{\partial}{\partial \xi}\wedge \frac{\partial}{\partial z_1},a\frac{\partial}{\partial z_1}\wedge \frac{\partial}{\partial z_2}+ (b_0+b_1\xi+b_2\xi^2)   \frac{\partial}{\partial z_2}\wedge \frac{\partial}{\partial \xi}+(c_0+c_1\xi+c_2\xi^2) \frac{\partial}{\partial \xi }\wedge \frac{\partial}{\partial z_1}]\\
&=-(Ac_1-Bc_0+2(Ac_2-Cc_0)\xi+(Bc_2-Cc_1)\xi^2)\frac{\partial}{\partial \xi}\wedge \frac{\partial}{\partial z_2}\wedge \frac{\partial}{\partial z_1}+k(Ab_1-Bb_0+2(Ab_2-Cb_0)\xi+(Bb_2-Cb_1)\xi^2)\frac{\partial}{\partial \xi}\wedge \frac{\partial}{\partial z_2}\wedge \frac{\partial}{\partial z_1}=0
\end{align*}}}
so that we have $A(- c_1+k b_1)-B(- c_0+k b_0)=0,\,\,\,\,\,A(-c_2+k b_2)-C(- c_0+k b_0)=0,\,\,\,\,\,B(-c_2+k b_2)-C(- c_1+k b_1)=0$. Hence
\begin{align*}
&(- c_0+k b_0, - c_1+k b_1,-c_2+kb_2 )=-t_1(A,B,C),\,\,\,\,\,\,\, t_1\in \mathbb{C}\\
&\iff (c_0,c_1,c_2)=k(b_0,b_1,b_2)+t_1(A,B,C),\,\,\,\,\,\,\,\,\,t_1\in \mathbb{C}
\end{align*}
Hence $\dim_\mathbb{C} ker(H^0(X,\wedge^2 \Theta_X)\xrightarrow{[\Lambda_0,-]}H^0(X,\wedge^3 \Theta_X))=5$ and any element of $ker(H^0(X,\wedge^2 \Theta_X)\xrightarrow{[\Lambda_0,-]}H^0(X,\wedge^3 \Theta_X))$ is of the form
\begin{align*}
t_0\frac{\partial}{\partial z_1}\wedge \frac{\partial}{\partial z_2}+(b_0+b_1\xi+b_2\xi^2)\frac{\partial}{\partial z_2}\wedge \frac{\partial}{\partial \xi} +(t_1A+kb_0+(t_1B+kb_1)\xi+(t_1C+kb_2)\xi^2)\frac{\partial}{\partial \xi}\wedge \frac{\partial}{\partial z_1}\\
\,\,\,\,\,\text{where}\,\,\,\,\,t_0,b_0,b_1,b_2,t_1\in \mathbb{C}.
\end{align*}

Since $rank\left(\begin{matrix} -B & A &0\\ -2C & 0 & 2A \\ 0 & -C & B   \end{matrix} \right)=2$, there exist $(F_1,F_2,F_3)\ne 0\in \mathbb{C}^3 $ such that
\begin{align*}
(F_0+F_1\xi+F_2\xi^2)\frac{\partial}{\partial z_2}\wedge \frac{\partial}{\partial \xi}+k(F_0+F_1\xi+F_2\xi^2)\frac{\partial}{\partial \xi}\wedge \frac{\partial}{\partial z_1} \in ker(H^0(X,\wedge^2 \Theta_X)\xrightarrow{[\Lambda_0,-]} H^0(X, \wedge^3 \Theta_X))
\end{align*}
is not in $im(H^0(X,\Theta_X)\xrightarrow{[\Lambda_0,-]} H^0(X,\wedge^2 \Theta_X))$. Then 
\begin{align*}
\dim_\mathbb{C} \mathbb{H}^1(X,\Theta_X^\bullet)=\dim_\mathbb{C} ker(H^0(X,\wedge^2 \Theta_X)\xrightarrow{[\Lambda_0,-]} H^0(X,\wedge^3 \Theta_X))/im(H^0(X,\Theta_X)\xrightarrow{[\Lambda_0,-]}H^0(X,\wedge^2 \Theta_X))=3
\end{align*}
and any element has an unique representation of the form
\begin{align}\label{p17}
t_0\frac{\partial}{\partial z_1}\wedge \frac{\partial}{\partial z_2}+t_2(F_0+F_1\xi+F_2\xi^2)\frac{\partial}{\partial z_2}\wedge \frac{\partial}{\partial \xi} +(t_1A+kt_2F_0+(t_1B+kt_2F_1)\xi+(t_1C+kt_2F_2)\xi^2)\frac{\partial}{\partial \xi}\wedge \frac{\partial}{\partial z_1}\\
\,\,\,\,\,\text{where}\,\,\,\,\,t_0,t_1,t_2 \in \mathbb{C}. \notag
\end{align}

Next let us find $ker(H^1(X,\Theta_X)\xrightarrow{[\Lambda_0,-]} H^1(X,\wedge^2 \Theta_X ) )$. We compute, from $(\ref{p24})$,
{\tiny{\begin{align*}
&[ D\frac{\partial}{\partial z_1}\wedge \frac{\partial}{\partial z_2}+(A+B\xi+C\xi^2)\frac{\partial}{\partial z_2}\wedge \frac{\partial}{\partial \xi}+k(A+B\xi+C\xi^2)\frac{\partial}{\partial \xi}\wedge \frac{\partial}{\partial z_1} ,\\
&\left(r_0\frac{\partial}{\partial z_1}+r_2\frac{\partial}{\partial z_2}\right)d\bar{z}_1+\left(s_0\frac{\partial}{\partial z_1}+s_1\frac{\partial}{\partial z_2}\right)d\bar{z}_2+(w_0+w_1\xi+w_2\xi^2)\frac{\partial}{\partial \xi}d\bar{z}_1+(v_0+v_1\xi+v_2\xi^2)\frac{\partial}{\partial \xi}d\bar{z}_2]\\
&=(Aw_1-Bw_0+2(Aw_2-Cw_0)\xi+(Bw_2-Cw_1)\xi^2)\frac{\partial}{\partial z_2}\wedge \frac{\partial}{\partial \xi}d\bar{z}_1+(Av_1-Bv_0+2(Av_2-Cv_0)\xi+(Bv_2-Cv_1)\xi^2)\frac{\partial}{\partial z_2}\wedge \frac{\partial}{\partial \xi}d\bar{z}_2\\
&-k(Aw_1-Bw_0+2(Aw_2-Cw_0)\xi+(Bw_2-Cw_1)\xi^2)\frac{\partial}{\partial z_1}\wedge \frac{\partial}{\partial \xi}d\bar{z}_1-k(Av_1-Bv_0+2(Av_2-Cv_0)\xi+(Bv_2-Cv_1)\xi^2)\frac{\partial}{\partial z_1}\wedge \frac{\partial}{\partial \xi}d\bar{z}_2=0
\end{align*}}}
Hence $\dim_\mathbb{C} ker(H^1(X,\Theta_X)\xrightarrow{[\Lambda_0,-]} H^1(X,\wedge^2 \Theta_X ) )=6$ and any element of $ker(H^1(X,\Theta_X)\xrightarrow{[\Lambda_0,-]} H^1(X,\wedge^2 \Theta_X ) )$ is of the form
\begin{align}\label{p18}
\left(t_3\frac{\partial}{\partial z_1}+t_4\frac{\partial}{\partial z_2}\right)d\bar{z}_1+\left(t_5\frac{\partial}{\partial z_1}+t_6\frac{\partial}{\partial z_2}\right)d\bar{z}_2+t_7(A+B\xi+C\xi^2)\frac{\partial}{\partial \xi}d\bar{z}_1+t_8(A+B\xi+C\xi^2)\frac{\partial}{\partial \xi}d\bar{z}_2
\end{align}

Hence $\dim_\mathbb{C} \mathbb{H}^1(X,\Theta_X^\bullet)=9$ and any element of $\mathbb{H}^1(X,\Theta_X^\bullet)$ has an unique representation of the form given by $(\ref{p17})$ and $(\ref{p18})$.

\begin{remark}
We note that $\dim_\mathbb{C} \mathbb{H}^2(X,\Theta_X^\bullet  )  \ne 0$ since $ker(H^2(X,\Theta_X)\xrightarrow{[\Lambda_0,-]}H^2(X, \wedge^2 \Theta_X))\ne0$.
Nevertheless we show that $(X,\Lambda_0)$ is unobstructed in Poisson deformations by constructing a Poisson analytic family of deformations of $(X,\Lambda_0)$  such that the Poisson Kodaira-Spencer  map is an isomorphism at $t=0$. 
\end{remark}

We set
\begin{align*}
\Lambda=\Lambda(t):&=t_0\frac{\partial}{\partial z_1}\wedge \frac{\partial}{\partial z_2}+t_2(F_0+F_1\xi+F_2\xi^2)\frac{\partial}{\partial z_2}\wedge \frac{\partial}{\partial \xi} +(t_1A+kt_2F_0+ (t_1B+kt_2F_1)\xi+(t_1C+kt_2F_2)\xi^2)\frac{\partial}{\partial \xi}\wedge \frac{\partial}{\partial z_1}\\
\phi=\phi(t):&=\left(t_3\frac{\partial}{\partial z_1}+t_4\frac{\partial}{\partial z_2}\right)d\bar{z}_1+\left(t_5\frac{\partial}{\partial z_1}+t_6\frac{\partial}{\partial z_2}\right)d\bar{z}_2+t_7(A+B\xi+C\xi^2)\frac{\partial}{\partial \xi}d\bar{z}_1+t_8(A+B\xi+C\xi^2)\frac{\partial}{\partial \xi}d\bar{z}_2
\end{align*}

We will show that there exist $\Lambda'(t)$ with coefficients in $H^0(X,\wedge^2 \Theta_X)$ and degree $\geq 2$ in $t$, and $\phi'(t)$ with holomorphic coefficients in $A^{0,1}(X,\Theta_X)$ and degree $\geq 2$ in $t$ such that by setting
\begin{align*}
\beta(t):&=\Lambda(t)+\Lambda'(t)\\
\alpha(t):&=\phi(t)+\phi'(t),
\end{align*}
$\beta(t)+\alpha(t)$ satisfies the integrability condition
\begin{align*}
L(\beta(t)+\alpha(t))+\frac{1}{2}[\beta(t)+\alpha(t),\beta(t)+\alpha(t)]=0,\,\,\,\,\,L=\bar{\partial}-+[\Lambda_0,-],
\end{align*}
equivalently
\begin{align}
[\Lambda_0,\beta(t)]+\frac{1}{2}[\beta(t),\beta(t)]&=0 \label{p11}\\
\bar{\partial} \beta(t)+[\Lambda_0,\alpha(t)]+[\beta(t),\alpha(t)]&=0 \label{p12}\\
\bar{\partial} \alpha(t)+\frac{1}{2}[\alpha(t),\alpha(t)]&=0 \label{p13}
\end{align}
so that $\beta(t)+\alpha(t)$ defines a Poisson analytic family of deformations of $(X,\Lambda_0)$ such that the Poisson Kodaira-Spencer map is an isomorphism at $t=0$ by Remark $\ref{o9}$. Hence $(X,\Lambda_0)$ is unobstructed in Poisson deformations.

First we note that
{\small{\begin{align*}
&[\Lambda(t),\Lambda(t)]\\
&=2[t_2(F_0+F_1\xi+F_2\xi^2)\frac{\partial}{\partial z_2}\wedge \frac{\partial}{\partial \xi}, (t_1A+kt_2F_0+ (t_1B+kt_2F_1)\xi+(t_1C+kt_2F_2)\xi^2)\frac{\partial}{\partial \xi}\wedge \frac{\partial}{\partial z_1}]\\
&=2t_1t_2(F_0B-F_1A+2(F_0C-F_2A)\xi+(F_1C-F_2B)\xi^2)\frac{\partial } {\partial \xi}\wedge \frac{\partial}{\partial z_1}\wedge \frac{\partial}{\partial z_2}\\
&=[\Lambda_0, -2t_1t_2(F_0+F_1\xi+F_2\xi^2)\frac{\partial}{\partial \xi} \wedge \frac{\partial}{\partial z_1}]
\end{align*}}}
We set 
\begin{align*}
\Lambda'(t):=t_1t_2(F_0+F_1\xi+F_2\xi^2)\frac{\partial}{\partial \xi}\wedge \frac{\partial}{\partial z_1}
\end{align*}
so that $[\Lambda_0, \Lambda'(t)]+\frac{1}{2}[\Lambda(t),\Lambda(t)]=0$. We note that $[\Lambda'(t),\Lambda'(t)]=0$, and $[\Lambda(t),\Lambda'(t)]=0$.
Then we have
\begin{align}\label{p14}
[\Lambda_0, \Lambda(t)+\Lambda'(t)]+\frac{1}{2}[\Lambda(t)+\Lambda'(t),\Lambda(t)+\Lambda'(t)]=0
\end{align}
Hence $(\ref{p11})$ is satisfied.
Second we note that 
{\tiny{\begin{align*}
&[\Lambda(t),\phi(t)]\\
&=t_2t_7[(F_0+F_1\xi+F_2\xi^2)\frac{\partial}{\partial z_2}\wedge \frac{\partial}{\partial \xi}, (A+B\xi+C\xi^2)\frac{\partial}{\partial \xi}]d\bar{z}_1
+t_2t_8[(F_0+F_1\xi+F_2\xi^2)\frac{\partial}{\partial z_2}\wedge \frac{\partial}{\partial \xi}, (A+B\xi+C\xi^2)\frac{\partial}{\partial \xi}]d\bar{z}_2\\
&+ kt_2t_7[(F_0+F_1\xi+F_2\xi^2)\frac{\partial}{\partial \xi}\wedge \frac{\partial}{\partial z_1}, (A+B\xi+C\xi^2)\frac{\partial}{\partial \xi}]d\bar{z}_1 + kt_2t_8[(F_0+F_1\xi+F_2\xi^2)\frac{\partial}{\partial \xi}\wedge \frac{\partial}{\partial z_1}, (A+B\xi+C\xi^2)\frac{\partial}{\partial \xi}]d\bar{z}_2 \\
&=t_2t_7(F_0B-F_1A+2(F_0C-F_2 A)\xi+(F_1C-F_2B)\xi^2)\frac{\partial}{\partial z_2}\wedge \frac{\partial}{\partial \xi}d\bar{z}_1+t_2t_8(F_0B-F_1A+2(F_0C-F_2 A)\xi+(F_1C-F_2B)\xi^2)\frac{\partial}{\partial z_2}\wedge \frac{\partial}{\partial \xi}d\bar{z}_2\\
&-kt_2t_7(F_0B-F_1A+2(F_0C-F_2 A)\xi+(F_1C-F_2B)\xi^2)\frac{\partial}{\partial z_1}\wedge \frac{\partial}{\partial \xi}d\bar{z}_1-kt_2t_8(F_0B-F_1A+2(F_0C-F_2 A)\xi+(F_1C-F_2B)\xi^2)\frac{\partial}{\partial z_1}\wedge \frac{\partial}{\partial \xi}d\bar{z}_2\\
&=[\Lambda_0,-t_2t_7(F_0+F_1\xi+F_2\xi^2)\frac{\partial}{\partial \xi} d\bar{z}_1]+[\Lambda_0, -t_2t_8(F_0+F_1\xi+F_2\xi^2)\frac{\partial}{\partial \xi} d\bar{z}_2]
\end{align*}}}
We set
\begin{align*}
\phi'(t)=t_2t_7 (F_0+F_1\xi+F_2\xi^2)\frac{\partial}{\partial \xi} d\bar{z}_1+t_2t_8(F_0+F_1\xi+F_2\xi^2)\frac{\partial}{\partial \xi}d\bar{z}_2
\end{align*}
so that $[\Lambda_0,\phi'(t)]+[\Lambda(t),\phi(t)]=0$. We note that $[\Lambda'(t),\phi'(t)]=0$ and
{\small{\begin{align*}
&[\Lambda(t),\phi'(t)]+[\Lambda'(t),\phi(t)]\\
&= [ t_1(A+B\xi+C\xi^2)\frac{\partial}{\partial \xi} \wedge \frac{\partial}{\partial z_1}, t_2t_7 (F_0+F_1\xi+F_2\xi^2)\frac{\partial}{\partial \xi} d\bar{z}_1+t_2t_8(F_0+F_1\xi+F_2\xi^2)\frac{\partial}{\partial \xi}d\bar{z}_2]\\
&+[t_1t_2(F_0+F_1\xi+F_2\xi^2)\frac{\partial}{\partial \xi}\wedge \frac{\partial}{\partial z_1}, t_7(A+B\xi+C\xi^2)\frac{\partial}{\partial \xi}d\bar{z}_1+t_8(A+B\xi+C\xi^2)\frac{\partial}{\partial \xi}d\bar{z}_2]=0
\end{align*}}}
Then since $\Lambda(t)+\Lambda'(t)$ is holomorphic, we have
\begin{align}\label{p15}
\bar{\partial}(\Lambda(t)+\Lambda'(t))+[\Lambda_0, \phi(t)+\phi'(t)]+[\Lambda(t)+\Lambda'(t),\phi(t)+\phi'(t)]=0
\end{align}
Hence $(\ref{p12})$ is satisfied. Lastly we note that $\phi(t)+\phi'(t)$ is holomorphic and
\begin{align*}
\phi(t)+\phi'(t)=&\left(t_3\frac{\partial}{\partial z_1}+t_4\frac{\partial}{\partial z_2}\right)d\bar{z}_1+\left(t_5\frac{\partial}{\partial z_1}+t_6\frac{\partial}{\partial z_2}\right)d\bar{z}_2\\
&+t_7(A+t_2F_0+(B+t_2F_1)\xi+(C+t_2 F_2)\xi^2)\frac{\partial}{\partial \xi}d\bar{z}_1+t_8(A+t_2F_0+(B+t_2F_1)\xi+(C+t_2 F_2)\xi^2)\frac{\partial}{\partial \xi} d\bar{z}_2
\end{align*}
so that we have
\begin{align}\label{p16}
\bar{\partial}(\phi(t)+\phi'(t))+\frac{1}{2}[\phi(t)+\phi'(t),\phi(t)+\phi'(t)]=0
\end{align}
Hence $(\ref{p13})$ is satisfied. From $(\ref{p14}),(\ref{p15}),(\ref{p16})$, $\beta(t):=\Lambda(t)+\Lambda'(t)$ and $\alpha(t):=\phi(t)+\phi'(t)$ defines a Poisson analytic family of deformations of $(X,\Lambda_0)$ such that the Poisson Kodaira-Spencer map is an isomorphism at $t=0$ so that $(X,\Lambda_0)$ is unobstructed in Poisson deformations.

It remains to determine obstructedness or unobstructedness of $(X,\Lambda_0)$ where  $\Lambda_0$ is given by
\begin{align*}
\Lambda_0=D\frac{\partial}{\partial z_1}\wedge \frac{\partial}{\partial z_2}+(A+B\xi+C\xi^2)\frac{\partial}{\partial \xi}\wedge \frac{\partial}{\partial z_1}
\end{align*}
where $D\in \mathbb{C}$ and $(A,B,C)\ne 0\in \mathbb{C}^3$. But by changing  $z_1$ and $z_2$, we get to the previous case. Hence $(X,\Lambda_0)$ is unobstructed in Poisson deformations.

We summarize Poisson deformations of $T\times \mathbb{P}_\mathbb{C}^1$ in Table \ref{p30}.

\begin{table}
\begin{center}
\begin{tabular}{| c | c | c | c} \hline
Poisson structure on $X=T\times \mathbb{P}_\mathbb{C}^1$& $\dim_\mathbb{C} \mathbb{H}^1(X,\Theta_X^\bullet)$ & Poisson deformations \\ \hline
$D\frac{\partial}{\partial z_1}\wedge \frac{\partial}{\partial z_2},D\in \mathbb{C}$ & 17 & obstructed\\ \hline
$D\frac{\partial}{\partial z_1}\wedge \frac{\partial}{\partial z_2}+(A+B\xi+C\xi^2)\frac{\partial}{\partial z_2}\wedge \frac{\partial}{\partial \xi}+ k(A+B\xi+C\xi^2)\frac{\partial}{\partial \xi}\wedge \frac{\partial}{\partial z_1}$ & 9 &  unobstructed  \\ \cline{4-4}
$D\in \mathbb{C}, (A,B,C)\ne 0\in \mathbb{C}^3$, $k\in \mathbb{C}$ &  &  \\ \hline
$D\frac{\partial}{\partial z_1}\wedge \frac{\partial}{\partial z_2}+(A+B\xi+C\xi^2)\frac{\partial}{\partial \xi}\wedge \frac{\partial}{\partial z_1}$ & 9 &  unobstructed  \\ \cline{4-4}
$D\in \mathbb{C}, (A,B,C)\ne 0\in \mathbb{C}^3$ &  &  \\ \hline
\end{tabular}
\end{center}
\caption{Poisson deformations of $T\times \mathbb{P}_\mathbb{C}^1$} \label{p30}
\end{table}

\bibliographystyle{amsalpha}
\bibliography{References-Rev9}

\providecommand{\bysame}{\leavevmode\hbox to3em{\hrulefill}\thinspace}
\providecommand{\MR}{\relax\ifhmode\unskip\space\fi MR }
\providecommand{\MRhref}[2]{%
  \href{http://www.ams.org/mathscinet-getitem?mr=#1}{#2}
}
\providecommand{\href}[2]{#2}
\begin{thebibliography}{LGPV13}

\bibitem[HX11]{Pin11}
Wei Hong and Ping Xu, \emph{Poisson cohomology of del {P}ezzo surfaces}, J.
  Algebra \textbf{336} (2011), 378--390. \MR{2802550 (2012i:14040)}

\bibitem[Kim14]{Kim15}
Chunghoon Kim, \emph{Theorem of existence and completeness for holomorphic
  {P}oisson structures}, preprint (2014).

\bibitem[Kod05]{Kod05}
Kunihiko Kodaira, \emph{Complex manifolds and deformation of complex
  structures}, english ed., Classics in Mathematics, Springer-Verlag, Berlin,
  2005, Translated from the 1981 Japanese original by Kazuo Akao. \MR{2109686
  (2005h:32030)}

\bibitem[KS58]{Kod58}
K.~Kodaira and D.~C. Spencer, \emph{On deformations of complex analytic
  structures. {I}, {II}}, Ann. of Math. (2) \textbf{67} (1958), 328--466.
  \MR{0112154 (22 \#3009)}

\bibitem[LGPV13]{Lau13}
Camille Laurent-Gengoux, Anne Pichereau, and Pol Vanhaecke, \emph{Poisson
  structures}, Grundlehren der Mathematischen Wissenschaften [Fundamental
  Principles of Mathematical Sciences], vol. 347, Springer, Heidelberg, 2013.
  \MR{2906391}

\bibitem[Weh81]{Weh81}
Joachim Wehler, \emph{Versal deformation of {H}opf surfaces}, J. Reine Angew.
  Math. \textbf{328} (1981), 22--32. \MR{636192}

\end{thebibliography}

\end{document}